\documentclass[10 pt,reqno]{amsart}
\usepackage{color}
\usepackage{enumerate}
\usepackage{amssymb,amsmath,amsthm,latexsym,mathrsfs}
\usepackage{amsbsy,amsfonts,mathtools,slashed}
\usepackage{graphicx,color,yfonts}
\usepackage[numbers,sort]{natbib}
\usepackage[latin9]{inputenc}
\usepackage{xcolor}

\usepackage[colorlinks=true]{hyperref}
\hypersetup{linkcolor=red,citecolor=blue,filecolor=dullmagenta,urlcolor=blue} 

\makeatletter

\renewcommand\subsection{\@startsection{subsection}{2}{\z@}%
  {-3.25ex\@plus -1ex \@minus -.2ex}%
  {1.5ex \@plus .2ex}%
  {\normalfont\normalsize\bfseries}}

\renewcommand\subsubsection{\@startsection{subsubsection}{3}{\z@}%
  {-3.25ex\@plus -1ex \@minus -.2ex}%
  {-1em}%
  {\normalfont\normalsize\bfseries}}

\makeatother

\newtheorem{thm}{Theorem}[section]
\newtheorem{lemma}[thm]{Lemma}
\newtheorem{prop}[thm]{Proposition}

\newtheorem{rem}[thm]{Remark}

\makeatother
\numberwithin{equation}{section}
\numberwithin{figure}{section}
\theoremstyle{plain}
\theoremstyle{plain}
\theoremstyle{plain}

\makeatother
\numberwithin{equation}{section}
\numberwithin{figure}{section}
\theoremstyle{plain}
\theoremstyle{plain}
\theoremstyle{plain}

\newcommand{\les}{\lesssim}
\newcommand{\lam}{{\lambda}}

\newcommand{\ve}{{\varepsilon}}
\newcommand{\de}{{\delta}}

\newcommand{\al}{{\alpha}}

\newcommand{\R}{{\mathbb R}}
\newcommand{\Z}{{\mathbb Z}}

\newcommand{\C}{{\mathbb C}}
\newcommand{\N}{{\mathbb N}}

\newcommand{\p}{\partial}

\newcommand{\brad}{{\bra{D}}}
\newcommand{\braxi}{{\bra{\xi}}}
\newcommand{\brat}{{\bra{t}}}

\newcommand{\freq}{{(\xi,\eta,\sigma)}}

\newcommand{\thez}{{\theta_0}}
\newcommand{\theo}{{\theta_1}}
\newcommand{\thet}{{\theta_2}}
\newcommand{\theth}{{\theta_3}}

\newcommand{\thej}{{\theta_j}}

\newcommand{\ep}{{\varepsilon}}

\newcommand{\cm}{{\rm CM}}

\def\normo#1{\left\|#1\right\|}

\def\abs#1{\left|#1\right|}
\def\bra#1{\left\langle #1\right\rangle}

\def\wt#1{\widetilde{#1}}
\def\wh#1{\widehat{#1}}

\def\norm#1{\left\|#1\right\|}
\def\normo#1{\left\|#1\right\|}

\def\abs#1{\left|#1\right|}
\def\bra#1{\left\langle #1\right\rangle}

\def\wt#1{\widetilde{#1}}
\def\wh#1{\widehat{#1}}

\def\jp#1{\left\langle#1\right\rangle}

\begin{document}
\title[Long-range scattering for Dirac equations]{Long-range scattering for 2D Dirac-Hartree equations}

\author{Kiyeon Lee}
\address{Department of Mathematics, University of Michigan, Ann Arbor, MI, USA}
\email{kiyeonl@umich.edu}

\author{Changhun Yang}
\address{Department of Mathematics, Chungbuk National University, Chungdae-ro1,
	Seowon-gu, Cheongju-si, Chungcheongbuk-do, Republic of Korea}
\email{chyang@chungbuk.ac.kr}

\begin{abstract}
We investigate the long-time behavior of small solutions to the Dirac-Hartree equation in two spatial dimensions. This model describes the mean-field dynamics of relativistic fermions interacting through the three-dimensional Coulomb potential $|x|^{-1}$, which gives rise to long-range effects in the scattering dynamics. We prove global well-posedness and long-range scattering (modified scattering) for small initial data in weighted Sobolev spaces. In this setting, long-range scattering means that, unlike linear scattering, an additional logarithmic phase correction is required to describe the precise asymptotics of nonlinear solutions. Our approach relies on the space-time resonance method, combined with special null structures inherent in the equation. Compared to the three-dimensional case \cite{CKLY2022,cloos}, the novelty lies in overcoming the weaker time decay inherent to the two dimensional problem.
\end{abstract}

\maketitle

\section{Introduction}
We consider the following Dirac-Hartree equation
\begin{align}
\left\{ \begin{aligned} \bigg(-i\partial_t  + \sum_{j=1}^2\alpha^jD_j + m \beta \bigg) \psi  &= \lam \left(|x|^{-1}*|\psi|^2\right) \psi  \qquad\mathrm{in}\;\;\R \times \mathbb{R}^{2},\\
	\psi(0) &= \psi_0,
\end{aligned}
\right.\label{main-eq:dirac}
\end{align}
where the unknown $\psi:\mathbb{R}^{1+2}\to\mathbb{C}^2$ is called a spinor, the mass $m\ge0$, and $\lambda\in\R\setminus\{0\}$. The differential operator is $D = -i \nabla$ and the Pauli matrices $\al^1,\al^2,$ and $\beta$ are defined as
\begin{align*}
		\al^1  = \begin{bmatrix}
		0 \; & 1 \\ 1 & 0
	\end{bmatrix}, \quad
	\al^2  = \begin{bmatrix}
		0 & i \\ -i & 0
	\end{bmatrix}, 	\quad
	\beta  = \begin{bmatrix}
		1 & 0 \\ 0 & -1
	\end{bmatrix}.
\end{align*} 
The Dirac--Hartree equation is a fundamental model in relativistic quantum mechanics that describes the dynamics of a quantum particle interacting with a self-consistent potential. In two spatial dimensions, it is particularly relevant for systems where particles are confined to a plane, such as electrons in graphene or other two-dimensional materials. In this setting, the self-consistent potential corresponds to the effective restriction of the three-dimensional Coulomb potential $|x|^{-1}$ to a plane where the electrons are confined, which is derived from the trace of the 3D Poisson equation. This formulation accurately captures long-range electrostatic interactions between planar particles and provides a mean-field description well-suited for modeling relativistic phenomena in condensed-matter systems. We refer to \cite{graphene2009} for a recent review and to \cite{HLS-2012,hajjmehats2014,arbuspar2018-jmp} for the massless case $(m=0)$ and references therein.

Any smooth solution to \eqref{main-eq:dirac} enjoys the charge conservation law:
\begin{align}\label{charge conservation}
\|\psi(t)\|_{L^2(\R^2)}= \|\psi_0\|_{L^2(\R^2)}.
\end{align}
For the massless case $m=0$, we have the scaling symmetry and the charge is invariant under the scaling, so \eqref{main-eq:dirac} is $L^2$-critical.

In this work, we focus on the case of a positive mass term $m>0$. The presence of the mass term models a spectral gap, which is physically relevant to describe graphene with broken lattice symmetry \cite{PhysRevLett.61.2015,Semenoff2012}. By scaling, we normalize the mass to $m=1$ throughout the paper.

We investigate the global well-posedness of \eqref{main-eq:dirac} for sufficiently small initial data and describe the corresponding asymptotic behavior of such solutions.
A global solution is said to \textit{scatter} in a given function space if it converges, as time tends to infinity, to a linear solution with an asymptotic profile; this phenomenon is usually referred to as \textit{linear scattering}.
However, for nonlinear equations with long-range interactions, the linear scattering may fail since the nonlinear contribution is not time-integrable. For the ``scattering-critical'' case where the nonlinear term is barely non-integrable in time, or diverges logarithmically, which corresponds to our main equation, a global solution converges to a linear solution modified by a logarithmic phase correction.
This phenomenon is referred to as \textit{modified scattering}, or equivalently, \textit{long-range scattering}.

\subsection{Previous Results}
Nonlinear Dirac equations have been extensively studied over the past several decades:
\begin{align}\label{NDE}
 \bigg(-i\partial_t  + \sum_{j=1}^d\alpha^jD_j + m \beta \bigg) \psi  =  \lam N(\psi) \;\;\mathrm{in}\;\;\R \times \mathbb{R}^{d}, \;\; d=2,3.
\end{align}
We review previous results concerning global existence and asymptotic behavior of solutions.

\noindent $(1)$ Hartree nonlinearity: $N(\psi)= \left(V*\langle \psi,\gamma\psi\rangle \right) \gamma\psi $, where the potential $V:\R^d\rightarrow \R$ and $\gamma$ is either identity matrix $I_d$ or one of the Dirac matrices $\{\alpha^1,\cdots,\alpha^d, \beta\}$. Each choice of $\gamma$ is typically associated with a physical interpretation.
\begin{enumerate}[(i)]
\item When $\gamma=I_d$ and $V=|x|^{-1}$ which, in the two-dimensional case, corresponds to our main equation \eqref{main-eq:dirac}, local well-posedness in $H^s(\R^d)$ for $s\ge\frac12$ and $d\ge2$ can be established via a standard contraction mapping argument by using only the Sobolev embedding; see, for instance, \cite{choz2006-siam} and subsection~\ref{subsec:Setup and main Theorem} below.
For low regularity results, we refer to \cite{hajjmehats2014,lee2021-bkms} and \cite{hele2014} for the cases $d=2$ and $d=3$, respectively. Since the energy associated with \eqref{main-eq:dirac} lacks coercivity, such local results cannot be extended to an arbitrary time interval.
Concerning the global dynamics, we refer to \cite{CKLY2022,cloos} for $d=3$, where the small data global existence and the asymptotic behavior of solutions (modified scattering) were established.
In the present work, we prove the corresponding global results in two dimensions. We emphasize that the linear scattering fails to occur in this case, as has been proven in \cite{chleoz,chhole} for $d=2,3$.
\item When $\gamma=I_d$ and $V$ is replaced by a less singular potential, one can indeed expect linear scattering. For instance, in the three-dimensional case with the Yukawa potential $V(x)=\frac{e^{-\mu |x|}}{|x|}$, linear scattering was established in \cite{yang2019,hetes2015}.
\item When $\gamma=\beta$, one can exploit a special structure, for instance, in the two-dimensional case, 
$$\langle \psi,\beta\psi\rangle = |\psi_1|^2-|\psi_2|^2.$$ 
This null structure plays a crucial role in estimating nonlinear interactions. Nevertheless, despite this structure, no global results have been obtained so far in the case of the Coulomb potential. By contrast, for less singular potentials, linear scattering can be established even for rough initial data, namely data close to critical regularity; see \cite{tesfahun2020-2d,chle2021-die,chleoz} for $d=2$, and \cite{YangCPAA,tesfahun2020-3d,chhole} for $d=3$.
\item  We refer to \cite{nakatsu2012} for results on the generalized Hartree type nonlinearity. 
\end{enumerate}

\noindent (2) Power type nonlinearity: $N(\psi)=\langle \psi,\gamma\psi\rangle\gamma\psi$, where $\gamma$ is either identity matrix $I_d$ or one of the Dirac matrices. This type of nonlinearity is also of significant importance.

In this case, long-range behavior appears solely in one dimension. Furthermore, only the one-dimensional equation is $L^2$-critical. 
We refer to \cite{Candy1dglobal,Candy1dLongrangeScattering} where global existence in $L^2(\R)$ and long-range scattering results were established.
For the multi-dimensional case, we refer to 
\cite{escobedo1997semilinear,MNNO2005,DL2022} for results valid in both the massless and massive settings, 
to \cite{BC2016} for the massless case, 
and to \cite{MNO2003,BH2016,BH2015} for the massive case.

\smallskip

\noindent (3) Coupled systems: Various physical models couple with the Dirac equation and have been studied widely. Here we restrict our attention to the two-dimensional setting and global results. We refer to \cite{DLMY2024,DW2024,BB2025} for results on the asymptotic behavior of solutions to Dirac-Klein-Gordon systems, and to \cite{DS2011} for the global existence of solutions to Maxwell-Dirac systems.

\subsection{Setup and Main Theorem}\label{subsec:Setup and main Theorem}
All Dirac matrices are hermitian and they satisfy 
\begin{align*}
\alpha^i\alpha^j+\alpha^j\alpha^i= 2\delta_{ij}I_2, \; \alpha^j\beta+\beta\alpha^j=0 \;\text{ for }\; i,j=1,2, \; \text{ and } \;\beta^2=I_2.
\end{align*}
Using this, as observed in \cite{anfosel2007, DS2011}, one verifies that the square of the linear Dirac operator is diagonalized as 
\begin{align*}
\bigg(\sum_{j=1}^{2}\al^{j}\xi_{j}+\beta\bigg)^{2}=\jp{\xi}^{2}I_{2},
\end{align*}
where $\langle\xi\rangle := (1+|\xi|^{2})^{\frac{1}{2}}$
for the frequency vector $\xi\in\mathbb{R}^{2}$. 
We introduce the Dirac projection operators $\Pi_{\pm}(D)$ associated to $\pm\langle \xi\rangle$ as the Fourier multiplier given by 
\begin{align*}
\Pi_{\pm}(D):=\frac{1}{2}\bigg(I_2\pm\frac{1}{\brad}\Big[\sum_{j=1}^{2}\al^{j}D_{j}+\beta\Big]\bigg),
\end{align*}
(See Notations below).
Indeed, these are projection operators satisfying 
\begin{equation*}
\Pi_{+}(D)+\Pi_{-}(D)=I_{2},\quad \Pi_{\pm}(D)\Pi_{\pm}(D)=\Pi_{\pm}(D),\quad \Pi_{\pm}(D)\Pi_{\mp}(D)=O_2. 
\end{equation*}
With the Dirac projection operators, the linear Dirac operator can be decomposed into 
\[
\sum_{j=1}^{2}\al^{j}D_{j}+\beta =\jp D(\Pi_{+}(D)-\Pi_{-}(D)).
\]
Letting $\psi_{\pm}=\Pi_{\pm}(D)\psi$, we have  $\psi=\psi_++\psi_-$ and the Dirac-Hartree equation \eqref{main-eq:dirac} can be rewritten as:
\begin{align}
\left\{ \begin{aligned}
(-i\partial_{t} +\brad)\psi_{+} & = \lambda\, \Pi_{+}(D)\Big[(|x|^{-1}*|\psi|^{2})\psi\Big],\\
(-i\partial_{t}-\brad)\psi_{-} & =  \lambda\, \Pi_{-}(D)\Big[(|x|^{-1}*|\psi|^{2})\psi\Big],\\
\psi_{+}(0)  = \psi_{0,+}:= & \Pi_{+}(D)\psi_{0} 
\text{ and }
\psi_{-}(0)  =\psi_{0,-}:=\Pi_{-}(D)\psi_{0}, 
\end{aligned}
\right.\label{maineq-decou}
\end{align} 
which is a system of half Klein-Gordon equations coupled through the Hartree nonlinearity. Denoting by $e^{\mp it\brad}\psi_{0,\pm}$ the
free solutions to \eqref{maineq-decou}
\[
e^{\mp it\brad}\psi_{0,\pm}(x) = \frac{1}{(2\pi)^{2}}\int_{\mathbb{R}^{2}}e^{i(x\cdot\xi\mp t\langle\xi\rangle)}\widehat{\psi_{0,\pm}}(\xi)\,d\xi,
\]
we can express the free solution to \eqref{main-eq:dirac} as
\begin{equation*}
U(t)\psi_{0}=e^{-it\brad}\Pi_{+}(D)\psi_{0}+e^{it\brad}\Pi_{-}(D)\psi_{0}.	
\end{equation*}
By Duhamel's principle, the solutions to \eqref{maineq-decou} satisfy the following integral equation
\begin{align}\label{inteq0} 
\left\{ \begin{aligned}
\psi_{+}(t) & =e^{- it\langle D\rangle}\psi_{0,+}+i\lambda\int_{0}^{t}e^{- i(t-s)\langle D\rangle}\Pi_{+}(D)\Big[(|x|^{-1}*|\psi|^{2})\psi\Big](s)\, ds, \\ 
\psi_{-}(t) & =e^{ it\langle D\rangle}\psi_{0,-}+i\lambda\int_{0}^{t}e^{ i(t-s)\langle D\rangle}\Pi_{-}(D)\Big[(|x|^{-1}*|\psi|^{2})\psi\Big](s)\, ds.
\end{aligned}  \right. 
\end{align}

 We now state our main theorem for \eqref{main-eq:dirac}. \begin{thm} \label{main-thm:semi} Let $n \ge 1500$ and $k=\frac{n}{100}$. There exists $\overline{\ve_{0}}>0$
satisfying the following:

Suppose that the initial data $\psi_{0}$ is small in a weighted space.
In other words, for any $\ve_{0}\le\overline{\ve_{0}}$, if $\psi_{0}$
satisfies 
\begin{align}
\|\psi_0\|_{H^{n}(\R^2)}+\|\langle x\rangle^2 \psi_0\|_{L_x^2(\R^2)}+\|\bra{\xi}^{k}\widehat{\psi_0}\|_{L_{\xi}^{\infty}(\R^2)}\le\ve_{0}.\label{condition-initial:semi}
\end{align}
Then the Cauchy problem \eqref{main-eq:dirac} with initial data $\psi_{0}$
has a unique global solution $\psi$ to \eqref{main-eq:dirac} decaying as
\begin{align}
\|\psi(t)\|_{L^{\infty}(\R^2)}\les\ve_{0}\bra{t}^{-1}.\label{global-bound:semi}
\end{align}
Moreover, there exists a scattering profile $\psi^\infty:\R^2\rightarrow\C^2$ such that 
\begin{align}\label{eq:modified-scattering}
	\left\| \bra{\xi}^k \mathcal F \Big\{  \psi (t) - U_{\mathcal B}(t)U(t) \psi^\infty\Big\} \right\|_{L_\xi^\infty(\R^2)} \les \ve_0 \bra{t}^{-\de},
\end{align}
for some $0<\de\ll1$, where $U(t)$ denotes the linear propagator to \eqref{main-eq:dirac} and the phase modification is defined by
\[
U_{\mathcal B}(t) := \sum_{\theta \in \{  \pm\} }e^{i\mathcal B_{\theta}(t,  D)}\Pi_{\theta}(D),
\]
with 
\begin{align}
\mathcal B_{\theta}(t,\xi)  
:=\sum_{\theta'\in\{+,-\}}
\frac{\lambda}{(2\pi)^{2}}
\int_{0}^{t}
\left( \int_{\mathbb{R}^{2}}
\left|\theta\frac{\xi}{\langle\xi\rangle}-\theta'\frac{\sigma}{\langle\sigma\rangle}\right|^{-1}
\abs{\wh{\Pi_{\theta'}(D)\psi}(\sigma)}^{2}d\sigma \right) \frac{ \rho(s^{-\frac{2}{n}}\xi)}{\langle s\rangle}ds,
\label{modified-phase}
\end{align}
where $\rho$ is a smooth cutoff function supported in $B(0,2)$.
\end{thm}

\begin{rem}
\begin{enumerate}
\item \eqref{eq:modified-scattering} implies modified scattering in $L^2(\R^2)$. We prefer to express the formula of phase modification (1.9) in Fourier space, since it naturally arises both from heuristic considerations and from the rigorous proof.
\item We do not aim at optimizing the regularity indices $n$ and $k$ and the time decay exponent $\delta>0$ in Theorem~\ref{main-thm:semi}.
\item The time decay rate of solutions in \eqref{global-bound:semi} is optimal, in the sense that the
nonlinear solutions decay at the same rate as the linear ones \eqref{eq:time-decay}.
\item The definition of $\mathcal B_\theta$ is admittedly complicated, but intrinsic to the ODE analysis (see proof of \eqref{eq:crucial-part} below). We also refer to \cite{kapu} for the simpler derivation in the Schr\"odinger case, which captures the essential idea.
\end{enumerate}	
\end{rem}
 
\subsection{Idea of Proof}
The phenomenon of modified scattering of small solutions typically arises when the nonlinearity involves long-range interactions, a phenomenon that has been extensively studied. Ozawa \cite{ozawa1991} first established modified scattering for the one-dimensional cubic nonlinear Schr\"odinger equation. Hayashi and Naumkin \cite{hayashi-naumkin1998} extended this result to higher dimensions, including the Hartree nonlinearity, by means of the factorization method.
Later, Kato and Pusateri \cite{kapu} provided an alternative proof based on the space-time resonance method developed by Germain-Masmoudi-Shatah \cite{gemasha2008,gemasha2012-jmpa,gemasha2012-annals}. All of these approaches rely on a bootstrap framework, where the central difficulty is to establish weighted energy estimates and sharp dispersive-type estimates. In particular, one needs to estimate the weighted norms such as $\|(x+2it\nabla_x)^ku(t)\|_{L^2(\R^d)}$ and control $\|u(t)\|_{L^\infty(\R^d)}$.

Pusateri \cite{pusa} further extended these ideas to equations with nonlocal dispersive operators, most notably the three-dimensional semi-relativistic Hartree equation
\begin{align}\label{eq:semi-relativistic}
-i\partial_t u + \sqrt{1-\Delta}\, u = \lambda (|x|^{-1}\ast|u|^2)u, \quad u:\R\times\R^3\rightarrow \C,
\end{align}
with the associated resonance function 
\begin{align}\label{resonance for scalar eqn}
 p(\xi,\eta,\sigma) = \langle \xi\rangle - \langle \xi+\eta\rangle  +\langle \xi+\eta+\sigma\rangle -\langle \xi+\sigma\rangle.
\end{align}
Unlike the Schr\"odinger equations, whose local operator admits natural vector fields such as $x+2it\nabla_x$, the half Klein-Gordon operator lacks such structure, making the analysis more delicate. In the weighted energy estimates, the central difficulty is to control the singularity generated by the Hartree potential, whose Fourier transform is $\mathcal{F}(|x|^{-1})(\eta)=c|\eta|^{-2}$.
For this scalar model, the key tool is the \textit{phase null structure}, namely $\nabla_\xi p(\xi,\eta,\sigma)|_{\eta=0}=0$, which ensures cancellation of singularities on the Fourier side, combined with Coifman-Meyer type estimates. In addition, space resonance analysis also plays a crucial role in obtaining the dispersive-type bounds.

This framework was subsequently applied to the three-dimensional Dirac-Hartree system \eqref{NDE} by Cho, Kwon, and the present authors \cite{CKLY2022}.
Compared to scalar equations, the Dirac system requires treating various resonance configurations arising from the spinor structure
(see \eqref{eq:nonlinear term} and \eqref{phase interaction q}) 
\begin{align*}
	q_{\mathbf{\Theta}}(\xi,\eta,\sigma)&=\theta \langle\xi\rangle - \theo\langle\xi+\eta\rangle  +\thet \bra{\xi+\eta+\sigma} - \theth \bra{\xi+\sigma},
\end{align*}
where $\mathbf{\Theta}=(\theta,\theta_1,\theta_2,\theta_3)$ and $\theta_j\in\{+,-\}$.
If all signs coincide, this reduces to the scalar resonance function in \eqref{resonance for scalar eqn}.
However, for certain sign configurations the phase null structure is no longer available, and one must instead exploit the \textit{Dirac null structure} or resort to normal form arguments relying on time non-resonances. This constituted the key new ingredient in the analysis of \cite{CKLY2022}.

In contrast to the Schr\"odinger case, where the strategy for modified scattering remains essentially the same in different dimensions, dispersive equations with nonlocal differential operators present qualitatively different technical challenges in low dimensions, in particular $2D$. The available time decay of solutions is weaker ($|t|^{-1}$, instead of $|t|^{-\frac32}$ in $3D$), which poses new difficulties. On the other hand, the singularity of the Hartree potential in the Fourier space is milder in $2D$, since $\mathcal{F}(|x|^{-1})(\eta)=c_d|\eta|^{-(d-1)}$ in $\R^d$, which is advantageous. Thus, one has to strike a delicate balance between the weaker decay and the milder singularity. 

While the overall strategy still relies on the bootstrap argument and space-time resonance analysis, both the weighted energy estimates and the control of scattering norms in two dimensions require a substantially different analytical approach and present qualitatively new challenges compared to the three-dimensional case. 
In particular, the second-order weighted terms produce a quadratic growth in time, which cannot be compensated by decay of solution, i.e., $|t|^{-1}$ alone unlike the $3D$ case. 
This necessitates the adoption of a substantially different strategy, where space-time resonance analysis becomes essential. 

This strategy was precisely realized in the two-dimensional semi-relativistic Hartree equation \eqref{eq:semi-relativistic}, where modified scattering was established by Kwon and the present authors  \cite{kly2023} through a careful use of space non-resonances.
A central feature of that work was the observation that space non-resonant structures, while improving time decay, may also produce additional singularities, which must in turn be canceled by suitable null structures. The balance between gaining decay and controlling the induced singularities lies at the heart of the method.

In the present paper concerning the Dirac equation, we must again deal with all sign configurations, including those absent from the scalar model. A novel element of our approach is the discovery of an algebraic identity \eqref{Key indentity2} that links certain multipliers to derivatives of the phase function, allowing us to combine phase null structures with space resonance analysis. Here, the Dirac null structure is also indispensable.

Finally, the asymptotic analysis leading to modified scattering follows the standard ODE-type scheme: introducing the interaction representation $\Psi_\theta=e^{it\langle D\rangle}\psi_\theta$, reducing the system to an effective ODE for the profile, and identifying the leading contribution, which involves a logarithmic phase correction. The main logarithmic divergence arises from the near-singular region involving $|\eta|^{-1}$, and one must precisely identify the relevant time-dependent scale to separate this region. The remaining contributions are shown to be integrable in time by exploiting either the Dirac null structure or the space-time resonance method, depending on the sign configuration.

\subsection{Future Directions} 
\;\; (1) Generalized potential $|x|^{-\gamma}$. 

\noindent We consider the following Dirac-Hartree equations with generalized potential:
\begin{align*}
\bigg(-i\partial_t  + \sum_{j=1}^d\alpha^jD_j + m \beta \bigg) \psi  &= \lam \left(|x|^{-\gamma}*|\psi|^2\right) \psi  \qquad\mathrm{in}\;\;\R \times \mathbb{R}^{d},
\end{align*}
where $0<\gamma<d$.
A heuristic computation yields that the time decay of the Hartree nonlinear term with the generalized potential, namely $\|\left(|x|^{-\gamma}*|\psi(t)|^2\right) \psi(t)\|_{L^2(\R^d)}$, when evaluated on a linear solution, behaves like $|t|^{-\gamma}$. From the viewpoint of scattering, i.e., time integrability of this term,
$1<\gamma<d$ corresponds to the short range regime and $\gamma \le 1$ to the long-range regime, with $\gamma=1$ corresponding to the critical case.
It remains an interesting open problem to establish small data linear scattering for the Dirac-Hartree equation with short range potentials, both in two and three dimensions.

\smallskip 

(2) Chern-Simons-Dirac systems.

\noindent Another physically important two-dimensional model coupled with the Dirac equation is the Chern-Simons-Dirac systems.
Under the Coulomb gauge condition, \footnote{\, We refer to \cite{bourcanma2014-dcds} for its derivation.} it reduces to a nonlinear Dirac equation with Hartree-type nonlinearity 
\begin{align}
    \bigg(-i\partial_{t}+ \sum_{j=1}^2\alpha^jD_j +m\beta\bigg)\psi & =N(\psi,\psi)\psi,\label{eq:csd-coulomb}
    \end{align}
where the unknown $\psi:\R^{1+2}\to\C^{2}$ and the nonlinear term is given as 
\[
N(\psi,\psi)=\frac{1}{\Delta}\left[\Big(\partial_{1}(\psi^{\dagger}\al^{2}\psi)-\partial_{2}(\psi^{\dagger}\al^{1}\psi)\Big)+\Big(\partial_{2}(|\psi|^{2})\al^{1}-\partial_{1}(|\psi|^{2})\al^{2}\Big)\right]
\]
While local well-posedness results have been studied extensively under various gauge choices \cite{Okamoto2013,bourcanma2014-dcds,HO2016,Pecher2016}, not only long-time behavior but also the global existence of solutions to \eqref{eq:csd-coulomb} remains unknown. In particular, the Hartree-type potentials appear as $\frac{\eta_j}{|\eta|^2}$ for $j=1,2$ in Fourier space exhibiting the same singular behavior of order $-1$ near the origin as the Coulomb potential $|x|^{-1}$ in \eqref{main-eq:dirac}.
Thus, \eqref{eq:csd-coulomb} can also be regarded as scattering-critical, and modified scattering is naturally expected. However, the nonlinearity consists of two distinct types with respect to the Dirac null structures: for the bilinear term $\psi^{\dagger}\al^{j}\psi$ the null structure is absent, whereas for the remaining terms it can be exploited.
Since our present analysis crucially depends on the null structure, it cannot be directly applied to solve \eqref{eq:csd-coulomb}.
Nevertheless, we expect that the methodology developed in this paper, combined with a refined study of the resonance set and associated null structures, will provide essential tools for addressing the global dynamics of \eqref{eq:csd-coulomb}.  

\smallskip 

(3) Massless case $(m=0)$.

\noindent The massless case $(m=0)$ is also of physical interest, but our argument does not apply in this setting, since even for linear solutions one can only expect the weaker decay rate $|t|^{-1/2}$.
On the other hand, the massless case enjoys a favorable scaling structure, which makes it possible to employ vector field methods. For instance, in \cite{DLW2021,DLMY2024,DL2022,DW2024},
such techniques were used to establish the asymptotic behavior of $2D$ cubic Dirac equations and coupled Dirac systems. We expect that combining these vector field approaches with the methods developed in this paper may lead to progress on the massless Dirac-Hartree equation.

\subsection*{Organization of the paper}
In Section~2, we collect the preliminary tools needed for the analysis of global existence and asymptotic behavior. We introduce the function spaces and the a priori assumption, describe the phase and spinor structures along with the associated null structure, recall Coifman-Meyer type multiplier estimates, and derive time decay estimates for frequency-localized solutions. In Section~3, we set up the bootstrap framework and prove the key contraction estimate, reducing the analysis to two central propositions: the weighted energy estimates (Proposition \ref{prop-energy}) and the scattering bounds (Proposition \ref{prop:scattering}). These results together close the bootstrap argument and yield the global existence and modified scattering behavior. In Section~4, we prove the weighted energy estimates  in Proposition \ref{prop-energy}, which constitute the core of our analysis. In Section~5, we establish the scattering bounds  in Proposition \ref{prop:scattering}, completing the modified scattering analysis.

\subsection*{Notations}  
\noindent $\bullet$ (Fourier transform)
$\mathcal{F}f(\xi)\big(=\widehat{f}(\xi)\big):=\int_{\R^2}e^{-ix\cdot\xi}f(x)dx$ and $\mathcal{F}^{-1}g(x):=\frac{1}{(2\pi)^2}\int_{\R^2}e^{ix\cdot\xi}g(\xi)d\xi$.

\noindent $\bullet$ (Mixed-normed spaces) For a Banach space $X$
and an interval $I$, $\psi \in L_{I}^{q}X$ if and only if  $\psi(t)\in X$ for a.e. $t\in I$
and $\|\psi\|_{L_{I}^{q}X}:=\|\|\psi(t)\|_{X}\|_{L_{I}^{q}}<\infty$. Especially,
we denote $L_{I}^{q}L_{x}^{r}=L_{t}^{q}(I;L_{x}^{r}(\R^2
))$, $L_{I,x}^{q}=L_{I}^{q}L_{x}^{q}$, and 
$L_{t}^{q}L_{x}^{r}=L_{\mathbb{R}}^{q}L_{x}^{r}$.

\noindent $\bullet$ As usual, different positive constants are denoted by the same letter $C$ unless otherwise specified.
$A\lesssim B$ and $A\gtrsim B$ mean that $A\le CB$ and $A\ge C^{-1}B$,
respectively for some $C>0$. $A\sim B$ means that $A\lesssim B$
and $A\gtrsim B$.

\noindent $\bullet$ (Japanese bracket)  $\langle x \rangle:=(1+|x|^{2})^{\frac{1}{2}}$.

\noindent $\bullet$ (Fourier multiplier) Let \(D:=-i\nabla\).
For a measurable function \(m:\R^2\to\C\), the Fourier multiplier \(m(D)\) is defined by
\[
\mathcal F(m(D)f)(\xi)=m(\xi)\widehat f(\xi).
\]
Fourier multipliers with vector-, matrix-, or tensor-valued symbols are defined in the same manner, by applying the symbol to the Fourier transform in the frequency variable.

\noindent $\bullet$ (Littlewood--Paley operators)
Let $\rho\in C_0^\infty(B(0,2))$ be a smooth radial cutoff function satisfying
\[
\rho(\xi)=1 \qquad \text{for } |\xi|\le 1.
\]

For $R>0$, we define the low- and high-frequency cutoff functions by
\[
\rho_{\le R}(\xi):=\rho\!\left(\frac{\xi}{R}\right),
\qquad
\rho_{>R}(\xi):=1-\rho_{\le R}(\xi).
\]

For each dyadic number $N\in2^{\mathbb N\cup\{0\}}$, we define the inhomogeneous dyadic cutoff function
\[
\rho_N(\xi):=
\begin{cases}
\rho(\xi), & N=1,\\[1ex]
\rho\!\left(\dfrac{\xi}{N}\right)
-
\rho\!\left(\dfrac{2\xi}{N}\right),
& N>1.
\end{cases}
\]
These cutoff functions satisfy
\[
\sum_{N\in2^{\mathbb N\cup\{0\}}}\rho_N(\xi)=1,
\qquad
\xi\in\mathbb R^2.
\]
The corresponding inhomogeneous Littlewood--Paley projection $P_N$ is defined by
\[
\mathcal F(P_Nf)(\xi)
=
\rho_N(\xi)\widehat f(\xi).
\]
For convenience, we write
\[
f_N:=P_Nf.
\]

For the homogeneous dyadic decomposition, we define the cutoff functions
\[
\chi_K(\xi):=
\rho\!\left(\frac{\xi}{K}\right)
-
\rho\!\left(\frac{2\xi}{K}\right),
\qquad
K\in2^{\mathbb Z}.
\]
These cutoff functions satisfy
\[
\sum_{K\in2^{\mathbb Z}}\chi_K(\xi)=1,
\qquad
\xi\in\mathbb R^2\setminus\{0\}.
\]
The corresponding homogeneous Littlewood--Paley projection $\dot P_K$ is defined by
\[
\mathcal F(\dot P_Kf)(\xi)
=
\chi_K(\xi)\widehat f(\xi).
\]

The families
$\{P_N\}_{N\in2^{\mathbb N_0}}$
and
$\{\dot P_K\}_{K\in2^{\mathbb Z}}$
constitute the inhomogeneous and homogeneous dyadic decompositions, respectively.

 Throughout the paper, the symbols \(K,L\in2^{\mathbb Z}\) will be used for homogeneous dyadic scales,
while \(N\in2^{\mathbb N\cup\{0\}}\) will always denote an inhomogeneous dyadic scale.
In particular, the cutoffs \(\chi_K\) and \(\chi_L\) are associated with homogeneous frequency localizations, whereas \(\rho_N\), $\rho_{\le R}$,
and
$\rho_{>R}$
are associated with the inhomogeneous frequency decomposition.

\noindent $\bullet$ (Inner product) For $\mathbf{u},\mathbf{v}\in \C^2$, $\langle \mathbf{u},\mathbf{v}\rangle = \mathbf{u}^{\dagger}\mathbf{v}$, where $\mathbf{u}^{\dagger} = \overline{\mathbf{u}^T}$.

\noindent $\bullet$ (Tensor product) Let $\textbf{A}=(A_i)\in \C^m, \textbf{B}=(B_i) \in \C^n$. Then $\textbf{A} \otimes \textbf{B}$ denotes the standard tensor product such that $$(\textbf{A} \otimes \textbf{B})_{ij} = A_iB_j, \quad 1\le i\le m, \; 1\le j\le n.
$$ For convenience, and whenever no confusion arises, we adopt the shorthand notation
$$ \textbf{A}\textbf{B} := \textbf{A} \otimes \textbf{B} ,\qquad
\mathbf A^k := \overbrace{\mathbf A \otimes \cdots \otimes \mathbf A}^{k\; \text{times}},\qquad \nabla^k := \overbrace{\nabla \otimes \cdots \otimes \nabla}^{k\; \text{times}}.
$$

\subsection*{Acknowledgement}
K. Lee was supported by the National Research Foundation of Korea (NRF) grant funded in part by the Korea government (No. RS-2025-00514043) and  (No. RS-2024-00463260). C. Yang was supported in part by the Korea government (No. 2021R1C1C1005700).

\section{Preliminaries}
\subsection{Solution space}
Given small initial data satisfying \eqref{condition-initial:semi}, by a standard local theory in weighted energy spaces, we readily obtain a small local solution $\psi_{\theta}(t)$ on $[0,T_0]$ for $\theta\in\{+,-\}$. Our goal is to establish global solutions and analyze their asymptotic behaviors.
To this end, we introduce an a priori smallness assumption on the solution, which permits time growth in the energy norms.
Let $\delta_0>0$ be sufficiently small.\footnote{The precise size of $\delta_0$ is irrelevant and will not be tracked throughout the proof.} For $\ve_{1}>0$ to be chosen later, we assume an a priori smallness of solutions: for a large time $T>0,$
\begin{align}
\|\psi\|_{\Sigma_{T}}
:=\|\Pi_+(D)\psi\|_{\Sigma_{T}^+} + \|\Pi_-(D)\psi\|_{\Sigma_{T}^-}
\les\ve_{1},\label{assumption-apriori}
\end{align}
where 
\begin{align*}
\begin{aligned}\|\phi\|_{\Sigma_{T}^{\pm}} & :=\sup_{t\in[0,T]}\Big[ \brat^{-\de_{0}}
	\|\phi(t)\|_{H^{n}(\R^2)}
	+\brat^{-\de_{0}}\|x e^{\pm it\jp D}\phi(t)\|_{L^{2}(\R^2)} \phantom{]}\\
 & \phantom{[}\qquad\qquad\qquad\quad+\brat^{-2\de_{0}}\|x^{2}e^{\pm it\jp D}\phi(t)\|_{L^{2}(\R^2)}+\left\Vert \braxi^{ k}\widehat{\phi(t)}\right\Vert _{L_{\xi}^{\infty}(\R^2)}\Big].
\end{aligned}
\end{align*}
Under the a priori assumption, we have the sharp pointwise decay of solutions.
\begin{prop}[Proposition~2.1 of \cite{kly2023}]
Assume that $\psi$ satisfies the a priori assumption \eqref{assumption-apriori}
for $T>0$ and small $\ve_{1}>0$. Then, we have 
\begin{align}
\|\psi(t)\|_{W^{\ell,\infty}}\lesssim \langle t\rangle^{-1}\ve_1,\label{eq:time-decay}
\end{align}
for any $0\le t\le T$ and $0 \le \ell \le k$. 
\end{prop}
Interpolating the decay estimates \eqref{eq:time-decay} with the charge conservation laws \eqref{charge conservation}, we obtain 
\begin{align}\label{Lptimedecay}
	\| \psi(t) \|_{L^p} \lesssim  \langle t\rangle^{-2(\frac12-\frac1p)}\ve_1,
\end{align}
for $2\le p \le \infty$.

\subsection{Null structure}
We often work with the interaction representation of $\psi_\theta(t)$ in order to track the scattering states:
\begin{align*}
\Psi_\theta(t,x):=e^{\theta it\langle D\rangle}\psi_\theta(t,x), \; \text{ for } \; \theta\in\{+,-\}.
\end{align*}
Then, \eqref{inteq0} can be rewritten in terms of $\Psi_{\theta}$ as
\begin{align}\label{eq in terms of Psi}
\Psi_\theta(t,x) = \psi_{0,\theta}(x)+i\lam \int_0^t e^{\theta is\langle D\rangle} \Pi_{\theta}(D)\Big[(|x|^{-1}*|\psi|^{2})\psi\Big](s)\, ds,
\end{align}
and, by taking the Fourier transform, we obtain
\begin{align}
\begin{aligned}  \widehat{\Psi_\theta}(t,\xi)&=\widehat{\psi_{0,\theta}}(\xi)+i\frac{\lam}{2\pi} \sum_{\theta' \in \{+,-\}}\int_0^t \mathcal{N}_{(\theta,\theta')}(s,\xi)ds,\\
 \mathcal N_{(\theta,\theta')}(s,\xi)&=
\int_{\mathbb{R}^{2}}e^{isp_{(\theta,\theta')}(\xi,\eta)} \Pi_\theta(\xi)\Pi_{\theta'}(\xi-\eta)|\eta|^{-1} \widehat{\Psi_{\theta'}}(s,\xi-\eta) \widehat{|\psi|^2}(s,\eta) d\eta,
\end{aligned}\label{eq:duhamel}
\end{align}
where the phase interaction is given by 
\begin{align*}
	p_{(\theta,\theta')}(\xi,\eta)=\theta \langle\xi\rangle - \theta'\langle\xi-\eta\rangle.
\end{align*}
We identify two null structures that play a crucial role in our analysis: one arising from the phase function and the other from the Dirac projection operators.

We introduce the notation
\begin{align}\label{def:Z}
	\boldsymbol{\mathcal{Z}}_{(\theta,\theta')}(x,y):=\theta\frac{x}{\langle x\rangle}-\theta'\frac{y}{\langle y\rangle}, \; \text{ for } \; x,y\in\R^2,
\end{align}
which has already appeared in the definition of phase correction $\mathcal{B}_{\theta}$. Then one easily finds that 
\begin{align*}
\nabla_{\xi}p_{(\theta,\theta')}(\xi,\eta)=\boldsymbol{\mathcal{Z} }_{(\theta,\theta')}(\xi,\xi-\eta).
\end{align*}
The following lemma can be readily verified by a direct computation, so we give the statement without the proof.
\begin{lemma}[Phase null structure]
Let $\theta,\theta' \in \{+,-\}$. For $x,y\in\R^2$, we have 
 \begin{align}\label{ineq:Phase null structure}
\min \left( 1, |x|, \frac{|\theta x -\theta' y|}{\bra{x}^3}\right)\lesssim \left| \boldsymbol{\mathcal{Z}}_{(\theta,\theta')}(x,y)\right| \lesssim \frac{|\theta x- \theta' y|}{\max(\langle x \rangle,\langle y \rangle)}.
 \end{align}
\end{lemma}
\noindent Here, the upper bound reflects the null structure of the phase function when the two signs coincide ($\theta=\theta'$), since 
$$|\nabla_{\xi}p_{(\theta,\theta)}(\xi,\eta)|=\big|\boldsymbol{\mathcal{Z} }_{(\theta,\theta)}(\xi,\xi-\eta)\big| \lesssim \frac{|\eta|}{\max(\langle\xi\rangle,\langle\xi-\eta\rangle)}.$$
We remark that the lower bound will be used several times throughout our analysis. In particular, it becomes crucial in exploiting the space resonance structure when the signs coincide. In this case, we alternatively use the following equivalent form 
\begin{align}\label{ineq:Phase lower bound}
\big|\boldsymbol{\mathcal{Z} }_{(\theta,\theta)}(x,y)\big|	 \gtrsim \frac{|x-y|}{\max(\langle x \rangle, \langle y \rangle) \min(\langle x \rangle, \langle y \rangle)^2}.
\end{align}

Second, when the two signs differ ($\theta\neq \theta'$), by following the argument in \cite[Lemma 3.1]{CKLY2022}, we obtain a null structure from the interaction of Dirac projection operators.
\begin{lemma}[Dirac null structure]
Let $\theta\in \{+,-\}$. For $\xi,\eta\in \R^2$ satisfying $|\eta|\ll |\xi|$, we have
\begin{align}\label{eq:null}
\left| \nabla_{\eta}^{m}\nabla_{\xi}^{n}\Pi_\theta(\xi) \Pi_{-\theta}(\xi \pm \eta)\right| \les |\eta|^{1-m}\max\left(
	\frac{1}{\langle\xi\rangle^{m+1}}\; , \;\frac{1}{\langle\xi-\eta \rangle^{m+1}} 
\right).
\end{align}
\end{lemma}

\subsection{Multiplier estimates}
To exploit the time decay \eqref{eq:time-decay} in our main proof, we next recall several useful auxiliary multiplier estimates.
\begin{lemma}[Coifman-Meyer operator estimates]
Let $1\le p,q,r\le \infty$. Assume that a tensor-valued multiplier $\mathcal{M}$ defined on $\R^2\times\R^2$ satisfies that
	\[
		\|\mathcal{M}\|_{\textup{CM}[(\R^2)^2]}:=\left\Vert \int\!\!\!\!\int_{\mathbb{R}^{2}\times \R^2}\mathcal{M}(\xi,\eta)e^{ix\cdot\xi}e^{iy\cdot\eta}\,d\eta d\xi \right\Vert _{L_{x,y}^{1}(\R^2\times\R^2)} < \infty.
	\]
	Then for $\frac1p +\frac1q =\frac12$, 
	\begin{align}\label{eq:coif-1}
		\left\Vert \int_{\mathbb{R}^{2}}\mathcal{M}(\xi,\eta)\widehat{\psi}(\xi\pm\eta)\widehat{\phi}(\eta)\,d\eta\right\Vert _{L_{\xi}^{2}(\R^2)}\les \|\mathcal{M}\|_{\textup{CM}[(\R^2)^2]} \|\psi\|_{L^{p}(\R^2)}\|\phi\|_{L^{q}(\R^2)}.
	\end{align}
and for $\frac1p +\frac1q + \frac1r =1$, 
\begin{align}\label{eq:coif-1-2}
	\left| \int\!\!\!\!\int_{\mathbb{R}^{2} \times \R^2}\mathcal{M}(\eta,\sigma)\widehat{\psi}(\eta \pm \sigma)\widehat{\phi}(\eta) \wh{\varphi}(\sigma) \,d\sigma d\eta\right|\les \|\mathcal{M}\|_{\textup{CM}[(\R^2)^2]} \|\psi\|_{L^{p}(\R^2)}\|\phi\|_{L^{q}(\R^2)}\|\varphi\|_{L^{r}(\R^2)}.
\end{align}
Moreover, if a multiplier $\widetilde{\mathcal{M}}$ defined on $\mathbb{R}^{2}\times \R^2\times\R^2$ satisfies that
\begin{equation*}
\|\widetilde{\mathcal{M}}\|_{\textup{CM}[(\R^2)^3]}:=\left\Vert \int\!\!\!\!\int\!\!\!\!\int_{\mathbb{R}^{2}\times \R^2\times\R^2}\widetilde{\mathcal{M}}(\xi,\eta,\sigma)e^{ix\cdot\xi}e^{iy\cdot\eta}e^{iz\cdot\sigma} \,d\sigma d\eta  d\xi\right\Vert _{L_{x,y,z}^{1}(\R^2\times\R^2\times\R^2)}<\infty 
\end{equation*}
then, for $\frac1p +\frac1q + \frac1r = \frac12$, one has
	\begin{align}\label{eq:coif-2}
	\left\Vert \int\!\!\!\!\int_{\mathbb{R}^{2}\times\R^2}\widetilde{\mathcal{M}}(\xi,\eta,\sigma)\widehat{\psi}(\xi-\eta)\widehat{\phi}(\eta-\sigma)\wh{\varphi}(\sigma)\,d\sigma d\eta \right\Vert _{L_{\xi}^{2}(\R^2)}\les   \|\widetilde{\mathcal{M}}\|_{\textup{CM}[(\R^2)^3]} \|\psi\|_{L^{p}(\R^2)}\|\phi\|_{L^{q}(\R^2)}\|\varphi\|_{L^{r}(\R^2)}.
\end{align}
\end{lemma}

\section{Proof of Theorem~\ref{main-thm:semi}}
To establish the global existence and the decay bound \eqref{global-bound:semi}, we employ a standard bootstrap argument based on the weighted energy norm and Fourier amplitude within $\Sigma_T$ (see \eqref{assumption-apriori}). As a preliminary step, it is necessary to verify the local well-posedness of \eqref{main-eq:dirac} in $\Sigma_T$. Since this follows from standard arguments (see, for instance, \cite{lee2021-bkms,choz2006-siam}), we omit the proof.

Now, given $T>0$, we assume that $\psi$ is a solution to \eqref{main-eq:dirac} on $[0,T]$ with initial data satisfying \eqref{condition-initial:semi}. Our task is to show that, for sufficiently small $\ve_1>0$, there exists a constant $C>0$ such that
\begin{align}\label{eq:contraction}
\|\psi\|_{\Sigma_T} \le \ve_0 + C\ve_1^3.
\end{align}
Once this bound is obtained, the bootstrap argument allows us to extend the solution globally in time. The desired inequality \eqref{eq:contraction} will be derived from the following two propositions.

\begin{prop}[Weighted energy estimate]\label{prop-energy} Assume
that $\psi\in C([0,T],H^{n}(\R^2))$ satisfies the a priori assumption \eqref{assumption-apriori} for some $\ve_{1}>0$ with initial data satisfying \eqref{condition-initial:semi}
for $\ve_{0}>0$.  Then the following estimates hold:
	\begin{align}
 \sup_{t\in[0,T]}\langle t\rangle^{-\de_{0}}\| \psi_{\theta }(t)\|_{H^{n}(\R^2)} &\le\ve_{0}+C\ve_{1}^{3},\label{eq:Sobolev}\\
 \sup_{t\in[0,T]}\langle t\rangle^{-\de_{0}}\| x e^{\theta it\jp D}\psi_{\theta }(t)\|_{L_x^{2}(\R^2)} &\le\ve_{0}+C\ve_{1}^{3},\label{eq:first-moment}\\
 \sup_{t\in[0,T]}\langle t\rangle^{-2\de_{0}}\| x^2 e^{\theta  it\jp D}\psi_{\theta }(t)\|_{L_x^{2}(\R^2)} & \le\ve_{0}+C\ve_{1}^{3},\label{eq:second-moment}
\end{align}
for $\theta \in \{+,-\}$.
\end{prop}
The proof of Proposition~\ref{prop-energy} constitutes the main part of our analysis and will be given in Section~4.
\begin{prop}\label{prop:scattering} Assume that $\psi \in C([0,T],H^n(\R^2))$ satisfies the a priori assumption \eqref{assumption-apriori}. 	
Then the following estimate holds:
\begin{align}\label{eq:scattering}	
	\left\|\bra{\xi}^k \left( e^{-i\mathcal{B}_\theta(t_2,\xi)} \wh{\Psi_\theta}(t_2,\xi) - e^{-i\mathcal{B}_\theta(t_1,\xi)} \wh{\Psi_\theta}(t_1,\xi)   \right) \right\|_{L_\xi^\infty(\R^2)} \les \ve_1^3 \bra{t_1}^{-\de}.
\end{align}
for $\theta\in\{+,-\}$, $t_1 \le t_2 \in [0,T]$ and some $0<\de \ll 1$.  The phase correction $\mathcal{B}_\theta$ is defined in \eqref{modified-phase}.
\end{prop}
The proof of Proposition~\ref{prop:scattering} will be given in Section~5. We observe that \eqref{eq:scattering} implies the boundedness of the scattering norm
\[
\sup_{t\in[0,T]}\|\langle \xi\rangle^{k} \widehat{\psi_{\theta}}(t)\|_{L_\xi^\infty(\R^2)} \le \ve_0 + C \ve_1^3.
\]
Together with the weighted energy estimates \eqref{eq:Sobolev}--\eqref{eq:second-moment}, this yields \eqref{eq:contraction},
which closes the bootstrap argument.

Finally, to describe the asymptotic behavior, we define the scattering profile as
\[
\psi_{\theta}^\infty := \mathcal F^{-1} \left( \lim_{t \to \infty}  e^{-i\mathcal B_\theta(t,\cdot)} \wh{\Psi_\theta}(t,\cdot) \right),\;\; \text{for } \theta\in\{+,-\}.
\]
Then Proposition~\ref{prop:scattering} immediately yields 
\begin{align*}	
\left\|\bra{\xi}^k \left(  \widehat{\psi_\theta}(t,\xi) -  e^{i\mathcal{B}_\theta(t,\xi)}e^{-\theta it\bra{\xi}}\wh{\psi_\theta^\infty}(\xi)   \right) \right\|_{L_\xi^\infty} \les \ve_0 \bra{t}^{-\de},
\end{align*}
which implies \eqref{eq:modified-scattering} upon setting $\psi^\infty:= \psi_+^\infty + \psi_-^\infty$.

\section{Weighted Energy estimates: Proof of Proposition~\ref{prop-energy}}
In this section, we establish the weighted energy estimates that play an important role in the bootstrap argument. Each of the three inequalities is proved separately in the following subsections.

\subsection{Proof of \eqref{eq:Sobolev}}
Using the Hardy--Littlewood--Sobolev type inequality (see Lemma 3.2 in \cite{choz2006-siam}), we estimate the nonlinear term 
\begin{align}\begin{aligned}\label{ineq:Hartree term}
\big\| (|x|^{-1}\ast |\psi|^2)\psi\big\|_{H^n(\R^2)}
&\lesssim  \big\| |x|^{-1}\ast |\psi|^2\|_{L^\infty(\R^2)}\|\psi\big\|_{H^n(\R^2)}
+  \big\| (|x|^{-1}\ast |\psi|^2)\|_{ W^{n,4}(\R^2)}\|\psi\big\|_{L^4(\R^2)} \\ 
&\lesssim \| \psi\|_{L^2(\R^2)}\|\psi\|_{L^\infty(\R^2)} \|\psi\|_{H^n(\R^2)} + \big\| |\psi|^2 \|_{ W^{n,\frac43}(\R^2)}\|\psi\big\|_{L^4(\R^2)} \\ 
&\lesssim \| \psi\|_{L^2(\R^2)}\|\psi\|_{L^\infty(\R^2)} \|\psi\|_{H^n(\R^2)} + \big\| \psi \|_{H^n(\R^2)}\|\psi\big\|_{L^4(\R^2)}^2. 
\end{aligned}\end{align}
Applying this inequality to \eqref{inteq0} and using the time decay estimates \eqref{Lptimedecay}, we obtain 
\begin{align*}
 &\|\psi_{\theta}(t)\|_{H^n(\R^2)} \\
 &\lesssim  
 \| \psi_{0,\theta}\|_{H^n(\R^2)} + \int_0^t \big\| (|x|^{-1}\ast |\psi(s)|^2)\psi(s)\big\|_{H^n(\R^2)} ds \\ 
 &\lesssim 
 \| \psi_{0,\theta}\|_{H^n(\R^2)} + \int_0^t \| \psi(s)\|_{L^2(\R^2)}\|\psi(s)\|_{L^\infty(\R^2)} \|\psi(s)\|_{H^n(\R^2)} + \big\| \psi(s) \|_{H^n(\R^2)}\|\psi(s)\big\|_{L^4(\R^2)}^2  ds \\ 
 &\lesssim 
 \ep_0 + \varepsilon_1^3 \bra{t}^{\delta_0},
\end{align*}
where we used that 
\begin{align*}
\|\psi(s)\|_{H^n(\R^2)} &\le
\|\psi_+(s)\|_{H^n(\R^2)} + \|\psi_-(s)\|_{H^n(\R^2)} \lesssim  \ep_1\langle s\rangle^{\delta_0}.
\end{align*}

\subsection{Proof of \eqref{eq:first-moment}}
By Plancherel's theorem, we have
\begin{align*}
\|x\Psi_\theta(t)\|_{L_x^2(\R^2)} \sim \|\nabla \wh{\Psi_\theta}(t)\|_{L^2(\R^2)},	
\end{align*}
and by Duhamel's principle \eqref{eq:duhamel}, 
we can write  
\begin{align*}
	\nabla_{\xi}\widehat{\Psi_\theta}(t,\xi) & =\nabla_{\xi}\widehat{\psi_{0,\theta}}(\xi)+ i \frac{\lam}{2\pi}  \sum_{\theta_1 \in \{\pm\}} \int_0^t \nabla_\xi \mathcal N_{(\theta,\theta_1)}(s,\xi) ds,
\end{align*}
where the nonlinear term is  
\begin{align*}
	\nabla_\xi \mathcal{N}_{(\theta,\theta_1)}(s,\xi) &=
	\mathcal I_{(\theta,\theta_1)}^1(s,\xi)
	+ \mathcal I_{(\theta,\theta_1)}^2(s,\xi),  \\ 
	\mathcal I_{(\theta,\theta_1)}^1(s,\xi)&:=\int_{\R^2} e^{is	p_{(\theta,\theta_1)}(\xi,\eta)}  |\eta|^{-1} \widehat{|\psi|^2}(s,\eta)\nabla_\xi \Big( \Pi_\theta(\xi)\Pi_{\theta_1}(\xi-\eta)\wh{\Psi_\theo}(s,\xi-\eta)  
	\Big)d\eta,\\
	\mathcal I_{(\theta,\theta_1)}^2(s,\xi)&:=  \int_{\R^2} is  e^{isp_{(\theta,\theta_1)}(\xi,\eta)} 
		\mathfrak{A}_{(\theta,\theta_1)}(\xi,\eta)
	 \widehat{|\psi|^2}(s,\eta)\wh{\Psi_\theo}(s,\xi-\eta)  d\eta,
\end{align*}
with the phase function
\begin{align*}
p_{(\theta,\theta_1)}(\xi,\eta)=\theta\langle \xi\rangle - \theta_1\langle \xi-\eta\rangle,
\end{align*}
and the multiplier
\begin{align}
	\mathfrak{A}_{(\theta,\theta_1)}(\xi,\eta)
	&:=\nabla_{\xi}p_{(\theta,\theta_1)}(\xi,\eta)	|\eta|^{-1} \Pi_\theta(\xi)\Pi_{\theta_1}(\xi-\eta) \label{def:A} \\ 
	&=\Big( \frac{\theta \xi}{\bra{\xi}} - \frac{\theta_1(\xi-\eta)}{\bra{\xi-\eta}} \Big)
	|\eta|^{-1} \Pi_\theta(\xi)\Pi_{\theta_1}(\xi-\eta). \nonumber
\end{align}
To prove \eqref{eq:first-moment}, it suffices to show that 
\begin{align*}
\|\mathcal I_{(\theta,\theta_1)}^\ell(s)\|_{L^2(\R^2)}\lesssim \ve_1^3	\langle s\rangle^{-1+\delta_0}, 
\end{align*}
for $\theta,\theta_1\in\{+,-\}$ and $\ell=1,2$.

Since the Dirac projection operator is bounded, under the a priori assumption, by using the time decay estimates \eqref{eq:time-decay}, we estimate 
\begin{align*}
\Big\|\mathcal I_{(\theta,\theta_1)}^1(s)\Big\|_{L^2(\R^2)} 
&\les  \big\| |x|^{-1}\ast |\psi(s)|^2\big\|_{L^\infty(\R^2)}\Big(\|\psi_{\theta_1}(s)\big\|_{L^2(\R^2)}  + \|x\Psi_{\theta_1}(s)\big\|_{L_x^2(\R^2)} \Big) \\ 
	&\les  \| \psi(s)\|_{L^2(\R^2)}\|\psi(s)\|_{L^\infty(\R^2)} \Big(\|\psi_{\theta_1}(s)\big\|_{L^2(\R^2)}  + \|x\Psi_{\theta_1}(s)\big\|_{L_x^2(\R^2)} \Big)  \\ 
	&\les \ve_1^3 \langle s\rangle^{-1+\delta_0}.
\end{align*}
Next, a direct computation based on the null structures \eqref{ineq:Phase null structure} and \eqref{eq:null} shows, for all choices of signs, that\footnote{See, e.g., \cite{CKLY2022} for details.} 
 \begin{align}\begin{aligned}\label{bound of m}
\left\|\left(\frac{|\eta|}{\bra{\eta}} \right)^{\de_0}\mathfrak{A}_{(\theta,\theta_1)}\right\|_{\textup{CM}[(\mathbb{R}^2)^2]}
	=\normo{\iint_{\R^2 \times\R^2} e^{i x \cdot\xi}e^{i y \cdot \eta} \left(\frac{|\eta|}{\bra{\eta}} \right)^{\de_0}\mathfrak{A}_{(\theta,\theta_1)}(\xi,\eta) d\eta d\xi}_{L_{x,y}^1(\R^2\times\R^2)} \les 1.
\end{aligned}\end{align}
Then, applying the multiplier estimates \eqref{eq:coif-1} together with this bound, we obtain
\begin{align*}
\Big\|\mathcal I_{(\theta,\theta_1)}^2(s)\Big\|_{L^2(\R^2)} 
 \les  |s| \|\psi_\theo(s)\|_{L^\infty(\R^2)}\| \bra{D}^{\de_0} |D|^{-\de_0} |\psi(s)|^2 \|_{L^2(\R^2)}  \les \ve_1^3 \bra{s}^{-1 + \de_0}.
\end{align*}

\subsection{Proof of \eqref{eq:second-moment}}
We now turn to the proof of \eqref{eq:second-moment}. By Plancherel's theorem, we have 
$$
\| x^2 \Psi_\theta(t)\|_{L_x^2(\R^2)}\sim \left\|\nabla^2\widehat{\Psi_\theta}(t) \,\right\|_{L^{2}(\R^2)}.
$$
Using the Duhamel formula \eqref{eq:duhamel}, we can express $\nabla_{\xi}^2\widehat{\Psi_\theta}$ as
\begin{align*}
	\nabla_{\xi}^2\widehat{\Psi_\theta}(t,\xi) & =\nabla_{\xi}^2\widehat{\psi_{0,\thez}}(\xi)+ i \frac{\lam}{2\pi}  \sum_{\theta_1 \in \{+,-\}} \int_0^t \nabla_\xi^2 \mathcal N_{(\theta,\theta_1)}(s,\xi) ds,
\end{align*}
where 
\begin{align*}
	\nabla_\xi^2 \mathcal N_{(\theta,\theta_1)}(s,\xi) &=
	\mathcal J_{(\theta,\theta_1)}^1(s,\xi)
	+ \mathcal J_{(\theta,\theta_1)}^2(s,\xi)
	+ \mathcal J_{(\theta,\theta_1)}^3(s,\xi),  \\ 
	\mathcal J_{(\theta,\theta_1)}^1(s,\xi)&= \int_{\R^2} e^{isp_{(\theta,\theta_1)}(\xi,\eta)}  |\eta|^{-1} \widehat{|\psi|^2}(s,\eta)\nabla_\xi^2 \Big( \Pi_\theta(\xi)\Pi_{\theta_1}(\xi-\eta)\wh{\Psi_\theo}(s,\xi-\eta)  
	\Big)d\eta ,\\
	\mathcal J_{(\theta,\theta_1)}^2(s,\xi)&=  \int_{\R^2} is  e^{isp_{(\theta,\theta_1)}(\xi,\eta)}  |\eta|^{-1} \widehat{|\psi|^2}(s,\eta)
	\nabla_{\xi}\Big( \nabla_{\xi}p_{(\theta,\theta_1)}(\xi,\eta) 
	\Pi_\theta(\xi)\Pi_{\theta_1}(\xi-\eta)
		 \wh{\Psi_\theo}(s,\xi-\eta) \Big)  d\eta,  \\ 
	\mathcal J_{(\theta,\theta_1)}^3(s,\xi)&= - \int_{\R^2}s^2 e^{isp_{(\theta,\theta_1)}(\xi,\eta)} |\eta|^{-1} \widehat{|\psi|^2}(s,\eta) 
	\big[\nabla_{\xi}p_{(\theta,\theta_1)}(\xi,\eta) \big]^2
	\Pi_\theta(\xi)\Pi_{\theta_1}(\xi-\eta) \wh{\Psi_\theo}(s,\xi-\eta)  d\eta,
\end{align*}
with the phase function
\begin{align*}
p_{(\theta,\theta_1)}(\xi,\eta)=\theta\langle \xi\rangle - \theta_1\langle \xi-\eta\rangle.
\end{align*}
Thus, to prove \eqref{eq:second-moment}, it remains to establish
\begin{align}\label{GoalofJ3}
\Bigg\| \int_0^t  	\mathcal J_{(\theta,\theta_1)}^\ell(s) ds \Bigg\|_{L^2(\R^2)} \les \langle t\rangle^{2\delta_0}\ep_1^3,
\end{align}
for $\theta,\theta_1\in\{+,-\}$ and $\ell=1,2,3$.

\medskip

We begin with the cases $\ell=1,2$ and prove the stronger pointwise bounds
\begin{align*}
\|\mathcal J_{(\theta,\theta_1)}^\ell(s)\|_{L^2(\R^2)}\lesssim \ve_1^3	\langle s\rangle^{-1+ 2\delta_0}.
\end{align*}
Similarly to the estimates for $	\mathcal I_{(\theta,\theta_1)}^1$, one can verify that
\begin{align*}
\|&\mathcal J_{(\theta,\theta_1)}^1(s)\|_{L^2(\R^2)} 	\\
&\les  \big\| |x|^{-1}\ast |\psi(s)|^2\big\|_{L^\infty(\R^2)}\Big(\|\psi_{\theta_1}(s)\big\|_{L^2(\R^2)}  + \|x\Psi_{\theta_1}(s)\big\|_{L_x^2(\R^2)} +  \|x^2\Psi_{\theta_1}(s)\big\|_{L_x^2(\R^2)} \Big)  \\ 
	&\les \| \psi(s)\|_{L^2(\R^2)}\|\psi(s)\|_{L^\infty(\R^2)} \Big(\|\psi_{\theta_1}(s)\big\|_{L^2(\R^2)}  + \|x\Psi_{\theta_1}(s)\big\|_{L_x^2(\R^2)} +  \|x^2\Psi_{\theta_1}(s)\big\|_{L_x^2(\R^2)} \Big)  \\ 
	&\les \ve_1^3 \langle s\rangle^{-1+2\delta_0}.
\end{align*}
Next, in analogy with the estimates for $\mathcal I_{(\theta,\theta_1)}^2$, one can show that  by Bernstein's inequality,
\begin{align*}
	\|\mathcal J_{(\theta,\theta_1)}^2(s)\|_{L^2(\R^2)} 
	&\les   |s| \big(\|\psi_\theo(s)\|_{L^2(\R^2)}+\|x\Psi_\theo(s)\|_{L_x^2(\R^2)}\big)\|\bra{D}^{\zeta}|D|^{-\zeta}|\psi(s) |^2\|_{L^\infty(\R^2)}  \\ 
	&\les  \bra{s}^{-1+2\delta_0} \ve_1^3,
   \end{align*}
for $\zeta:= \frac{\de_0}2$. Here we applied the multiplier estimates \eqref{eq:coif-1}, using \eqref{bound of m} and the following bound 
\begin{align*}
	&\left\|\left(\frac{|\eta|}{\bra{\eta}} \right)^{\zeta}\nabla_\xi \mathfrak{A}_{(\theta,\theta_1)} \right\|_{\textup{CM}[(\R^2)^2]}
	=\normo{\iint_{\mathbb{R}^2\times\mathbb{R}^2} e^{i x \cdot\xi}e^{i y \cdot \eta} \left(\frac{|\eta|}{\bra{\eta}}\right)^\zeta\nabla_{\xi}\mathfrak{A}_{(\theta,\theta_1)}(\xi,\eta) d\eta d\xi}_{L_{x,y}^1(\R^2 \times\R^2)} \les 1,
\end{align*}
where the multiplier $\mathfrak{A}_{(\theta,\theta_1)}$ is defined in \eqref{def:A}.

\medskip

It remains to consider $\mathcal J_{(\theta,\theta_1)}^3$.
Let $\theta_2,\theta_3\in\{+,-\}$ and $\mathbf{\Theta}:=(\theta,\theta_1,\theta_2,\theta_3)$.
We decompose  
\begin{align}\begin{aligned}\label{J3}
	\mathcal J_{(\theta,\theta_1)}^3(s,\xi)& =\sum_{\theta_2,\theta_3\in\{+,-\}}  
	\mathcal J_{\mathbf{\Theta}}^3(s,\xi) \\ 
	\mathcal J_{\mathbf{\Theta}}^3(s,\xi)&= -
	\int_{\R^2}s^2 e^{isp_{(\theta,\theta_1)}(\xi,\eta)}  
	\mathfrak{M}_{(\theta,\theta_1)}(\xi,\eta)
	\widehat{\langle \psi_{\theta_2},\psi_{\theta_3}\rangle }(s,\eta)\wh{\Psi_\theo}(s,\xi-\eta)  d\eta,
\end{aligned}\end{align}
where
\begin{align}
p_{(\theta,\theta_1)}(\xi,\eta)&=\theta\langle \xi\rangle - \theta_1\langle \xi-\eta\rangle, \nonumber \\ 
\mathfrak{M}_{(\theta,\theta_1)}(\xi,\eta)&:=|\eta|^{-1}\big(\nabla_{\xi}p_{(\theta,\theta_1)}(\xi,\eta) \big)^2 \Pi_{\theta}(\xi)\Pi_{\theta_1}(\xi-\eta). \label{def:M}
\end{align}
We observe that, regardless of the sign configuration, the multiplier is bounded 
\begin{equation}\label{boundedness of M}
|\mathfrak{M}_{(\theta,\theta_1)}(\xi,\eta)| \lesssim 1,
\end{equation}
thanks to the phase null structure \eqref{ineq:Phase null structure} when $\theta=\theta_1$ and the Dirac null structure \eqref{eq:null} when $\theta\neq\theta_1$. Therefore, the apparent singularity $|\eta|^{-1}$ is canceled by the null structures. The remaining task is to recover the time loss and establish the desired time decay.

We divide the analysis of $\mathcal J_{\mathbf{\Theta}}^3$ into three cases according to the sign configuration:
\begin{align}\label{Classification of signs}
 \left\{ \begin{aligned}
\textbf{ Case~1: } &\mathbf{\Theta}\in \boldsymbol{\mathfrak{S}}_1 :=\big\{(\theta,\theta,\theta',\theta')\;|\;\theta,\theta'\in\{+,-\}\big\}, \\ 
\textbf{ Case~2: } &\mathbf{\Theta}\in \boldsymbol{\mathfrak{S}}_2:= \big\{ (\theta,-\theta,\theta,-\theta) , (\theta,-\theta,\theta',\theta'), (\theta,\theta,\theta',-\theta') \;|\;\theta,\theta'\in\{+,-\}\big\}, \\ 
\textbf{ Case~3: } &\mathbf{\Theta}\in \boldsymbol{\mathfrak{S}}_3:=\big\{(\theta,-\theta,-\theta,\theta)\;|\;\theta\in\{+,-\}\big\}.
 \end{aligned}\right.
\end{align}
For Case~1, we rely heavily on the phase null structure (see \eqref{phase null structure}).
For Cases~2 and~3, we exploit the spacetime resonance structure. More precisely, in Case~2 we use the time resonance structure, whereas in Case~3 we rely on the space resonance structure.

To establish \eqref{GoalofJ3}, we will in most cases prove the stronger pointwise estimate
\begin{align*}
\|\mathcal J_{\mathbf{\Theta}}^3(s)\|_{L^2(\R^2)}
\lesssim
\langle s\rangle^{-1+2\delta_0}\varepsilon_1^3,
\end{align*}
for all $\mathbf{\Theta}=(\theta,\theta_1,\theta_2,\theta_3)$.
This immediately yields \eqref{GoalofJ3}. The only exception occurs in Case~2, where we estimate the time integral in \eqref{GoalofJ3} directly by exploiting the time non-resonance structure.

\subsubsection{Estimates for Case 1.}
Let $\mathbf{\Theta}_1= (\theta,\theta,\theta_2,\theta_2) \in \boldsymbol{\mathfrak{S}}_1$ be fixed. We prove that 
\begin{align*}
\|\mathcal J_{\mathbf{\Theta}_1}^3(s)\|_{L^2(\R^2)}\lesssim \ve_1^3	\langle s\rangle^{-1+2\delta_0}.
\end{align*}
 In this case, the key observation is that the phase always exhibits the null structure:
\begin{align}\label{phase null structure}
	\big(\nabla_{\xi}p_{(\theta,\theta)}(\xi,\eta)\big)^2
	=\Big( \frac{ \xi}{\bra{\xi}} - \frac{\xi-\eta}{\bra{\xi-\eta}} \Big)^2.
\end{align}
Consequently, $\mathcal J_{\mathbf{\Theta}_1}^3$ is written as
\begin{align*}
\mathcal J_{\mathbf{\Theta}_1}^3(s,\xi)&=-
	\int_{\R^2}s^2 e^{\theta is( \langle\xi\rangle-\langle\xi-\eta\rangle  )} |\eta|^{-1} 
	\Big( \frac{ \xi}{\bra{\xi}} - \frac{\xi-\eta}{\bra{\xi-\eta}} \Big)^2 
	\Pi_{\theta}(\xi)\Pi_{\theta}(\xi-\eta) \\
	&\hspace{8cm} \times\widehat{| \psi_{\theta_2}|^2 }(s,\eta) \wh{\Psi_{\theta}}(s,\xi-\eta)  d\eta.
\end{align*}
Hence, this particular case can be estimated using a methodology similar to that presented in \cite[Section 3]{kly2023}. For the convenience of readers, we include the proof, although those already acquainted with the details may skip the estimates for this case.

We recall the following time decay estimates for $\C^2$-valued functions (spinor) to \eqref{inteq0} under the a priori assumption \eqref{assumption-apriori}. Most of these inequalities follow directly from the corresponding scalar-valued estimates proved in \cite{kly2023}, with only minor modifications. For this reason, we simply recall the results here without providing detailed proofs.

\begin{lemma}[Lemma 3.3 of \cite{kly2023}]\label{Lem:fre localized bounds}
Let $\psi$ satisfy the a priori assumption \eqref{assumption-apriori} for some $\ve_{1}>0$. Then for a dyadic number $K\in 2^{\Z}$, 
\begin{align*}
\big\| \dot{P}_{K} \psi_{\theta}(s)\big\|_{L^\infty(\R^2)} &\lesssim \min(K,\langle K\rangle^{-k}\langle s\rangle^{-1} )\ve_1, \\ 
\big\| \dot{P}_{K} \psi_{\theta}(s)\big\|_{L^2(\R^2)} &\lesssim \min(K^{\frac12}\langle K\rangle^{-k}, \langle K\rangle^{-n})\ve_1,
\end{align*}
where $\psi_{\theta}=\Pi_{\theta}(D)\psi$ for $\theta\in\{+,-\}$.
\end{lemma}

\begin{lemma}[Lemma~3.4--3.5 of \cite{kly2023}]
Let $\psi$ satisfy the a priori assumption \eqref{assumption-apriori} for some $\ve_{1}>0$. Then for a dyadic number $K\in 2^{\Z}$, 
	\begin{align}
		\left\|\dot{P}_{K}\Big(|\psi_\theta(s)|^2\Big)\right\|_{L^{\infty}(\R^2)} & \les\min\big(\langle K\rangle^{-k}\langle s\rangle^{-2},K^{2}\big)\ve_{1}^{2},\nonumber \\
	\left\|\dot{P}_{K}\Big(|\psi_\theta(s)|^2\Big)\right\|_{L^{2}(\R^2)}\, & \les \min\big( K\langle K\rangle^{-\frac k2}, K^{-1}\bra{K}^{-1}\bra{s}^{-2 + \frac32 \de_0} \big)\ve_{1}^{2}, \label{eq:norm-two} 
	\end{align}
where $\psi_{\theta}=\Pi_{\theta}(D)\psi$ for $\theta\in\{+,-\}$.
\end{lemma}

We begin with a homogeneous dyadic decomposition to write 
\begin{align*}
\mathcal J_{\mathbf{\Theta}_1}^3(s,\xi)&=
	\sum_{(K_{0},K_{1},K_2)\in (2^{\Z})^3}	\mathcal J_{\mathbf{\Theta}_1;(K_{0},K_{1},K_2)}^3(s,\xi),\\
	\mathcal J_{\mathbf{\Theta}_1;(K_{0},K_{1},K_2)}^3(s,\xi) &:=-s^2\int_{\mathbb{R}^{2}}\mathfrak{M}_{(\theta,\theta)}(\xi,\eta)e^{isp_{(\theta,\theta)}(\xi,\eta)}\widehat{\dot{P}_{K_1}\Psi_{\theta}}(\xi-\eta)\widehat{\dot{P}_{K_2}(|\psi_\thet|^{2})}(\eta)\,d\eta,
\end{align*}
where we recall from \eqref{def:M} that 
\begin{align*}
\mathfrak{M}_{(\theta,\theta)}(\xi,\eta)
	:=\Big( \frac{ \xi}{\bra{\xi}} - \frac{\xi-\eta}{\bra{\xi-\eta}} \Big)^2	|\eta|^{-1} \Pi_{\theta}(\xi)\Pi_{\theta}(\xi-\eta).
\end{align*}

In contrast to the bound \eqref{boundedness of M}, for $\mathbf{\Theta}_1\in \boldsymbol{\mathfrak{S}}_1$, the phase null structure arising from $\big( \frac{ \xi}{\bra{\xi}} - \frac{\xi-\eta}{\bra{\xi-\eta}} \big)\big|_{\eta=0}=0$ yields an additional gain of $K_2$ for the frequency-localized multiplier. More precisely, one readily verifies that
 \begin{align}
\sup_{\xi,\eta\in\R^2}\left|\mathfrak{M}_{(\theta,\theta)}(\xi,\eta)\chi_{(K_0,K_1,K_2)}(\xi,\eta)\right|&\les K_2\langle K_0\rangle^2\max(\langle K_0\rangle, \langle K_1\rangle)^{-2}, \label{pointwise bound of m2} \\ 
    \left\|\mathfrak{M}_{(\theta,\theta)}\chi_{(K_0,K_1,K_2)}\right\|_{\cm[(\R^2)^2]}  &\les  K_2\langle K_0\rangle^2\max(\langle K_0\rangle, \langle K_1\rangle)^{-2}, \label{multiplier norm of m2}
\end{align} 
where $\chi_{(K_0,K_1,K_2)}(\xi,\eta):=\chi_{K_{0}}(\xi)\chi_{K_2}(\eta)\chi_{K_{1}}(\xi-\eta)$.

Using \eqref{pointwise bound of m2} and the first part of \eqref{eq:norm-two}, we estimate the sum over those indices $K_0$ such that $K_0\le \langle s\rangle^{-3}$ 
\begin{align*}
&\sum_{\substack{ (K_0,K_1,K_2)\in (2^{\Z})^3\\ K_0 \le \langle s\rangle^{-3}}}\|\mathcal J_{\mathbf{\Theta}_1;(K_0,K_1,K_2)}^3(s)\|_{L^2(\R^2)} \\
&\les |s|^2\sum_{\substack{ (K_0,K_1,K_2)\in (2^{\Z})^3\\ K_0 \le \langle s\rangle^{-3}}}\|\chi_{K_{0}}\|_{L^{2}(\R^2)}\normo{\mathfrak{M}_{(\theta,\theta)}\chi_{(K_0,K_1,K_2)}}_{L^{\infty}((\R^2)^2 )}\normo{\dot{P}_{K_2}|\psi_\thet(s)|^{2}}_{L^{2}(\R^2)}\normo{\dot{P}_{K_{1}}{\Psi_\theta}(s)\,}_{L^{2}(\R^2)} \\ 
&\les   |s|^2   \sum_{\substack{ (K_0,K_1,K_2)\in (2^{\Z})^3\\ K_0 \le \langle s\rangle^{-3}}}K_0K_2\langle K_0\rangle^2\max(\langle K_0\rangle, \langle K_1\rangle)^{-2}K_2\langle K_2\rangle^{-\frac{k}{2}}K_1\langle K_1\rangle^{-k}\ep_1^3\\ 
	&\les \langle s\rangle^{-1}\ep_1^3.
\end{align*}
On the other hand, by the multiplier inequalities \eqref{eq:coif-1} with the bound \eqref{multiplier norm of m2}, one has 
\begin{align*}
	&\sum_{\substack{ (K_0,K_1,K_2)\in (2^{\Z})^3 \\ K_0 \ge \langle s\rangle^{-3}}}\|\mathcal J_{\mathbf{\Theta}_1;(K_0,K_1,K_2)}^3(s)\|_{L^2(\R^2)} \\
	&	\les   |s|^2   \sum_{\substack{ (K_0,K_1,K_2)\in (2^{\Z})^3 \\ K_0 \ge \langle s\rangle^{-3}}}\left\|\mathfrak{M}_{(\theta,\theta)}\chi_{(K_0,K_1,K_2)}\right\|_{\cm[(\R^2)^2]} \normo{\dot{P}_{K_2}|\psi_\thet(s)|^{2}}_{L^{2}(\R^2)}\normo{\dot{P}_{K_1}{\psi_{\theta}}(s)\,}_{L^{\infty}(\R^2)} \\ 
	&	\les   |s|^2   \sum_{\substack{ (K_0,K_1,K_2)\in (2^{\Z})^3 \\ K_0 \ge \langle s\rangle^{-3}}}K_2\langle K_0\rangle^2\max(\langle K_0\rangle, \langle K_1\rangle)^{-2} \normo{\dot{P}_{K_2}|\psi_\thet(s)|^{2}}_{L^{2}(\R^2)}\normo{\dot{P}_{K_1}{\psi_{\theta}}(s)\,}_{L^{\infty}(\R^2)} \\ 
	&\lesssim  \langle s\rangle^{-1+2\delta_0}\ep_1^3,
\end{align*}
where we used \eqref{eq:norm-two} in the last inequality by splitting the summation over $K_2$ into the cases $K_2 \le \langle s\rangle^{-1+\frac34\delta_0}$ and $K_2 \ge \langle s\rangle^{-1+\frac34\delta_0}$.

\subsubsection{Estimates for Case 2.}
Here, we exploit the time non-resonance structure arising from the interaction of the cubic nonlinearity. Substituting the following representation 
\begin{align*}
	\langle \psi_{\theta_2},\psi_{\theta_3}\rangle(s,\eta) = \frac1{(2\pi)^2} \int_{\R^2}e^{is(\theta_2\langle \sigma\rangle -\theta_3 \langle \eta+\sigma\rangle)}
	\langle \Psi_{\theta_2}(s,\sigma),\Psi_{\theta_3}(s,\eta+\sigma)\rangle d\sigma
\end{align*}
into \eqref{J3} and performing the change of variables  $\sigma \mapsto \xi-\eta+\sigma$ and $\eta \mapsto -\eta$, we obtain 
\begin{align}\begin{aligned}\label{eq:expanded duhamel formula}
\mathcal J_{\mathbf{\Theta}}^3(s,\xi)&:=- \iint_{\mathbb{R}^2\times\mathbb{R}^2}  s^2 e^{is q_{\mathbf{\Theta}}(\xi,\eta,\sigma)}
\mathfrak{M}_{(\theta,\theta_1)}(\xi,-\eta) \\
&\hspace{4cm} \times\widehat{\Psi_{\theta_1}}(s,\xi+\eta)
\big\langle \widehat{\Psi_{\theta_2}}(s,\xi+\eta+\sigma),
\widehat{\Psi_{\theta_3}}(s,\xi+\sigma) \big\rangle d\sigma d\eta,
\end{aligned}\end{align}
where $\mathbf{\Theta}=(\theta,\theta_1,\theta_2,\theta_3)$. The associated phase function is given by 
\begin{align}\label{phase interaction q}
	q_{\mathbf{\Theta}}(\xi,\eta,\sigma)&=\theta \langle\xi\rangle - \theo\langle\xi+\eta\rangle  +\thet \bra{\xi+\eta+\sigma} - \theth \bra{\xi+\sigma},
\end{align}
and the multiplier $\mathfrak{M}_{(\theta,\theta_1)}$, introduced in \eqref{def:M}, is given by
\begin{align*}
\mathfrak{M}_{(\theta,\theta_1)}(\xi,-\eta)=    |\eta|^{-1}\big(\nabla_{\xi}p_{(\theta,\theta_1)}(\xi,-\eta) \big)^2 \Pi_{\theta}(\xi)\Pi_{\theta_1}(\xi+\eta).
\end{align*}

Recall from \eqref{Classification of signs} that 
$$\boldsymbol{\mathfrak{S}}_2:= \{ (\theta,-\theta,\theta,-\theta) , (\theta,-\theta,\theta',\theta'), (\theta,\theta,\theta',-\theta'), \;|\;\theta,\theta'\in\{\pm\}\}$$
For these sign configurations, at least three terms in the phase function have the same sign. Indeed, the phase function is given by 
\begin{align*}
q_{\mathbf{\Theta}}(\xi,\eta,\sigma)
= \left\{ \begin{aligned} 
\theta \left( \langle\xi\rangle +\langle\xi+\eta\rangle  +\bra{\xi+\eta+\sigma} + \bra{\xi+\sigma}\right)
&  \quad \text{ for } \mathbf{\Theta}\in\{(\theta,-\theta,\theta,-\theta) \,|\,\theta\in \{\pm\}\}, \\ 
\theta \langle\xi\rangle +\theta\langle\xi+\eta\rangle  +\theta' \bra{\xi+\eta+\sigma} - \theta' \bra{\xi+\sigma}
& \quad \text{ for } \mathbf{\Theta}\in \{(\theta,-\theta,\theta',\theta') \,|\,\theta,\theta'\in \{\pm\}\}, \\ 
\theta \langle\xi\rangle - \theta\langle\xi+\eta\rangle  +\theta' \bra{\xi+\eta+\sigma} + \theta' \bra{\xi+\sigma}
& \quad \text{ for } \mathbf{\Theta}\in \{(\theta,\theta,\theta',-\theta') \,|\,\theta,\theta'\in \{\pm\}\},
\end{aligned} \right.
\end{align*}
and since, among the four terms in the phase function, any one of them can be written as the sum of the remaining three, one can verify that
\begin{align}\label{time resonance}
	|q_{\mathbf{\Theta}}(\xi,\eta,\sigma)| \gtrsim 
	\left\{ \begin{aligned}
		\max(\langle \xi\rangle, \langle \xi+\eta\rangle, \langle \xi+\sigma\rangle, \langle \xi+\eta+\sigma\rangle) &  \quad \text{ for } \mathbf{\Theta}\in\{(\theta,-\theta,\theta,-\theta) \,|\,\theta\in \{\pm\}\}\\ 
	\frac{1}{\max(\langle \xi\rangle, \langle \xi+\eta\rangle, \langle \xi+\sigma\rangle, \langle \xi+\eta+\sigma\rangle)} & \quad \text{ for } \mathbf{\Theta}\in \{(\theta,-\theta,\theta',\theta')\,,\,(\theta,\theta,\theta',-\theta') \,|\,\theta,\theta'\in \{\pm\}\} \end{aligned} \right.
\end{align}
which allows us to exploit the time non-resonant structure.

We begin by performing an inhomogeneous dyadic decomposition to write
\begin{align}\begin{aligned}\label{JthetaLN}
	\mathcal J_{\mathbf{\Theta}}^3(s,\xi)
	&=\sum_{\mathbf{N}:=(N_0,N_1,N_2,N_3)\in (2^{\mathbb{N}\cup\{0\}})^4} \mathcal J_{\mathbf{\Theta};\mathbf{N}}^3(s,\xi) \\ 
	 \mathcal J_{\mathbf{\Theta};\mathbf{N}}^3(s,\xi) &=-  s^2\iint_{\mathbb{R}^{2+2}}e^{isq_{\mathbf{\Theta}}(\xi,\eta,\sigma)} 	\mathfrak{M}_{(\theta,\theta_1)}(\xi,-\eta) \rho_{N_0}(\xi) \\
	 & \qquad \times 
     \widehat{\Psi_{\theo;N_1}}(s,\xi+\eta) \bra{\wh{\Psi_{\thet;N_2}}(s,\xi+\eta+\sigma),\wh{\Psi_{\theth;N_3}}(s,\xi+\sigma)} \,d\sigma d\eta ,
\end{aligned}\end{align}
where $\mathbf{N}:=(N_0,N_1,N_2,N_3)\in (2^{\mathbb{N}\cup\{0\}})^4$ and $\Psi_{\theta_i;N_i}:=P_{N_i}(D)\Pi_{\theta_i}(D)\Psi$ for $i=1,2,3$. Similarly, we write $\psi_{\theta_i;N_i}:=P_{N_i}(D)\Pi_{\theta_i}(D)\psi$. 

We fix $\mathbf{\Theta}_2=(\theta,\theta_1,\theta_2,\theta_3)\in\boldsymbol{\mathfrak{S}}_2$ and prove that
\begin{align*}
\Bigg\| \int_0^t  \sum_{ \mathbf{N}\in (2^{\N\cup\{0\}})^4 } \mathcal J_{\mathbf{\Theta}_2;\mathbf{N}}^3(s) ds \Bigg\|_{L^2(\R^2)} \les \langle t\rangle^{2\delta_0}\ep_1^3.
\end{align*}
Using the following identity
\begin{align*}
	e^{isq_{\mathbf{\Theta}_2}} = -i\frac{\partial_s e^{isq_{\mathbf{\Theta}_2}}}{q_{\mathbf{\Theta}_2}},
\end{align*}
we obtain, by integration by parts in time, that 
\begin{align}\begin{aligned}
\int_0^t \mathcal J_{\mathbf{\Theta}_2;\mathbf{N}}^3(s,\xi) ds & = 
\mathcal J_{\mathbf{\Theta}_2;\mathbf{N}}^{3,1}(s,\xi)\big|_{s=0}^{s=t} + \int_0^t\mathcal J_{\mathbf{\Theta}_2;\mathbf{N}}^{3,2}(s,\xi)ds + \int_0^t\mathcal J_{\mathbf{\Theta}_2;\mathbf{N}}^{3,3}(s,\xi)ds, \\ 
\mathcal J_{\mathbf{\Theta}_2;\mathbf{N}}^{3,1}(s,\xi)&=is^2\iint_{\mathbb{R}^{2}\times \R^2} \frac{\mathfrak{M}_{(\theta,\theta_1)}(\xi,-\eta)  }{q_{\mathbf{\Theta}_2}(\xi,\eta,\sigma)}\rho_{N_0}(\xi) e^{isq_{\mathbf{\Theta}_2}(\xi,\eta,\sigma)}  \\ 
&\times  \widehat{\Psi_{\theo;N_1}}(s,\xi+\eta)  \bra{\wh{\Psi_{\thet;N_2}}(s,\xi+\eta+\sigma),\wh{\Psi_{\theth;N_3}}(s,\xi+\sigma)}  d\sigma d\eta,	\\
\mathcal J_{\mathbf{\Theta}_2;\mathbf{N}}^{3,2}(s,\xi)=&-2i s \iint_{\mathbb{R}^{2}\times \R^2}\frac{\mathfrak{M}_{(\theta,\theta_1)}(\xi,-\eta)  }{q_{\mathbf{\Theta}_2}(\xi,\eta,\sigma)}\rho_{N_0}(\xi)
e^{isq_{\mathbf{\Theta}_2}(\xi,\eta,\sigma)} \nonumber \\ 
&\times   \widehat{\Psi_{\theo;N_1}}(s,\xi+\eta)  \bra{\wh{\Psi_{\thet;N_2}}(s,\xi+\eta+\sigma),\wh{\Psi_{\theth;N_3}}(s,\xi+\sigma)}  d\sigma d\eta ,   \\
\mathcal J_{\mathbf{\Theta}_2;\mathbf{N}}^{3,3}(s,\xi)=&-i s^2 \iint_{\mathbb{R}^{2}\times \R^2} \frac{\mathfrak{M}_{(\theta,\theta_1)}(\xi,-\eta)  }{q_{\mathbf{\Theta}_2}(\xi,\eta,\sigma)}\rho_{N_0}(\xi)
e^{isq_{\mathbf{\Theta}_2}(\xi,\eta,\sigma)}   \\ 
&\times \partial_s \left[\widehat{\Psi_{\theo;N_1}}(s,\xi+\eta)  \bra{\wh{\Psi_{\thet;N_2}}(s,\xi+\eta+\sigma),\wh{\Psi_{\theth;N_3}}(s,\xi+\sigma)}  \right] d\sigma d\eta.
\end{aligned}\end{align}
We note that the null structures used in
\eqref{boundedness of M} yield the refined estimate for the frequency localized multiplier:
\begin{align*}
\sup_{\xi,\eta,\sigma\in \R^2} \left| \mathfrak{M}_{(\theta,\theta_1)}(\xi,-\eta)   \rho_{\mathbf{N}}(\xi,\eta,\sigma)   \right| \les  \max( N_0 ,  N_1 )^{-1}
\end{align*}
where
\begin{align}\label{dyadicdecomposition}
	\rho_{\mathbf{N}}(\xi,\eta,\sigma):=\rho_{N_0}(\xi)\rho_{N_1}(\xi+\eta)\rho_{N_2}(\xi+\eta+\sigma)\rho_{N_3}(\xi+\sigma).
\end{align}
Together with the lower bound for $q_{\mathbf{\Theta}_2}$ in \eqref{time resonance}, this implies that 
\begin{align*}
  \left\| \frac{\mathfrak{M}_{(\theta,\theta_1)}\rho_{\mathbf{N}}}{q_{\mathbf{\Theta}_2}} \right\|_{\text{CM}[(\R^{2})^3]} 
&:= \left\Vert \int\!\!\!\!\int\!\!\!\!\int_{(\R^2)^3}\frac{\mathfrak{M}_{(\theta,\theta_1)}(\xi,-\eta)}{q_{\mathbf{\Theta}_2}(\xi,\eta,\sigma)}\rho_{\mathbf{N}}(\xi,\eta,\sigma)e^{ix\cdot\xi}e^{iy\cdot\eta}e^{iz\cdot\sigma} \,d\sigma d\eta  d\xi\right\Vert _{L_{x,y,z}^{1}((\R^2)^3)}
\\ & \les   N_{\max}^{\,10},
\end{align*}
where $N_{\max}$ denotes $\max\{N_i:0\le i\le3\}$.

On the support of the multiplier, the frequency relation implies
\[
N_0\lesssim \max(N_1,N_2,N_3).
\]
Thus, it suffices to consider the three frequency regions
\[
N_j\sim N_{\max},\qquad j=1,2,3.
\]
By symmetry, we only estimate the contribution corresponding to the frequency region \(N_1\sim N_{\max}\);
the other two regions are estimated in the same way.

We recall the following inhomogeneous frequency-localized estimates, which follow directly from Lemma~\ref{Lem:fre localized bounds} and will be used repeatedly throughout the paper.
For $\psi$ satisfying the a priori assumption
\eqref{assumption-apriori} for some $\varepsilon_1>0$, we have
\begin{align}\begin{aligned}\label{Frequency localized inequalities:Inhomogeneous}
\| P_N \psi_{\theta}(s)\|_{L^\infty(\R^2)} &\lesssim  N^{-k}\langle s\rangle^{-1} \ve_1, \\ 
\| P_N \psi_{\theta}(s)\|_{L^2(\R^2)} &\lesssim  N^{-n}\langle s\rangle^{\delta_0}\ve_1,
\end{aligned}\end{align}
for all dyadic numbers $N\in2^{\mathbb N\cup\{0\}}$ and $\theta\in\{+,-\}$.

Applying \eqref{eq:coif-2} together with the above multiplier bound and the frequency-localized estimates in \eqref{Frequency localized inequalities:Inhomogeneous}, we obtain, for $\mathbf{\Theta}_2\in \boldsymbol{\mathfrak{S}}_2$,
\begin{align*}
&\sum_{  \mathbf{N}\in (2^{\mathbb{N}\cup\{0\}})^4 \,:\, N_{\max}\sim N_1   }   \|\mathcal J_{\mathbf{\Theta}_2;\mathbf{N}}^{3,1}(s)\|_{L^2(\R^2)} \\
&\les \sum_{   \mathbf{N}\in (2^{\mathbb{N}\cup\{0\}})^4 \,:\, N_{\max}\sim N_1  }   s^2 N_{1}^{10}\|\psi_{\theta_1;N_1}(s)\|_{L^2(\R^2)} \| \psi_{\thet;N_2}(s)\|_{L^\infty(\R^2)} \|\psi_{\theth;N_3}(s)\|_{L^\infty(\R^2)} \\
	&\les 	\langle s\rangle^{\delta_0}\sum_{   \mathbf{N}\in (2^{\mathbb{N}\cup\{0\}})^4 \,:\, N_{\max}\sim N_1   }   N_1^{10-n}N_2^{-k}N_3^{-k}   \ve_1^3 \les \bra{s}^{2\de_0}\ve_1^3.
\end{align*}
The same argument yields
\begin{align*}
\int_0^t\sum_{   \mathbf{N}\in (2^{\mathbb{N}\cup\{0\}})^4 \,:\, N_{\max}\sim N_1   } \|\mathcal J_{\mathbf{\Theta}_2;\mathbf{N}}^{3,2}(s)\|_{L^2(\R^2)} ds \les 	\int_0^t \bra{s}^{-1+2\de_0} \ve_1^3   ds \les \bra{t}^{2\de_0}\ve_1^3.
\end{align*}
For $\mathcal J_{\mathbf{\Theta}_2;\mathbf{N}}^{3,3}$,
we only consider the contribution in which the time derivative falls on
$\widehat{\Psi_{\theta_2;N_2}}$,
and denote the corresponding term by
$\widetilde{\mathcal J}_{\mathbf{\Theta}_2;\mathbf{N}}^{3,3}$.
The remaining contributions are estimated in exactly the same way.
Differentiating the Duhamel formula \eqref{eq in terms of Psi} with respect to time,
and using \eqref{ineq:Hartree term} together with \eqref{Frequency localized inequalities:Inhomogeneous},
we obtain for $N\in2^{\mathbb{N}\cup\{0\}}$
\begin{align} 
\|\partial_s P_N\Psi_{\theta}(s)\|_{L^2(\R^2)} \les N^{-n} \bra{s}^{-1+\delta_0} \ve_1^3. \label{eq:esti-timederi-f}
\end{align}
Using this and arguing as before, we deduce
\begin{align*}
&\int_0^t \sum_{   \mathbf{N}\in (2^{\mathbb{N}\cup\{0\}})^4 \,:\, N_{\max} \sim N_1   }\|\widetilde{\mathcal J}_{\mathbf{\Theta}_2;\mathbf{N}}^{3,3}(s)\|_{L^2(\R^2)} ds \\
&\les \int_0^t \sum_{   \mathbf{N}\in (2^{\mathbb{N}\cup\{0\}})^4 \,:\, N_{\max}\sim N_1   } s^2 N_1^{10}  \| \psi_{\theta_1;N_1}(s)\|_{L^\infty(\R^2)} \|\partial_s \Psi_{\theta_2;N_2}(s)\|_{L^2(\R^2)} \|\psi_{\theth;N_3}(s)\|_{L^\infty(\R^2)}ds \\
&\les  \int_0^t \bra{s}^{-1+\de_0} \sum_{  \mathbf{N}\in (2^{\mathbb{N}\cup\{0\}})^4 \,:\, N_{\max}\sim N_1  } N_1^{10-k}  N_2^{-n}N_3^{-k}   \ve_1^3 ds \les \bra{t}^{2\de_0}\ve_1^3.
\end{align*}

\subsubsection{Estimates for Case 3.} 
Let $\mathbf{\Theta}_3=(\theta,-\theta,-\theta,\theta)\in \boldsymbol{\mathfrak{S}}_3$ be fixed. We employ the dyadic decomposition introduced in \eqref{JthetaLN} and aim to prove the pointwise estimate
\begin{align*}
\sum_{\mathbf{N}:=(N_0,N_1,N_2,N_3)\in (2^{\N\cup\{0\}})^4}\|\mathcal J_{\mathbf{\Theta}_3;\mathbf{N}}^3(s)\|_{L^2(\R^2)}\lesssim \ve_1^3	\langle s\rangle^{-1+2\delta_0}.
\end{align*}

In this sign configuration, the phase function is given by
\begin{align*}
q_{\theta}(\xi,\eta,\sigma):=q_{\mathbf{\Theta}_3}(\xi,\eta,\sigma)  = \theta \big(\langle\xi\rangle + \langle\xi+\eta\rangle - \bra{\xi+\eta+\sigma} -  \bra{\xi+\sigma}\big)
\end{align*}
and we also have
\begin{align*}
	\big(\nabla_{\xi}p_{(\theta,-\theta)}(\xi,-\eta)\big)^2=\left( \frac{\xi}{\bra{\xi}} + \frac{\xi+\eta}{\bra{\xi+\eta}}\right)^2,
\end{align*}
for which neither the time resonance nor the phase null structure can be exploited. The key observation is the following algebraic identity: 
\begin{align*}
	\big(\nabla_{\xi}p_{(\theta_0,\theta_1)}(\xi,-\eta)\big)^2 
	= \big( \nabla_\xi q_{\mathbf{\Theta}}(\xi,\eta,\sigma) - \nabla_\sigma q_{\mathbf{\Theta}}(\xi,\eta,\sigma) \big)^2 
\text{ for any } \mathbf{\Theta}=(\theta_0,\theta_1,\theta_2,\theta_3).
\end{align*}
In particular, for $\mathbf{\Theta}_3\in \boldsymbol{\mathfrak{S}}_3$, we have 
\begin{align}\label{Key indentity2}
\big(\nabla_{\xi}p_{(\theta,-\theta)}(\xi,-\eta)\big)^2 
= \big(\nabla_\xi q_{\theta}(\xi,\eta,\sigma) \big)^2 
+ \nabla_\sigma q_{\theta}(\xi,\eta,\sigma) \big( \nabla_\sigma q_{\theta}(\xi,\eta,\sigma) - 2\nabla_\xi q_{\theta}(\xi,\eta,\sigma) \big).
\end{align}
Substituting this identity into \eqref{JthetaLN}, we obtain 
\begin{align}\begin{aligned}\label{J3decomposition}
\mathcal J_{\mathbf{\Theta}_3;\mathbf{N}}^3(s,\xi)
&=	\mathbb{J} _{\theta;\mathbf{N}}(s,\xi)
	+ 	\widetilde{\mathbb{J} }_{\theta;\mathbf{N}}(s,\xi),\\ 
	\mathbb{J} _{\theta;\mathbf{N}}(s,\xi)&=- s^2 \iint_{\mathbb{R}^2\times\mathbb{R}^2}  e^{is q_{\theta}(\xi,\eta,\sigma)}
	|\eta|^{-1}\big(\nabla_\xi q_{\theta}(\xi,\eta,\sigma)\big)^2
	 \Pi_\theta(\xi)\Pi_{-\theta}(\xi+\eta) \\
	& \times \rho_{N_0}(\xi)  \widehat{\Psi_{-\theta;N_1}}(s,\xi+\eta)
	\big\langle \widehat{\Psi_{-\theta;N_2}}(s,\xi+\eta+\sigma),
	\widehat{\Psi_{\theta;N_3}}(s,\xi+\sigma) \big\rangle d\sigma d\eta, \\ 
	\widetilde{\mathbb{J} }_{\theta;\mathbf{N}}(s,\xi)&=-s^2  \iint_{\mathbb{R}^2\times\mathbb{R}^2}  e^{is q_{\theta}(\xi,\eta,\sigma)}
	|\eta|^{-1}\nabla_\sigma q_{\theta}(\xi,\eta,\sigma)\big( \nabla_\sigma q_{\theta}(\xi,\eta,\sigma) - 2\nabla_\xi q_{\theta}(\xi,\eta,\sigma) \big)
	 \Pi_\theta(\xi)\Pi_{-\theta}(\xi+\eta) \\
	& \times \rho_{N_0}(\xi) \widehat{\Psi_{-\theta;N_1}}(s,\xi+\eta)
	\big\langle \widehat{\Psi_{-\theta;N_2}}(s,\xi+\eta+\sigma),
	\widehat{\Psi_{\theta;N_3}}(s,\xi+\sigma) \big\rangle d\sigma d\eta. 
\end{aligned}\end{align}
Before estimating the two terms
in \eqref{J3decomposition} separately, we first treat the high-frequency regime.
In this regime, the particular choice of the sign configuration
\(\mathbf{\Theta}=(\theta,\theta_1,\theta_2,\theta_3)\)
plays no role, and the desired estimate follows solely from the high Sobolev regularity assumption.
More precisely, since the multiplier in
\(\mathcal J_{\mathbf{\Theta};\mathbf N}^3\)
is uniformly bounded with respect to
\(\mathbf{\Theta}\)
by \eqref{boundedness of M}, H\"older's inequality together with the frequency-localized estimates in \eqref{Frequency localized inequalities:Inhomogeneous} yields
\begin{align*}
  & \sum_{ \substack{ \mathbf{N}=(N_0,N_1,N_2,N_3)\in (2^{\mathbb{N}\cup\{0\}})^4 \\ N_{\max} \ge \langle s\rangle^{3/n} }} \| \mathcal J_{\mathbf{\Theta};\mathbf{N}}^3(s)\|_{L^2(\mathbb{R}^2)} \\ 
   &\qquad\lesssim \sum_{ \mathbf{N}\in (2^{\mathbb{N}\cup\{0\}})^4  \,:\, N_{\max} \ge \langle s\rangle^{3/n}  } s^2\|\rho_{N_0}\|_{L^2(\R^2)}\|\rho_{\min(N_1,N_2,N_3)}\|_{L^2(\R^2)}\prod_{i=1}^3\|\psi_{\theta_i;N_i}(s)\|_{L^2(\R^2)} \\ 
   &\qquad\lesssim \langle s\rangle^{-1+2\delta_0}\ep_1^3.
\end{align*} 
where, as before, $N_{\max}=\max\{N_i:0\le i\le 3\}$.
Thus, it suffices to show that
\begin{align}\label{ineq:mahtbbJ}
 \sum_{\substack{\mathbf{N}:=(N_0,N_1,N_2,N_3)\in (2^{\N\cup\{0\}})^4 \\ N_{\max}\le \langle s\rangle^{3/n}} } \Big( \| \mathbb{J} _{\theta;\mathbf{N}}(s)\|_{L^2(\R^2)} + \| \widetilde{\mathbb{J}} _{\theta;\mathbf{N}}(s)\|_{L^2(\R^2)}\Big) \lesssim \ve_1^3	\langle s\rangle^{-1+2\delta_0}.
\end{align}
We first consider
\(\widetilde{\mathbb J}_{\theta;\mathbf N}\),
which is easier to handle,
and then turn to
\(\mathbb J_{\theta;\mathbf N}\).

\smallskip

\noindent\underline{\textit{(1) Estimates for }$\widetilde{\mathbb{J} }_{\theta;\mathbf{N}}$.}  
We observe that $\nabla_\sigma q_{\theta}$ exhibits a time decay effect, which we now exploit.  Indeed, integrating by parts in $\sigma$ using the relation
\begin{align*}
	\nabla_{\sigma}q_{\theta}(\xi,\eta,\sigma)e^{is q_{\theta}(\xi,\eta,\sigma)} = -\frac{i }{s}
	\nabla_\sigma e^{is q_{\theta}(\xi,\eta,\sigma)},
\end{align*}
we decompose, with a slight abuse of notation, 
\begin{align*}\begin{aligned}
\widetilde{\mathbb{J} }_{\theta;\mathbf{N}}(s,\xi)
&= 	\widetilde{\mathbb{J} }_{\theta;\mathbf{N}}^{1}(s,\xi)+\widetilde{\mathbb{J} }_{\theta;\mathbf{N}}^{2}(s,\xi) \\	
	\widetilde{\mathbb{J} }_{\theta;\mathbf{N}}^1(s,\xi)&=-is\iint_{\mathbb{R}^{2}\times \R^2} e^{isq_{\theta}(\xi,\eta,\sigma)}
\widetilde{\mathfrak{U}}_{\theta}(\xi,\eta,\sigma)\rho_{N_0}(\xi) \\ 
	&\qquad\quad \times\widehat{\Psi_{-\theta;N_1}}(s,\xi+\eta)\nabla_\sigma \big\langle \widehat{\Psi_{-\theta;N_2}}(s,\xi+\eta+\sigma),
	\widehat{\Psi_{\theta;N_3}}(s,\xi+\sigma) \big\rangle d\sigma d\eta , \\ 
\widetilde{\mathbb{J} }_{\theta;\mathbf{N}}^{2}(s,\xi) &=-is\iint_{\mathbb{R}^{2}\times \R^2} e^{is q_{\theta}(\xi,\eta,\sigma)}
	\nabla_\sigma 		\widetilde{\mathfrak{U}}_{\theta}(\xi,\eta,\sigma)\rho_{N_0}(\xi) \\ 
	&\qquad\quad \times \widehat{\Psi_{-\theta;N_1}}(s,\xi+\eta)\big\langle \widehat{\Psi_{-\theta;N_2}}(s,\xi+\eta+\sigma),
	\widehat{\Psi_{\theta;N_3}}(s,\xi+\sigma) \big\rangle d\sigma d\eta,
\end{aligned}\end{align*}
where
\begin{align*}
\widetilde{\mathfrak{U}}_{\theta}(\xi,\eta,\sigma)&:=\big( \nabla_\sigma q_{\theta}(\xi,\eta,\sigma) - 2\nabla_\xi q_{\theta}(\xi,\eta,\sigma) \big)
	 |\eta|^{-1}\Pi_\theta(\xi)\Pi_{-\theta}(\xi+\eta).
\end{align*}
Thanks to the Dirac null structure \eqref{eq:null} arising from 
$\Pi_\theta(\xi)\Pi_{-\theta}(\xi+\eta)$,
one can verify that 
\begin{align*}
 \big\| \widetilde{\mathfrak{U}}_{\theta} \rho_{\mathbf{N}}\big\|_{\text{CM}[(\R^2)^3]} &:=  \left\Vert \int\!\!\!\!\int\!\!\!\!\int_{(\R^2)^3}\widetilde{\mathfrak{U}}_{\theta}(\xi,\eta,\sigma) \rho_{\mathbf{N}}(\xi,\eta,\sigma)e^{ix\cdot\xi}e^{iy\cdot\eta}e^{iz\cdot\sigma} \,d\sigma d\eta  d\xi\right\Vert _{L_{x,y,z}^{1}((\R^2)^3)} \\ &\les 1,
\end{align*}
and, similarly, \begin{align*}
\big\| (\nabla_{\sigma}\widetilde{\mathfrak{U}}_{\theta} )\rho_{\mathbf{N}}\big\|_{\text{CM}[(\R^2)^3]}
\les \frac{1}{\min( N_2,N_3)}.
\end{align*} 
With these multiplier bounds at hand, we apply \eqref{eq:coif-2} to the respective terms and deduce 
\begin{align*}
& \sum_{ \mathbf{N}\in (2^{\N\cup\{0\}})^4 \,:\, N_{\max}\le \langle s\rangle^{3/n} } 
\Big\|\widetilde{\mathbb{J} }_{\theta;\mathbf{N}}(s)\Big\|_{L^2(\R^2)} \\ 
	&\les  	 \sum_{ \mathbf{N}\in (2^{\N\cup\{0\}})^4 \,:\, N_{\max}\le \langle s\rangle^{3/n} }  |s| 
	\Big(
	\big(\|x\Psi_{-\theta;N_2}(s)\|_{L_x^2(\R^2)}+\|\psi_{-\theta;N_2}(s)\|_{L^2(\R^2)}\big)\|\psi_{\theta;N_3}(s)\|_{L^\infty}  \\ 
	&\qquad\qquad\quad+ \|\psi_{-\theta;N_2}(s)\|_{L^\infty}\big(\|x\Psi_{\theta;N_3}(s)\|_{L_x^2(\R^2)}  +\|\psi_{\theta;N_3}(s)\|_{L^2(\R^2)}\big)
	\Big)  \|\psi_{-\theta;N_1}(s)\|_{L^\infty(\R^2)} \\ 
	&\les \bra{s}^{-1+2\de_0}\ve_1^3.
\end{align*}
Here we used the following frequency-localized weighted estimate.
Suppose that $\psi$ satisfies the a priori assumption
\eqref{assumption-apriori}
for some $\varepsilon_1>0$.
Then, for every dyadic number
\(N\in2^{\mathbb N\cup\{0\}}\)
and
\(\theta\in\{+,-\}\),
\begin{align}
\label{frequency localized with x weight}
\|x\Psi_{\theta;N}(s)\|_{L_x^2(\R^2)}
\lesssim
\langle s\rangle^{\delta_0}\varepsilon_1.
\end{align}
Indeed,
\begin{align*}
\|x\Psi_{\theta;N}(s)\|_{L_x^2(\R^2)}
\approx
\|\nabla\widehat{\Psi_{\theta;N}}(s)\|_{L^2(\R^2)}
\lesssim
\|(\nabla\rho_N)\widehat{\Psi_\theta}(s)\|_{L^2(\R^2)}
+
\|P_N(x\Psi_\theta)(s)\|_{L_x^2(\R^2)}
\lesssim
\langle s\rangle^{\delta_0}\varepsilon_1.
\end{align*}

\smallskip

\noindent\underline{\textit{(2) Estimates for }$\mathbb{J} _{\theta;\mathbf{N}}$.} 
We consider the term $\mathbb{J}_{\theta;\mathbf{N}}$  in \eqref{J3decomposition}.
To compensate for the quadratic time-loss factor $s^2$, we exploit the following space resonance 
\begin{align*}
|\nabla_\eta q_{\theta}(\xi,\eta,\sigma)| &= \left|   \frac{\xi+\eta}{ \bra{\xi+\eta}} - \frac{\xi+\eta+\sigma}{ \bra{\xi+\eta+\sigma}} \right|   \\ 
&\gtrsim  \frac{|\sigma|}{\max(\bra{\xi+\eta} , \bra{\xi+\eta+\sigma})\min(\bra{\xi+\eta} , \bra{\xi+\eta+\sigma})^2},
\end{align*}
which is consistent with \eqref{ineq:Phase lower bound}.
The singularity arising from this procedure will be compensated by the multiplier, since it exhibits a null structure with respect to $\sigma$:
\begin{align}\begin{aligned}\label{Null nabla xi}
	|\nabla_{\xi}q_{\theta}(\xi,\eta,\sigma)| &\le \left|   \frac{\xi}{ \bra{\xi}} - \frac{\xi+\sigma}{ \bra{\xi+\sigma}} \right|  + \left|   \frac{\xi+\eta}{ \bra{\xi+\eta}} - \frac{\xi+\eta+\sigma}{ \bra{\xi+\eta+\sigma}} \right|  \\ 
	&\les |\sigma| \Big( \frac{1}{\max(\langle\xi\rangle, \langle\xi+\sigma\rangle)}  + \frac{1}{\max(\langle\xi+\eta\rangle, \langle\xi+\eta+\sigma\rangle)} \Big),
\end{aligned}\end{align}
where we used \eqref{ineq:Phase null structure}.
To close the weighted energy estimates, we will repeatedly perform integration by parts in $\eta$, up to three times.
However, this inevitably leads to higher-order singularities in the potential term $\nabla_{\eta}^{m}|\eta|^{-1}$ whenever integration by parts is performed $m$ times. To bound such singularities, besides the Dirac null structure from $\Pi_{\theta}(\xi)\Pi_{-\theta}(\xi+\eta)$, we also employ the one hidden in the inner product $\langle \widehat{\Psi_{\theta}}(s,\xi+\eta+\sigma),
	\widehat{\Psi_{-\theta}}(s,\xi+\sigma) \rangle$.

We now proceed with the proof based on this strategy. 
Since each integration by parts in $\eta$ introduces additional singular factors involving both $|\eta|$ and $|\sigma|$, we further introduce a homogeneous dyadic decomposition of the $\eta$- and $\sigma$-variables adapted to the singular behavior near the origin.
\begin{align*}
\mathbb{J} _{\theta;\mathbf{N}}(s,\xi)&=
\sum_{\mathbf{L}:=(L_1,L_2)\in (2^{\Z})^2} \mathbb{J} _{\theta;\mathbf{N},\mathbf{L}}(s,\xi), \\ 
\mathbb{J} _{\theta;\mathbf{N},\mathbf{L}}(s,\xi)&=- s^2 \iint_{\mathbb{R}^2\times\mathbb{R}^2}  e^{is q_{\theta}(\xi,\eta,\sigma)}
\mathfrak{U}_{\theta}(\xi,\eta,\sigma)\rho_{N_0}(\xi)\chi_{L_1}(\eta)\chi_{L_2}(\sigma)  \\
& \times \widehat{\Psi_{-\theta;N_1}}(s,\xi+\eta)
\big\langle \widehat{\Psi_{-\theta;N_2}}(s,\xi+\eta+\sigma),
	\widehat{\Psi_{\theta;N_3}}(s,\xi+\sigma) \big\rangle d\sigma d\eta,    
\end{align*}
where 
\begin{align*}
\mathfrak{U}_{\theta}(\xi,\eta,\sigma)   :=\big( \nabla_\xi q_{\theta}(\xi,\eta,\sigma)\big)^2
	|\eta|^{-1} \Pi_\theta(\xi)\Pi_{-\theta}(\xi+\eta).
\end{align*}
To establish \eqref{ineq:mahtbbJ}, we prove that 
\begin{align*}
\sum_{\mathbf{N}\in (2^{\N\cup\{0\}})^4\,:\, N_{\max}\le \langle s\rangle^{3/n}}\sum_{\mathbf{L}:=(L_1,L_2)\in (2^{\Z})^2}\| \mathbb{J} _{\theta;\mathbf{N},\mathbf{L}}(s)\|_{L^2(\R^2)}\lesssim \ve_1^3	\langle s\rangle^{-1+2\delta_0}.
\end{align*}

We first consider the low- and high-frequency contributions with respect to the dyadic variables $L_1$ and $L_2$. More precisely, all contributions satisfying
$$ L_{\min}:=\min(L_1,L_2) \le \langle s\rangle^{-3/2} \; \text{ or } \; L_{\max}:=\max(L_1,L_2)\ge \langle s\rangle^{3/n} $$
can be estimated directly using H\"older's inequality and the frequency-localized bounds \eqref{Frequency localized inequalities:Inhomogeneous}.
Indeed, by the Dirac null structure \eqref{eq:null}, the multiplier satisfies
\begin{align*}
\sup_{\xi,\eta,\sigma\in\R^2} \big| \mathfrak{U}_{\theta}(\xi,\eta,\sigma) \big| \lesssim  1.
\end{align*}
Suppose $N_{\max}=N_1$. Using H\"older's inequality, we have
\begin{align*}
&\|\mathbb J_{\theta;\mathbf{N},\mathbf{L}}(s)\|_{L^2(\R^2)}  \\
&\quad\lesssim  s^2    \|\chi_{L_1}\|_{L^1(\R^2)}\|\chi_{L_2}\|_{L^1(\R^2)}\|\Psi_{-\theta;N_{1}}(s)\|_{L^2(\R^2)}\|\widehat{\Psi_{-\theta;N_{2}}}(s)\|_{L^\infty(\R^2)}\|\widehat{\Psi_{\theta;N_{3}}}(s)\|_{L^\infty(\R^2)} \\ 
&\quad\lesssim s^2 L_1^{\,2}L_2^{\,2}N_1^{-n}N_2^{-k}N_3^{-k}\ve_1^3.
\end{align*}
Summing over all low- and high-frequency regimes, we obtain
\begin{align*}
\sum_{\mathbf{N}\in (2^{\N\cup\{0\}})^4\,:\, N_{\max}=N_1\le \langle s\rangle^{3/n}}
\sum_{\mathbf{L}\in (2^{\Z})^2\, : \, L_{\min}\le \langle s\rangle^{-3/2} \text{ or } L_{\max}\ge \langle s\rangle^{3/n} } \| \mathbb{J} _{\theta;\mathbf{N},\mathbf{L}}(s)\|_{L^2(\R^2)}\lesssim \ve_1^3	\langle s\rangle^{-1+2\delta_0},
\end{align*}
where, for the high-frequency regime, we used $L_1,L_2 \lesssim N_{\max}$, which follows from the frequency relation. The contributions with $N_{\max}=N_2$ or $N_3$ are estimated in exactly the same way.

We now consider the remaining contributions. More precisely, we prove that \begin{align*}
\sum_{\mathbf{N}\in (2^{\N\cup\{0\}})^4\,:\, N_{\max}\le \langle s\rangle^{3/n}}
\sum_{\mathbf{L}\in (2^{\Z})^2\, : \, \langle s\rangle^{-3/2} \le L_1,L_2\le  \langle s\rangle^{3/n} } \| \mathbb{J} _{\theta;\mathbf{N},\mathbf{L}}(s)\|_{L^2(\R^2)}\lesssim \ve_1^3	\langle s\rangle^{-1+2\delta_0}.
\end{align*}
Using the relation
\begin{align*}
e^{isq_{\theta}(\xi,\eta,\sigma)} = -i \left(\frac{\nabla_\eta q_{\theta} \nabla_\eta e^{isq_{\theta}} }{s|\nabla_\eta q_{\theta}|^2} \right)(\xi,\eta,\sigma),
\end{align*}
we integrate by parts to obtain 
\begin{align*}
\mathbb{J}_{\theta;\mathbf{N},\mathbf{L}}(s,\xi)&=	\mathbb{J}_{\theta;\mathbf{N},\mathbf{L}}^{1}(s,\xi)+	\mathbb{J}_{\theta;\mathbf{N},\mathbf{L}}^{2}(s,\xi) \\ 
\mathbb{J}_{\theta;\mathbf{N},\mathbf{L}}^{1}(s,\xi) &=-is \iint_{\mathbb{R}^{2}\times \R^2} e^{isq_{\theta}(\xi,\eta,\sigma)}\left(\mathfrak{U}_\theta\frac{\nabla_{\eta}q_\theta}{|\nabla_{\eta}q_\theta|^2}\right)(\xi,\eta,\sigma)\rho_{N_0}(\xi)\chi_{L_1}(\eta)\chi_{L_2}(\sigma) \\  
&\quad\times \nabla_\eta  \Big(  \widehat{\Psi_{-\theta;N_1}}(s,\xi+\eta)
	\big\langle \widehat{\Psi_{-\theta;N_2}}(s,\xi+\eta+\sigma),
	\widehat{\Psi_{\theta;N_3}}(s,\xi+\sigma) \big\rangle  \Big) d\sigma d\eta , \\ 
\mathbb{J}_{\theta;\mathbf{N},\mathbf{L}}^{2}(s,\xi)&=-is\iint_{\mathbb{R}^{2}\times \R^2} e^{isq_{\theta}(\xi,\eta,\sigma)}
\nabla_\eta	\Big( \mathfrak{U}_\theta\frac{\nabla_{\eta}q_\theta}{|\nabla_{\eta}q_\theta|^2}\chi_{L_1}\Big )  \freq \rho_{N_0}(\xi)\chi_{L_2}(\sigma) \\
&\quad\times \widehat{\Psi_{-\theta;N_1}}(s,\xi+\eta)
\big\langle \widehat{\Psi_{-\theta;N_2}}(s,\xi+\eta+\sigma),
	\widehat{\Psi_{\theta;N_3}}(s,\xi+\sigma) \big\rangle d\sigma d\eta.
\end{align*}
A direct computation yields that 
\begin{align*}
&\left\|\mathfrak{U}_\theta\frac{\nabla_{\eta}q_\theta}{|\nabla_{\eta}q_\theta|^2}\rho_{\mathbf{N}}\chi_{L_2}\right\|_{\cm[(\R^2)^3]} \\
&:= \left\Vert \int\!\!\!\!\int\!\!\!\!\int_{(\R^2)^3}\left[ \mathfrak{U}_\theta  \frac{\nabla_{\eta}q_\theta}{|\nabla_{\eta}q_\theta|^2} \rho_{\mathbf{N}}\right](\xi,\eta,\sigma)\chi_{L_2}(\sigma)e^{ix\cdot\xi}e^{iy\cdot\eta}e^{iz\cdot\sigma} \,d\sigma d\eta  d\xi\right\Vert _{L_{x,y,z}^{1}((\R^2)^3)}\\ 
&\les L_2 N_{\max}^{\,9},
\end{align*} 
where $\rho_{\mathbf{N}}$ is defined in \eqref{dyadicdecomposition}.  
Applying the multiplier inequality \eqref{eq:coif-2} with this bound, and then using the frequency-localized inequalities \eqref{Frequency localized inequalities:Inhomogeneous} and \eqref{frequency localized with x weight}, we estimate 
\begin{align*}
&\sum_{\mathbf{N}\in (2^{\N\cup\{0\}})^4, \,\mathbf{L}\in (2^{\Z})^2\,:\, N_{\max}\le \langle s\rangle^{3/n},\,\langle s\rangle^{-3/2} \le L_1,L_2\le  \langle s\rangle^{3/n}}
\big\|\mathbb{J}_{\theta;\mathbf{N},\mathbf{L}}^{1}(s)\big\|_{L^2(\R^2)} \\
&\les \sum_{ \substack{N_{\max}\le \langle s\rangle^{3/n} \\ \langle s\rangle^{-3/2} \le L_1,L_2\le  \langle s\rangle^{3/n}}} |s| L_2N_{\max}^{\,9} \Big(\|x \Psi_{-\theta;N_1}(s)\|_{L_x^2(\R^2)} \|\psi_{-\theta;N_2}(s)\|_{L^\infty(\R^2)} \\ 
&\hspace{4cm} +\|x \Psi_{-\theta;N_2}(s)\|_{L_x^2(\R^2)} \|\psi_{-\theta;N_1}(s)\|_{L^\infty(\R^2)}\Big) \|\psi_{\theta;N_3}(s)\|_{L^\infty(\R^2)}  \\
&\lesssim \bra{s}^{-1+\de_0} \ve_1^3\sum_{ \substack{N_{\max}\le \langle s\rangle^{3/n} \\ \langle s\rangle^{-3/2} \le L_1,L_2\le  \langle s\rangle^{3/n}}} L_2N_{\max}^{\,9}(N_2^{-k}+N_1^{-k})N_3^{-k}\\ 
&\les \bra{s}^{-1+2\de_0}\ve_1^3.
\end{align*}
Here and in what follows, we repeatedly use the frequency restrictions together with the rapid decay in the inhomogeneous dyadic frequencies. Since \(n\) is sufficiently large and \(k\) is also large, all losses arising from the dyadic summations, including logarithmic losses, are absorbed into the harmless factor \(\langle s\rangle^{\delta_0}\).

Next, we estimate $\mathbb{J}_{\theta;\mathbf{N},\mathbf{L}}^{2}$.  We begin by integrating by parts in $\eta$ once more to obtain 
\begin{align*}
	\mathbb{J}_{\theta;\mathbf{N},\mathbf{L}}^{2}(s,\xi)&=	\mathbb{J}_{\theta;\mathbf{N},\mathbf{L}}^{2,1}(s,\xi)+	\mathbb{J}_{\theta;\mathbf{N},\mathbf{L}}^{2,2}(s,\xi), \\ 
	\mathbb{J}_{\theta;\mathbf{N},\mathbf{L}}^{2,1}(s,\xi)&=\iint_{\mathbb{R}^{2}\times \R^2} 
	\Big[ e^{isq_{\theta}}  \frac{\nabla_\eta q_{\theta}  }{|\nabla_\eta q_{\theta}|^2}
	\nabla_\eta	\Big( \mathfrak{U}_\theta\frac{\nabla_{\eta}q_\theta}{|\nabla_{\eta}q_\theta|^2} \chi_{L_1}\Big )  \Big] \freq \rho_{N_0}(\xi)\chi_{L_2}(\sigma)\\ 
	&\quad\times\nabla_\eta \left( \widehat{\Psi_{-\theta;N_1}}(s,\xi+\eta)
	\big\langle \widehat{\Psi_{-\theta;N_2}}(s,\xi+\eta+\sigma),
	\widehat{\Psi_{\theta;N_3}}(s,\xi+\sigma) \big\rangle \right) d\sigma d\eta, \\
\mathbb{J}_{\theta;\mathbf{N},\mathbf{L}}^{2,2}(s,\xi)&= \iint_{\mathbb{R}^{2}\times \R^2} 	\Big[ e^{isq_{\theta}}\nabla_\eta\Big(\frac{\nabla_\eta q_{\theta}  }{|\nabla_\eta q_{\theta}|^2}
	\nabla_\eta	\Big( \mathfrak{U}_\theta	 \frac{\nabla_{\eta}q_\theta}{|\nabla_{\eta}q_\theta|^2}\chi_{L_1}\Big ) \Big)\Big]  \freq \rho_{N_0}(\xi)\chi_{L_2}(\sigma) \\
&\quad\times \widehat{\Psi_{-\theta;N_1}}(s,\xi+\eta)
	\big\langle \widehat{\Psi_{-\theta;N_2}}(s,\xi+\eta+\sigma),
	\widehat{\Psi_{\theta;N_3}}(s,\xi+\sigma) \big\rangle  d\sigma d\eta.
\end{align*}
We consider $\mathbb{J}_{\theta;\mathbf{N},\mathbf{L}}^{2,1}$.
A direct calculation, together with the Dirac null structure \eqref{eq:null}, yields the following pointwise bound for the  multiplier:
\begin{align*}
\normo{ \frac{\nabla_\eta q_{\theta}  }{|\nabla_\eta q_{\theta}|^2}
\nabla_\eta	\Big( \mathfrak{U}_\theta\frac{\nabla_{\eta}q_\theta}{|\nabla_{\eta}q_\theta|^2}\chi_{L_1}\Big )\rho_{\mathbf{N}} \freq \chi_{L_2}(\sigma)}_{L_{\xi,\eta,\sigma}^\infty((\R^2)^3)} \les L_1^{-1}N_{\max}^{\,4}.
\end{align*}
Then, using H\"older's inequality, we estimate the contribution from the range $L_1\le \langle s\rangle^{-1}$ as 
\begin{align*}
&\sum_{\substack{\mathbf{N}\in (2^{\N\cup\{0\}})^4, \,\mathbf{L}\in (2^{\Z})^2 \\  N_{\max}\le \langle s\rangle^{3/n}\,,\, \langle s\rangle^{-3/2}\le L_1,L_2\le\langle s\rangle^{3/n}\\ L_1 \le \bra{s}^{-1} }}  
\|\mathbb{J}_{\theta;\mathbf{N},\mathbf{L}}^{2,1}(s)\|_{L^2(\R^2)} 	\\
&\les  \sum_{\substack{N_{\max}\le \langle s\rangle^{3/n}\,,\, \langle s\rangle^{-3/2}\le L_1,L_2\le\langle s\rangle^{3/n}\\ L_1 \le \bra{s}^{-1} }}L_1L_2^2N_{\max}^{\,4} \Big( \|x\Psi_{-\theta;N_1}(s)\|_{L_x^2(\R^2)}\|\wh{\Psi_{-\theta;N_2}}(s)\|_{L^\infty(\R^2)} \\ 
&\qquad\qquad\qquad +  \|\wh{\Psi_{-\theta;N_1}}(s)\|_{L^\infty(\R^2)}\|x\Psi_{-\theta;N_2}(s)\|_{L_x^2(\R^2)} \Big) \|\wh{\Psi_{\theta;N_3}}(s)\|_{L^\infty(\R^2)} \\
&\les  \bra{s}^{-1+2\de_0}\ve_1^3. 
\end{align*}
On the other hand, using the multiplier inequality \eqref{eq:coif-2} together with
\begin{align*}
&\Big\| \frac{\nabla_\eta q_{\theta}  }{|\nabla_\eta q_{\theta}|^2}
	\nabla_\eta	\Big( \mathfrak{U}_\theta\frac{\nabla_{\eta}q_\theta}{|\nabla_{\eta}q_\theta|^2} \chi_{L_1}\Big )\rho_{\mathbf{N}}\chi_{L_2}\Big\|_{\text{CM}[(\R^2)^3]} \\
&\quad := \left\Vert \int\!\!\!\!\int\!\!\!\!\int_{(\R^2)^3}
\frac{\nabla_\eta q_{\theta}  }{|\nabla_\eta q_{\theta}|^2}
\nabla_\eta	\Big( \mathfrak{U}_\theta\frac{\nabla_{\eta}q_\theta}{|\nabla_{\eta}q_\theta|^2}\chi_{L_1}\Big )\rho_{\mathbf{N}}(\xi,\eta,\sigma)\chi_{L_2}(\sigma)
e^{ix\cdot\xi}e^{iy\cdot\eta}e^{iz\cdot\sigma} \,d\sigma d\eta  d\xi\right\Vert _{L_{x,y,z}^{1}((\R^2)^3)}\\ 
&\quad \les L_1^{-1}\bra{N_{\max}}^{14}, 
\end{align*}
we bound the remaining contribution (corresponding to $L_1\ge\langle s\rangle^{-1}$) by 
\begin{align*}
&\sum_{\substack{\mathbf{N}\in (2^{\N\cup\{0\}})^4, \,\mathbf{L}\in (2^{\Z})^2 \\  N_{\max}\le \langle s\rangle^{3/n},\, \langle s\rangle^{-3/2}\le L_1,L_2\le\langle s\rangle^{3/n}\\ L_1 \ge \bra{s}^{-1} }}  \|\mathbb{J}_{\theta;\mathbf{N},\mathbf{L}}^{2,1}(s)\|_{L^2(\R^2)}  \\ 
&\quad \les  \sum_{\substack{N_{\max}\le \langle s\rangle^{3/n},\, \langle s\rangle^{-3/2}\le L_1,L_2\le\langle s\rangle^{3/n}\\ L_1 \ge \bra{s}^{-1} }} L_1^{-1} N_{\max}^{14} \|x\Psi_{-\theta;N_1}(s)\|_{L^2(\R^2)}\|\psi_{-\theta;N_2}(s)\|_{L^\infty(\R^2)}\|\psi_{\theta;N_3}(s)\|_{L^\infty(\R^2)} \\
 &\quad \les  \bra{s}^{-1+2\de_0}\ve_1^3.
\end{align*}

Finally, for $\mathbb{J}_{\theta;\mathbf{N},\mathbf{L}}^{2,2}$, we integrate by parts in $\eta$ again to write 
\begin{align*}
	\mathbb{J}_{\theta;\mathbf{N},\mathbf{L}}^{2,2}(s,\xi)
	&=\mathbb{J}_{\theta;\mathbf{N},\mathbf{L}}^{2,2,1}(s,\xi)
+\mathbb{J}_{\theta;\mathbf{N},\mathbf{L}}^{2,2,2}(s,\xi) \\ 
\mathbb{J}_{\theta;\mathbf{N},\mathbf{L}}^{2,2,1}(s,\xi)&=\frac{i}{s}\iint_{\mathbb{R}^{2}\times \R^2} e^{isq_{\theta}}\frac{\nabla_\eta q_{\theta}   }{|\nabla_\eta q_{\theta}|^2} \nabla_\eta\Big(\frac{\nabla_\eta q_{\theta}  }{|\nabla_\eta q_{\theta}|^2}
	\nabla_\eta	\Big( \mathfrak{U}_\theta\frac{\nabla_{\eta}q_\theta}{|\nabla_{\eta}q_\theta|^2} \chi_{L_1} \Big ) \Big)  \freq \rho_{N_0}(\xi)\chi_{L_2}(\sigma)\\
	&\quad \times	\nabla_\eta \left( \widehat{\Psi_{-\theta;N_1}}(s,\xi+\eta)
	\big\langle \widehat{\Psi_{-\theta;N_2}}(s,\xi+\eta+\sigma),
	\widehat{\Psi_{\theta;N_3}}(s,\xi+\sigma) \big\rangle \right) d\sigma d\eta \\ 
\mathbb{J}_{\theta;\mathbf{N},\mathbf{L}}^{2,2,2}(s,\xi)&=\frac{i}{s}\iint_{\mathbb{R}^{2}\times \R^2} e^{isq_{\theta} }
	\nabla_{\eta}\left(\frac{\nabla_\eta q_{\theta}    }{|\nabla_\eta q_{\theta} |^2}\nabla_\eta\Big\{ \frac{\nabla_\eta q_{\theta}   }{|\nabla_\eta q_{\theta} |^2}
	\nabla_\eta		\Big( \mathfrak{U}_\theta\frac{\nabla_{\eta}q_\theta}{|\nabla_{\eta}q_\theta|^2} \chi_{L_1} \Big )\Big\} \right) \freq \rho_{N_0}(\xi)\chi_{L_2}(\sigma)\\
	&\quad \times \widehat{\Psi_{-\theta;N_1}}(s,\xi+\eta)
	\big\langle \widehat{\Psi_{-\theta;N_2}}(s,\xi+\eta+\sigma),
	\widehat{\Psi_{\theta;N_3}}(s,\xi+\sigma) \big\rangle  d\sigma d\eta.
\end{align*}
Using both the phase null structure \eqref{Null nabla xi} and the Dirac null structure \eqref{eq:null}, one can verify the following pointwise multiplier bounds: 
\begin{align*}
	\normo{ \frac{\nabla_\eta q_{\theta}   }{|\nabla_\eta q_{\theta}|^2} \nabla_\eta\Big(\frac{\nabla_\eta q_{\theta}  }{|\nabla_\eta q_{\theta}|^2}
	\nabla_\eta	\Big( \mathfrak{U}_\theta\frac{\nabla_{\eta}q_\theta}{|\nabla_{\eta}q_\theta|^2} \chi_{L_1} \Big ) \Big) \rho_{\mathbf{N}}\freq\chi_{L_2}(\sigma)}_{L_{\xi,\eta,\sigma}^\infty((\R^2)^3)}&\les L_1^{-2}L_2^{-1}N_{\max}^{\,7},  \\ 
	 \normo{\nabla_{\eta}\left(\frac{\nabla_\eta q_{\theta}   }{|\nabla_\eta q_{\theta}|^2} \nabla_\eta\Big(\frac{\nabla_\eta q_{\theta}  }{|\nabla_\eta q_{\theta}|^2}
	\nabla_\eta	\Big( \mathfrak{U}_\theta\frac{\nabla_{\eta}q_\theta}{|\nabla_{\eta}q_\theta|^2} \chi_{L_1}\Big ) \Big) \right)\rho_{\mathbf{N}}\freq \chi_{L_2}(\sigma)}_{L_{\xi,\eta,\sigma}^\infty((\R^2)^3)} &\les L_1^{-3}L_2^{-1}N_{\max}^{\,9}.
\end{align*} 
From the first pointwise bound, we estimate $\mathbb{J}_{\theta;\mathbf{N},\mathbf{L}}^{2,2,1}$ depending on where the derivative $\nabla_\eta$ falls:
\begin{align*}
&\sum_{\substack{\mathbf{N}\in (2^{\N\cup\{0\}})^4, \,\mathbf{L}\in (2^{\Z})^2 \\  N_{\max}\le \langle s\rangle^{3/n},\, \langle s\rangle^{-3/2}\le L_1,L_2\le\langle s\rangle^{3/n} }} \|\mathbb{J}_{\theta;\mathbf{N},\mathbf{L}}^{2,2,1}(s)\|_{L^2(\R^2)} \\	
&\les  \sum_{N_{\max}\le \langle s\rangle^{3/n},\, \langle s\rangle^{-3/2}\le L_1,L_2\le\langle s\rangle^{3/n} } |s|^{-1} L_1^{-2}L_2^{-1}N_{\max}^{\,7} \|\chi_{L_1}\|_{L^1(\R^2)}\|\chi_{L_2}\|_{L^1(\R^2)}\|\wh{\Psi_{\theta;N_3}}(s)\|_{L^\infty(\R^2)}\\ 
&\qquad\times \Big(  \|x\Psi_{-\theta;N_1}(s)\|_{L_x^2(\R^2)}\|\wh{\Psi_{-\theta;N_2}}(s)\|_{L^\infty(\R^2)} + \|\wh{\Psi_{-\theta;N_1}}(s)\|_{L^\infty(\R^2)}\|x\Psi_{-\theta;N_2}(s)\|_{L^2(\R^2)}\Big)  \\
&\les   |s|^{-1+\delta_0}\ve_1^3\sum_{N_{\max}\le \langle s\rangle^{3/n},\, \langle s\rangle^{-3/2}\le L_1,L_2\le\langle s\rangle^{3/n} }  L_2N_{\max}^{\,7}  N_3^{-k} (N_2^{-k}+N_1^{-k})  \\
	&\les \bra{s}^{-1+2\de_0}\ve_1^3.
\end{align*}
For $\mathbb{J}_{\theta;\mathbf{N},\mathbf{L}}^{2,2,2}$, the Dirac null structure in the inner product
\[
\big\langle \widehat{\Psi_{-\theta;N_2}}(s,\xi+\eta+\sigma),
\widehat{\Psi_{\theta;N_3}}(s,\xi+\sigma) \big\rangle
\]
plays a crucial role in compensating for the resulting singularity. Using this null structure together with the second pointwise bound, we estimate
\begin{align*}
&\sum_{\substack{\mathbf{N}\in (2^{\N\cup\{0\}})^4, \,\mathbf{L}\in (2^{\Z})^2 \\  N_{\max}\le \langle s\rangle^{3/n},\, \langle s\rangle^{-3/2}\le L_1,L_2\le\langle s\rangle^{3/n} }}  \| \mathbb{J}_{\theta;\mathbf{N},\mathbf{L}}^{2,2,2}(s)\|_{L^2(\R^2)} 	\\ 
&\les   |s|^{-1}\sum_{N_{\max}\le \langle s\rangle^{3/n},\, \langle s\rangle^{-3/2}\le L_1,L_2\le\langle s\rangle^{3/n} } 	L_1^{-3}L_2^{-1} N_{\max}^{\,9} \||\eta|\chi_{L_1}\|_{L_{\eta}^1(\R^2)}\|\chi_{L_2}\|_{L^1(\R^2)} \\ 
&\qquad\qquad\qquad\qquad\qquad\times \|\Psi_{-\theta;N_1}(s)\|_{L^2(\R^2)}\|\wh{\Psi_{-\theta;N_2}}(s)\|_{L^\infty(\R^2)}\|\wh{\Psi_{\theta;N_3}}(s)\|_{L^\infty(\R^2)}\\
&\les   |s|^{-1+\delta_0}\ve_1^3\sum_{N_{\max}\le \langle s\rangle^{3/n},\, \langle s\rangle^{-3/2}\le L_1,L_2\le\langle s\rangle^{3/n} } 	L_2N_{\max}^{\,9}   N_1^{-n}N_2^{-k}N_3^{-k} \\
&\les  \bra{s}^{-1+2\de_0}\ve_1^3.
\end{align*}

\section{Modified scattering: Proof of Proposition~\ref{prop:scattering}}

\global\long\def\freq{{(\xi,\eta,\sigma)}}%

In this section, we prove Proposition~\ref{prop:scattering}, which establishes the modified scattering behavior of solutions. 
To show estimate \eqref{eq:scattering}, it is enough to prove
\begin{align}
\sum_{\theta\in\{+,-\}}\left\|\bra{\xi}^k \left( e^{-i\mathcal{B}_{\theta}(t_2,\xi)} \wh{\Psi_{\theta}}(t_2,\xi) - e^{-i\mathcal{B}_{\theta}(t_1,\xi)} \wh{\Psi_{\theta}}(t_1,\xi)   \right) \right\|_{L_\xi^\infty(\R^2)}\les \ve_{1}^{3} M^{-\delta},\label{goal-modi}
\end{align}
for some $0<\delta\ll \delta_0$ whenever  $t_1,t_2 \in\left[M/4, 2M \right] \subset [ 0,T]$
for a dyadic number $M\in2^{\mathbb{N}}$.
By Duhamel's formula \eqref{eq:duhamel} and expanding the inner product term as in \eqref{eq:expanded duhamel formula}, we write the profile $\widehat{\Psi_{\theta}}$ as  
\[
\wh{\Psi_\theta}(t,\xi) = \wh{\psi_{0,\theta}}(\xi) + i \frac{\lam}{(2\pi)^3} \sum_{\theta_1,\theta_2,\theta_3\in\{+,-\} } \int_0^t \mathcal N_{\mathbf{\Theta}}(s,\xi) ds,
\]
where $\mathbf{\Theta}=(\theta,\theta_1,\theta_2,\theta_3)$ and 
\begin{align}\begin{aligned}\label{eq:nonlinear term}
	\mathcal N_{\mathbf{\Theta}}(s,\xi) &=  \iint_{\R^{2} \times \R^2}e^{isq_{\mathbf{\Theta}}\freq}|\eta|^{-1}\Pi_{\theta}(\xi)\Pi_{\theta_1}(\xi+\eta)\wh{\Psi_{\theo}}(s,\xi+\eta)\\
	&\hspace{4cm}\times\bra{\wh{\Psi_\thet}(s,\xi+\eta+\sigma),\wh{\Psi_\theth}(s,\xi+\sigma)}d\eta d\sigma,
\end{aligned}\end{align}
with the phase function $q_{\mathbf{\Theta}}$ defined in \eqref{phase interaction q}.
Now, we write the difference as 
\begin{align}\begin{aligned}\label{difference:Modified profile}
	&e^{-i\mathcal{B}_{\theta}(t_2,\xi)} \wh{\Psi_{\theta}}(t_2,\xi) - e^{-i\mathcal{B}_{\theta}(t_1,\xi)} \wh{\Psi_{\theta}}(t_1,\xi) \\ 
	&=\int_{t_1}^{t_2} \partial_s \Big(e^{-i\mathcal{B}_{\theta}(s,\xi)} \wh{\Psi_{\theta}}(s,\xi)  \Big) ds \\ 
	&=i\frac{\lam}{(2\pi)^3} \int_{t_1}^{t_2}e^{-i\mathcal{B}_\theta(s,\xi)}\Big( \sum_{\theta_1,\theta_2,\theta_3\in\{+,-\} } \mathcal N_{\mathbf{\Theta}}(s,\xi)  -\frac{(2\pi)^3}{\lam}\left[\partial_{s}\mathcal{B}_{\theta}(s,\xi)\right]\wh{\Psi_\theta}(s,\xi)\Big) ds.
\end{aligned}\end{align}

Heuristically, the dominant contribution of the nonlinearity \eqref{eq:nonlinear term} arises from the integration near the origin, that is, the low-frequency regime, where the singular factor $|\eta|^{-1}$ comes into play. 
In contrast, the integration over the high-frequency regime can be shown to be integrable in time by exploiting the various space-time resonance structures together with the null structures.

Even concerning the low-frequency regime, among the nonlinear terms with different sign combinations, those corresponding to configurations with $\theta\neq\theta_1$ or $\theta_2\neq\theta_3$ benefit from the Dirac null structure, which compensates the singularity $|\eta|^{-1}$ near the origin. Consequently, such terms can also be shown to be integrable in time and thus treated as error terms.

In summary, the genuine dominant contribution comes from the low-frequency regime of those nonlinear terms arising from the remaining sign category, namely when $\theta=\theta_1$ and $\theta_2=\theta_3$ simultaneously, where no Dirac null structure is available. These are precisely the terms that necessitate the phase modification $\mathcal B_\theta$ and must therefore be canceled through it.

More precisely, we first decompose $\mathcal N_{\mathbf{\Theta}}$ into low- and high-frequency parts with respect to $\eta$, using a threshold $L$ adapted to the time scale, while the high-frequency part is further subdivided dyadically:
\begin{align*} 
\mathcal N_{\mathbf{\Theta}}(s,\xi) =  \mathcal N_{\mathbf{\Theta};L}(s,\xi)+\sum_{L_1\in2^{\Z},L_1>L}\mathcal N_{\mathbf{\Theta};L_1}(s,\xi),
\end{align*}
where $L\in2^{\Z}$ is chosen such that
\begin{align}\label{eq:lzero}
	L\sim M^{-1+\delta_1 } \text{ for some small } \delta_1>0 \text{ with } \delta_0 \ll \delta_1,
\end{align}
and 
\begin{align*}
	\mathcal N_{\mathbf{\Theta};L}(s,\xi) & := \iint_{\R^2 \times \R^2}e^{isq_{\mathbf{\Theta}}\freq}\rho_{\le L}(\eta)|\eta|^{-1}\Pi_{\theta}(\xi)\Pi_{\theta_1}(\xi+\eta)\wh{\Psi_{\theo}}(s,\xi+\eta)\\
	&\hspace{4cm}\times\bra{\wh{\Psi_\thet}(s,\xi+\eta+\sigma),\wh{\Psi_\theth}(s,\xi+\sigma)}d\eta d\sigma,\\
	\mathcal N_{\mathbf{\Theta}; L_1}(s,\xi) & :=\iint_{\R^2 \times \R^2}e^{isq_{\mathbf{\Theta}}\freq}\chi_{ L_1}(\eta)|\eta|^{-1}\Pi_{\theta}(\xi)\Pi_{\theta_1}(\xi+\eta)\wh{\Psi_{\theo}}(s,\xi+\eta)\\
	&\hspace{4cm}\times\bra{\wh{\Psi_\thet}(s,\xi+\eta+\sigma),\wh{\Psi_\theth}(s,\xi+\sigma)}d\eta d\sigma.
\end{align*}
Substituting this dyadic decomposition into \eqref{difference:Modified profile}, the proof of \eqref{goal-modi}  reduces to estimating separately the low-frequency part 
\begin{align}\begin{aligned}\label{eq:part-modification1}
\Bigg|\int_{t_{1}}^{t_{2}}e^{-i\mathcal{B}_{\theta}(s,\xi)}\Bigg( \,  \sum_{  \theta_1,\theta_2,\theta_3\in\{+,-\} } \mathcal N_{\mathbf{\Theta};L}(s,\xi)-\frac {(2\pi)^3}\lam \left[\partial_{s} \mathcal{B}_\theta(s,\xi)\right]\wh{\Psi_\theta}(s,\xi)\Bigg)ds\Bigg|&\les\ve_{1}^{3}M^{-\de}\bra{\xi}^{-k}  
\end{aligned}\end{align}
and the high-frequency contributions
\begin{align}
\left|\int_{t_{1}}^{t_{2}}e^{-i\mathcal{B}_\theta(s,\xi)}\sum_{L_1>L}\mathcal N_{\mathbf{\Theta};L_1}(s,\xi)ds\right|  \les\ve_{1}^{3}M^{-\de}\bra{\xi}^{-k}, \text{ for all } \mathbf{\Theta}.	\label{eq:part-scattering}
\end{align}

Secondly, in the estimates \eqref{eq:part-modification1}, the main contribution among the nonlinear terms with different sign combinations comes precisely from those configurations that do not exhibit the Dirac null structure, namely when $\theta=\theta_1$ and $\theta_2=\theta_3$ simultaneously. These terms, corresponding precisely to $\boldsymbol{\mathfrak{S}}_1$ as defined in \eqref{Classification of signs}, namely $\{(\theta,\theta,\theta',\theta')\;|\;\theta,\theta'\in\{+,-\}\}$, are exactly the cases for which the cancellation effect of $\partial_s\mathcal{B}_{\theta}(s,\xi)$ is essential.
On the other hand, all remaining sign combinations, where the Dirac null structure can be exploited, are integrable in time and therefore contribute only error terms. Hence, the proof of \eqref{eq:part-modification1} further reduces to establishing the desired bounds separately for the essential contributions
\begin{align}\label{eq:part-modification}
\left|\int_{t_{1}}^{t_{2}}e^{-i\mathcal{B}_{\theta}(s,\xi)}\Big( \sum_{ \substack{ \theta_1,\theta_2,\theta_3\in\{+,-\} \\ \mathbf{\Theta}=(\theta,\theta_1,\theta_2,\theta_3)\in \boldsymbol{\mathfrak{S}}_1 } } \mathcal N_{\mathbf{\Theta};L}(s,\xi)  -\frac{(2\pi)^3}{\lam}\left[\partial_{s}\mathcal{B}_{\theta}(s,\xi)\right]\wh{\Psi_\theta}(s,\xi)\Big)ds\right|&\les\ve_{1}^{3}M^{-\de}\bra{\xi}^{-k} 
\end{align}
and for the integrable error terms
\begin{align}\label{eq:error}
\sum_{ \substack{ \theta_1,\theta_2,\theta_3\in\{+,-\} \\ \mathbf{\Theta}=(\theta,\theta_1,\theta_2,\theta_3)\notin \boldsymbol{\mathfrak{S}}_1 } } \bigg|\int_{t_{1}}^{t_{2}}e^{-i\mathcal{B}_\theta(s,\xi)}  \mathcal N_{\mathbf{\Theta};L}(s,\xi)ds \bigg| & \les\ve_{1}^{3}M^{-\de}\bra{\xi}^{-k}.
\end{align}

\subsection{Proof of \eqref{eq:part-modification}}
For the estimate \eqref{eq:part-modification}, when $|\xi|\gtrsim  M^{\frac2n}$, the desired bound \eqref{goal-modi} follows immediately  without the phase modification $\mathcal{B}_\theta(s,\xi)$ (as will be shown below). For this reason, the definition of the correction term $\mathcal{B}_\theta$ introduced earlier in \eqref{modified-phase} was deliberately chosen to include the low-frequency cut off  $\rho(s^{-\frac{2}{n}}\xi)$, thereby restricting the modification to the genuinely relevant regime.
Accordingly, we separate the analysis into the two regimes
$|\xi| \les M^{\frac2n}$ and $|\xi|\gtrsim  M^{\frac2n}$, and more precisely, \eqref{eq:part-modification} is decomposed into
\begin{align}
\int_{t_{1}}^{t_{2}}\bigg|\sum_{ \substack{ \theta_1,\theta_2,\theta_3\in\{+,-\} \\ \mathbf{\Theta}=(\theta,\theta_1,\theta_2,\theta_3)\in \boldsymbol{\mathfrak{S}}_1 } }\mathcal N_{\mathbf{\Theta};L}(s,\xi)\rho\big( s^{-\frac2n} \xi  \big) - \frac{(2\pi)^3}\lam \big[ \partial_{s}\mathcal{B}_\theta(s,\xi)\big]\wh{\Psi_{\theta}}(s,\xi)\bigg|ds&\les\ve_{1}^{3}M^{-\de}\bra{\xi}^{-k}\label{eq:crucial-part}, \\ 
\sum_{ \substack{ \theta_1,\theta_2,\theta_3\in\{+,-\} \\ \mathbf{\Theta}=(\theta,\theta_1,\theta_2,\theta_3)\in \boldsymbol{\mathfrak{S}}_1 } }\int_{t_{1}}^{t_{2}}\left|\mathcal N_{\mathbf{\Theta};L}(s,\xi) \left(1-\rho\left( s^{-\frac2n} \xi \right) \right)\right| ds &\les\ve_{1}^{3}M^{-\de}\bra{\xi}^{-k}.\label{eq:decay-part}
\end{align}

\begin{proof}[Proof of \eqref{eq:crucial-part}]
In view of this observation, the contribution of the phase modification appears only in \eqref{eq:crucial-part}. 
Here, we may assume $|\xi|\les M^{\frac2n}$ and it suffices to estimate the integrand
\begin{align*}
\bigg|\sum_{ \substack{ \theta_1,\theta_2,\theta_3\in\{+,-\} \\ \mathbf{\Theta}=(\theta,\theta_1,\theta_2,\theta_3)\in \boldsymbol{\mathfrak{S}}_1 } } \mathcal N_{\mathbf{\Theta};L}(s,\xi)- \frac{(2\pi)^3}\lam \left[\partial_{s}\mathcal B_{\theta}(s,\xi) \right]\wh{\Psi_{\theta}}(s,\xi)\bigg| & \les \ve_{1}^{3} M^{-(1+\de)}\bra{\xi}^{-k}.
\end{align*}

First, let $\mathbf{\Theta}_1=(\theta,\theta,\theta_2,\theta_2)\in\boldsymbol{\mathfrak{S}}_1$ be fixed.
To carry out the phase modification, we further expand the phase as follows:
\begin{align*}
	q_{\mathbf{\Theta}_1}\freq & =\theta\Big(\braxi- \bra{\xi+\eta}\Big) + \thet\Big( \bra{\xi+\eta+\sigma} - \bra{\xi+\sigma}\Big)\\
& =\left(\theta\frac{-|\eta|^{2}-2\eta\cdot\xi}{\braxi+\bra{\xi+\eta}}+\thet\frac{|\eta|^{2}+2\eta\cdot(\xi+\sigma)}{\bra{\xi+\eta+\sigma} + \bra{\xi+\sigma}}\right)\\
	& =\eta\cdot\left(-\theta\frac{\xi}{\braxi}+\thet\frac{\xi+\sigma}{\bra{\xi+\sigma}}\right)+O\left(|\eta|^{2}\right)\\
	&=-\eta\cdot\boldsymbol{\mathcal{Z}}_{(\theta,\theta_2)}(\xi,\xi+\sigma)+O\left(|\eta|^{2}\right),
\end{align*}
where $\boldsymbol{\mathcal{Z}}_{(\theta,\theta_2)}$ is defined in \eqref{def:Z}.
Then, we approximate $\mathcal N_{\mathbf{\Theta}_1,L}$ by the following term
\begin{align*}
	\mathcal N_{\mathbf{\Theta}_1;L}^1(s,\xi) & :=\iint_{\R^2\times \R^2}e^{-is\eta\cdot\boldsymbol{\mathcal{Z}}_{(\theta,\theta_2)}(\xi,\xi+\sigma)} \rho_{\le L}(\eta)|\eta|^{-1}\Pi_{\theta}(\xi)\wh{\Psi_{\theta}}(s,\xi+\eta)\\
	&\hspace{4cm}\times\bra{\wh{\Psi_\thet}(s,\xi+\eta+\sigma),\wh{\Psi_\thet}(s,\xi+\sigma)}d\eta d\sigma.
\end{align*}
Indeed,
\begin{align*}
	& \left|\mathcal N_{\mathbf{\Theta}_1;L}(s,\xi)-\mathcal N_{\mathbf{\Theta}_1;L}^1(s,\xi)\right|\\
	& \les\iint_{\R^2\times \R^2}|s|\left| q_{\mathbf{\Theta}_1}(\xi,\eta,\sigma)+\eta\cdot\boldsymbol{\mathcal{Z}}_{(\theta,\theta_2)}(\xi,\xi+\sigma) \right||\eta|^{-1}\rho_{\le L}(\eta)\\
	& \hspace{2.2cm}\times\left|\wh{\Psi_{\theta}}(s,\xi+\eta)\bra{\wh{\Psi_\thet}(s,\xi+\eta+\sigma),\wh{\Psi_\thet}(s,\xi+\sigma)}\right|d\eta d\sigma\\
	& \les M\iint_{\R^2 \times \R^2}|\eta|\left|\rho_{\le L}(\eta)\wh{\Psi_{\theta}}(s,\xi+\eta)\bra{\wh{\Psi_\thet}(s,\xi+\eta+\sigma),\wh{\Psi_\thet}(s,\xi+\sigma)}\right|d\eta d\sigma\\
	& \les ML\big\|\rho_{\le L}\big\|_{L^1(\R^2)}\normo{\Psi_\thet(s)}_{L^{2}(\R^2)}^2\norm{\wh{\Psi_{\theta}}(s)}_{L^{\infty}(\R^2)}\\
	& \les \ve_{1}^{3} M^{-2 +3\de_1} \les \ve_1^3 M^{-(1+\de)}\bra{\xi}^{-k},
\end{align*}
where we used the size of $L$ in \eqref{eq:lzero}, the restriction $|\xi| \les M^{\frac 2n}$, and the condition $\delta\ll \delta_1\ll1$ in the last inequality. 
We remark that the Dirac null structures do not play any role in this case.

We further approximate $\mathcal N_{\mathbf{\Theta}_1;L}^1$ by $\mathcal N_{\mathbf{\Theta}_1;L}^{\,2}$, defined as follows:
\begin{align*}
\mathcal N_{\mathbf{\Theta}_1;L}^{\,2}(s,\xi)&:=\iint_{\R^2 \times \R^2}e^{-is\eta\cdot\boldsymbol{\mathcal{Z}}_{(\theta,\theta_2)}(\xi,\xi+\sigma)}\rho_{\le L}(\eta)|\eta|^{-1}\wh{\Psi_{\theta}}(s,\xi)\left|\wh{\Psi_\thet}(s,\xi+\sigma)\right|^{2}d\eta d\sigma.
\end{align*}
By setting $R=L^{-\frac12}$ and using the mean value theorem, we see that 
\begin{align*}
 \left|\wh{\Psi_\theta}(s,\zeta+\eta)-\wh{\Psi_\theta}(s,\zeta)\right|& \les\left|\wh{\rho_{>R}\Psi_\theta}(s,\zeta+\eta)-\wh{\rho_{>R}\Psi_\theta}(s,\zeta)\right|+\left|\wh{\rho_{\le R}\Psi_\theta}(s,\zeta+\eta)-\wh{\rho_{\le R}\Psi_\theta}(s,\zeta)\right|\\
& \les\normo{\wh{\rho_{>R}\Psi_\theta(s)}}_{L^{\infty}(\R^2)}+L\normo{\nabla\wh{\rho_{\le R}\Psi_\theta}(s)}_{L^{\infty}(\R^2)}\\
& \les R^{-1}\normo{x^{2}\Psi_\theta(s)}_{L_{x}^{2}(\R^2)}+ L\|\rho_{\le R}\|_{L^2(\R^2)} \normo{x\Psi_\theta(s)}_{L_{x}^{2}(\R^2)}\\
& \les L^{\frac12}M^{2\de_{0}},
\end{align*}
From this and \eqref{eq:lzero}, we estimate
\begin{align*}
	& \abs{ \mathcal N_{\mathbf{\Theta}_1;L}^1(s,\xi)-\mathcal N_{\mathbf{\Theta}_1;L}^{\,2}(s,\xi) }\\
	& \les\iint_{\R^2 \times \R^2}\rho_{\le L}(\eta)|\eta|^{-1}\abs{\wh{\Psi_{\theta}}(s,\xi+\eta)\bra{\wh{\Psi_\thet}(s,\xi+\eta+\sigma),\wh{\Psi_\thet}(s,\xi+\sigma)}-\wh{\Psi_{\theta}}(s,\xi)\left|\wh{\Psi_\thet}(s,\xi+\sigma)\right|^{2}}d\eta d\sigma\\
	& \les \ve_1^3 L^{\frac32}M^{2\de_{0}}
	\les \ve_1^3 M^{-\frac32+\frac32\delta_1+2\delta_0}
	\les \ve_{1}^{3} M^{-(1+\de)} \bra{\xi}^{-k}.
\end{align*}
After a change of variables, we rewrite $\mathcal N_{\mathbf{\Theta}_1;L}^{\,2}$ as 
\begin{align*}
\mathcal N_{\mathbf{\Theta}_1;L}^{\,2}(s,\xi)&=\iint_{\R^2 \times \R^2}e^{-is\eta\cdot\boldsymbol{\mathcal{Z}}_{(\theta,\theta_2)}(\xi,\sigma)}\rho_{\le L}(\eta)|\eta|^{-1}\wh{\Psi_{\theta}}(s,\xi)\left|\wh{\Psi_\thet}(s,\sigma)\right|^{2}d\eta d\sigma.
\end{align*}

Now it remains to show a final approximation
\begin{align*}
\bigg| \sum_{ \theta_2\in\{+,-\} } \mathcal{N}_{(\theta,\theta,\theta_2,\theta_2);L}^{\,2}(s,\xi) - \frac{(2\pi)^3}\lam \left[\partial_{s}\mathcal B_{\theta}(s,\xi)\right]\wh{\Psi_{\theta}}(s,\xi)\bigg| \les \ve_{1}^{3} M^{-(1+\de)}\bra{\xi}^{-k}.
\end{align*}
Recalling the definition of $\mathcal B_\theta$ in\eqref{modified-phase}, we deduce that
\begin{align}\begin{aligned}\label{FinalApprox2}
& \bigg| \sum_{ \theta_2\in\{+,-\} } \mathcal N_{(\theta,\theta,\theta_2,\theta_2);L}^{\,2}(s,\xi)-\frac{(2\pi)^3}\lam \left[\partial_{s}\mathcal B_{\theta}(s,\xi)\right]\wh{\Psi_{\theta}}(s,\xi)\bigg|\\
& \les \sum_{\theta_2\in\{+,-\}}\abs{\wh{\Psi_{\theta}}(s,\xi)} \\
&\qquad \times  \left|\iint_{\R^2 \times \R^2}e^{-is\eta\cdot\boldsymbol{\mathcal{Z}}_{(\theta,\theta_2)}(\xi,\sigma)}\rho_{\le L}(\eta)|\eta|^{-1}\abs{\wh{\Psi_\thet}(s,\sigma)}^{2}d\eta d\sigma-\frac{2\pi}{s}\int_{\R^2}|\boldsymbol{\mathcal{Z}}_{(\theta,\theta_2)}(\xi,\sigma)|^{-1}\abs{\wh{\Psi_\thet}(s,\sigma)}^{2}d\sigma\right|\\
& \les \sum_{\theta_2\in\{+,-\}} \abs{\wh{\Psi_\theta}(s,\xi)} \\
&\qquad\qquad \times  \left|\int_{\R^{2}}\left|\int_{\R^{2}}e^{-is\eta\cdot\boldsymbol{\mathcal{Z}}_{(\theta,\theta_2)}(\xi,\sigma)}\rho_{\le L}(\eta)|\eta|^{-1} d\eta    - (2\pi)\big|s\boldsymbol{\mathcal{Z}}_{(\theta,\theta_2)}(\xi,\sigma)\big|^{-1} \times\right|\abs{\wh{\Psi_\thet}(s,\sigma)}^{2}d\sigma\right|.
\end{aligned}\end{align}
By using the following identity
\[
\frac{2\pi}{\big|s\boldsymbol{\mathcal{Z}}_{(\theta,\theta_2)}(\xi,\sigma)\big|}=\lim_{A\to\infty}\int_{\R^2} e^{-is\eta\cdot\boldsymbol{\mathcal{Z}}_{(\theta,\theta_2)}(\xi,\sigma)}\rho_{\le A}(\eta) \frac{1}{|\eta|}d\eta,
\]
we estimate for $A \gg L$,
\begin{align}
	\begin{aligned}\label{eq:final-approx-1}
	&\left|\int_{\R^{2}}e^{-is\eta\cdot\boldsymbol{\mathcal{Z}}_{(\theta,\theta_2)}(\xi,\sigma)}|\eta|^{-1}\rho_{\le L}(\eta)  d\eta    - (2\pi)\big|s\boldsymbol{\mathcal{Z}}_{(\theta,\theta_2)}(\xi,\sigma)\big|^{-1} \right|\\
	&= \left|\int_{\R^{2}}e^{-is\eta\cdot\boldsymbol{\mathcal{Z}}_{(\theta,\theta_2)}(\xi,\sigma)}|\eta|^{-1}\left(\rho_{\le L}(\eta)-\rho_{\le A}(\eta)\right)d\eta\right|\\
	& =\big|s\boldsymbol{\mathcal{Z}}_{(\theta,\theta_2)}(\xi,\sigma)\big|^{-2}\left|\int_{\R^{2}}\left(\nabla_\eta^2e^{-is\eta\cdot\boldsymbol{\mathcal{Z}}_{(\theta,\theta_2)}(\xi,\sigma)}\right)|\eta|^{-1}\left(\rho_{\le L}(\eta)-\rho_{\le A}(\eta)\right)d\eta\right|\\
	&\hspace{3cm}{\Big \downarrow}\,\,{\text{as } A\to\infty}\\
	&\qquad M^{-2}\big|\boldsymbol{\mathcal{Z}}_{(\theta,\theta_2)}(\xi,\sigma)\big|^{-2}L^{-1} \les M^{-1-\de_1}\big|\boldsymbol{\mathcal{Z}}_{(\theta,\theta_2)}(\xi,\sigma)\big|^{-2}
	\end{aligned}
\end{align}
and
\begin{align}
\begin{aligned}\label{eq:final-approx-2}
&\left|\int_{\R^{2}}e^{-is\eta\cdot\boldsymbol{\mathcal{Z}}_{(\theta,\theta_2)}(\xi,\sigma)}|\eta|^{-1}\rho_{\le L}(\eta)  d\eta    - (2\pi)\big|s\boldsymbol{\mathcal{Z}}_{(\theta,\theta_2)}(\xi,\sigma)\big|^{-1} \right| \\
&\qquad\les  L + \big|s\boldsymbol{\mathcal{Z}}_{(\theta,\theta_2)}(\xi,\sigma)\big|^{-1}\\
&\qquad \sim M^{-1+\delta_1} \max \left(1 ,M^{-\delta_1}\big|\boldsymbol{\mathcal{Z}}_{(\theta,\theta_2)}(\xi,\sigma)\big|^{-1} \right).
\end{aligned}
\end{align}
Interpolating \eqref{eq:final-approx-1} and \eqref{eq:final-approx-2}, we have
\begin{align*}
\left|\int_{\R^{2}}e^{-is\eta\cdot\boldsymbol{\mathcal{Z}}_{(\theta,\theta_2)}}|\eta|^{-1}\rho_{\le L}(\eta)  d\eta    - (2\pi)\big|s\boldsymbol{\mathcal{Z}}_{(\theta,\theta_2)}(\xi,\sigma)\big|^{-1} \right| \les M^{-1-\frac12\de_1}\big|\boldsymbol{\mathcal{Z}}_{(\theta,\theta_2)}(\xi,\sigma)\big|^{-\frac32}.
\end{align*}
Plugging these bounds into \eqref{FinalApprox2}, we obtain 
\begin{align*}
&\left|\sum_{ \theta_2\in\{+,-\} } \mathcal 
 N_{(\theta,\theta,\theta_2,\theta_2);L}^{\,2}(s,\xi)-\frac{(2\pi)^3}\lam \left[\partial_{s}\mathcal B_\theta(s,\xi)\right]\wh{\Psi_\theta}(s,\xi)\right| \\ 
&\qquad \les \sum_{ \theta_2\in\{+,-\} } M^{-1-\frac12\de_1}\abs{\wh{\Psi_\theta}(s,\xi)}\left|\int_{\R^{2}}\big|\boldsymbol{\mathcal{Z}}_{(\theta,\theta_2)}(\xi,\sigma)\big|^{-\frac32}\abs{\wh{\Psi_\thet}(s,\sigma)}^{2}d\sigma\right|.
\end{align*}
From the lower bound of $\big|\boldsymbol{\mathcal{Z}}_{(\theta,\theta_2)}(\xi,\sigma)\big|$ in \eqref{ineq:Phase null structure}, we obtain under the a priori assumption \eqref{assumption-apriori} that
\begin{align*}
\left|\int_{\R^{2}}\big|\boldsymbol{\mathcal{Z}}_{(\theta,\theta_2)}(\xi,\sigma)\big|^{-\frac32}\abs{\wh{\Psi_\thet}(s,\sigma)}^{2}d\sigma\right|\les  \ve_1^2,	
\end{align*}
which completes the proof of \eqref{eq:crucial-part}.
\end{proof}

\begin{proof}[Proof of \eqref{eq:decay-part}]
It suffices to show that, for $\mathbf{\Theta}_1\in\boldsymbol{\mathfrak{S}}_1$,
\begin{align*}
\Big|  \mathcal N_{\mathbf{\Theta}_1;L} (s,\xi)\Big|\les\ve_{1}^{3}M^{-(1+\de)}\bra{\xi}^{-k} \text{ whenever } |\xi| \gtrsim M^{\frac2n}.
\end{align*}
We further decompose $\mathcal N_{\mathbf{\Theta}_1;L}$ dyadically as follows:
\begin{align*}
\mathcal N_{\mathbf{\Theta}_1;L}(s,\xi)=\sum_{L_1\in 2^{\Z}\,:\,L_{1}\le 2^{10}L}\mathcal N_{\mathbf{\Theta}_1;L,L_1}(s,\xi),
\end{align*}
where
\begin{align*}
\mathcal N_{\mathbf{\Theta}_1;L,L_1}(s,\xi)	& =\int\!\!\!\! \int_{\R^2 \times \R^2}e^{isq_{\mathbf{\Theta}_1}\freq}\rho_{\le L}(\eta)|\eta|^{-1} \Pi_{\theta}(\xi)\Pi_{\theta}(\xi+\eta)\wh{\Psi_{\theta}}(s,\xi+\eta)\\
&\hspace{6cm}\times\bra{\wh{\Psi_\thet}(s,\xi+\eta+\sigma),\wh{\Psi_\thet}(s,\xi+\sigma)} d\eta d\sigma.
\end{align*}
In view of the a priori assumption \eqref{assumption-apriori}, we estimate
\begin{align*}
& \left|	\mathcal N_{\mathbf{\Theta}_1;L,L_1}(s,\xi)\right|\\
& \;\;\les L_{1}^{-1}\bra{\xi}^{-k}\iint_{\R^2 \times \R^2}\left|\bra{\xi}^k\chi_{L_{1}}(\eta)\rho_{\le L}(\eta)\wh{\Psi_{\theta}}(s,\xi+\eta)\bra{\wh{\Psi_\thet}(s,\xi+\eta+\sigma),\wh{\Psi_\thet}(s,\xi+\sigma)}\right|d\eta d\sigma\\
& \;\;\les L_{1}^{-1}\bra{\xi}^{-k}\|\chi_{L_1}\|_{L^1(\R^2)}\left\|\bra{\xi}^k\wh{\Psi_{\theta}}(s,\xi)\,\right\|_{L_{\xi}^{\infty}(\R^2)}\|\psi_\thet(s)\|_{H^{n}(\R^2)}\|\psi_\thet(s)\|_{H^n(\R^2)}\\
	& \;\;\les \ve_{1}^{3} L_{1}M^{2\de_{0}}\bra{\xi}^{-k}.
\end{align*}
On the other hand, by H\"older's inequality, we also have 
\begin{align*}
\left|\mathcal N_{\mathbf{\Theta}_1;L,L_1}(s,\xi)\right|\;\;\les  L_1^{-1}\|\chi_{L_1}\|_{L^2(\R^2)}\bra{\xi}^{-n}\|\psi_{\theta}(s)\|_{H^{n}(\R^2)}\|\psi_\thet(s)\|_{H^{n}(\R^2)}^2 \les \ve_{1}^{3} M^{-2}M^{3\de_{0}},
\end{align*}
where we used the support condition $|\xi|\gtrsim M^{\frac2n}$. Combining these two bounds gives
\begin{align*}
	\sum_{L_1\in 2^{\Z}\,:L_{1}\le 2^{10}L}\Big|\mathcal N_{\mathbf{\Theta}_1;L,L_1}(s,\xi)\Big|& \les \ve_1^3 \left( \sum_{L_{1}\le M^{-2}}L_{1}M^{2\de_{0}}\bra{\xi}^{-k}+\sum_{L_{1}>M^{-2}}M^{-2+3\de_{0}}\right) \\
	& \les \ve_{1}^{3} M^{-(1+\de_{0})}\bra{\xi}^{-k}.
\end{align*}
\end{proof}

\subsection{Proof of \eqref{eq:error}}
We now turn to the low-frequency regime corresponding to the remaining sign configurations, namely when $\mathbf{\Theta} \notin \boldsymbol{\mathfrak{S}}_1$. Let $\mathbf{\Theta} = (\theta,\theo,\thet,\theth) \notin \boldsymbol{\mathfrak{S}}_1 $ be fixed. Since either $\theta \neq \theo$ or $\thet \neq \theth$, the Dirac null structures
arising from $\Pi_\theta(D)\Pi_\theo(D)$ or $\Pi_\thet(D)\Pi_\theth(D)$ play a crucial role in compensating
the singularity, without any cancellation effect from the phase modification. For instance, if $\theta \neq \theta_1$, by the Dirac null structure \eqref{eq:null}, we estimate
\begin{align*}
	 \left|\mathcal N_{{\mathbf{\Theta}};L}(s,\xi)\right|
	& \les\int_{\R^{2+2}}\Big|\Pi_{\theta}(\xi)\Pi_{\theo}(\xi+\eta)\Big| \left|\rho_{\le L}(\eta) \right| |\eta|^{-1}\\
	& \qquad\qquad\qquad\times\left|\wh{\Psi_{\theo}}(s,\xi+\eta)\bra{\wh{\Psi_{\theta_2}}(s,\xi+\eta+\sigma),\wh{\Psi_{\theth}}(s,\xi+\sigma)}\right|d\eta d\sigma\\
	& \les L^{2}\normo{\wh{\Psi_{\theo}}(s)}_{L^{\infty}(\R^2)}\normo{\Psi_{\thet}(s)}_{L^{2}(\R^2)}\normo{\Psi_{\theta_3}(s)}_{L^{2}(\R^2)}\\
	& \les M^{-2+2\de_1}\ve_{1}^{3}.
\end{align*}
The remaining cases can be estimated similarly.

\subsection{Proof of \eqref{eq:part-scattering}}
We now prove the estimate for the high-frequency contribution
\begin{align*}
\sum_{L_1>M^{-1+\delta_1 }}\left|\int_{t_{1}}^{t_{2}}e^{-i\mathcal{B}_{\theta}(s,\xi)}\mathcal N_{\mathbf{\Theta};L_1}(s,\xi)ds\right| & \les\ve_{1}^{3}M^{-\de}\bra{\xi}^{-k},  \text{ for all } \mathbf{\Theta}.
\end{align*}
We remark that, in the following analysis, the Dirac null structure does not play any role.

We further localize the frequencies as follows:
\begin{align*}
	\mathcal N_{\mathbf{\Theta};L_1} & =\sum_{\widetilde{\mathbf{N}}\in (2^{\N\cup\{0\}})^3}\mathcal N_{\mathbf{\Theta}; L_1, \widetilde{\mathbf{N}}}(s,\xi),\\
	\mathcal N_{\mathbf{\Theta}; L_1, \widetilde{\mathbf{N}}}(s,\xi) & :=\Pi_{\theta}(\xi)\iint_{\R^2 \times \R^2}e^{isq_{\mathbf{\Theta}}\freq}\chi_{ L_1}(\eta)|\eta|^{-1} \wh{\Psi_{\theo;N_1}}(s,\xi+\eta)\\
	&\hspace{4cm}\times\bra{\wh{\Psi_{\thet;N_2}}(s,\xi+\eta+\sigma),\wh{\Psi_{\theth;N_3}}(s,\xi+\sigma)}d\eta d\sigma,	
\end{align*}
where  $\widetilde{\mathbf{N}}:=(N_{1},N_{2},N_{3}) \in (2^{\N \cup \{0\} })^3$ is a 3-tuple of dyadic numbers \footnote{
Unlike the previous section, where we used the multi-index
$\mathbf N=(N_0,N_1,N_2,N_3)$ including the output frequency,
only the input frequencies are dyadically decomposed and summed over here.
Since we only derive pointwise estimates in $\xi$, the output frequency is treated separately and will later be associated with a dyadic scale $N_0\sim |\xi|$.
} and $\Psi_{\thej;N_j} = P_{N_j} \Psi_\thej$. 
Then, we prove that 
\begin{align}
	\sum_{L_1\in 2^{\Z}\,:L_1>M^{-1+\delta_1 }} \sum_{\widetilde{\mathbf{N}}\in(2^{\N\cup\{0\}})^3} \left| \int_{t_1}^{t_2}e^{-i \mathcal B_\theta(s,\xi)}\mathcal N_{\mathbf{\Theta}; L_1, \widetilde{\mathbf{N}}}(s,\xi) ds\right|\les\ve_{1}^{3}M^{-\de}\bra{\xi}^{-k}.\label{eq:goal-high}
\end{align}
Let $\widetilde{N}_{\max}:=\max{(N_{1},N_{2},N_{3})}$.
Under the a priori assumption \eqref{assumption-apriori}, by H\"older inequality, we readily obtain 
\[
\left|\mathcal N_{\mathbf{\Theta}; L_1, \widetilde{\mathbf{N}}}(s,\xi)\right|\les \prod_{j=1}^{3}\normo{\wh{\Psi_{\thej;N_j}}(s)}_{L^{2}(\R^2)} \lesssim  M^{3\de_{0}}\ve_1^3 \prod_{j=1}^{3}\bra{N_{j}}^{-n},
\]
which implies \eqref{eq:goal-high} when the summation runs over $\widetilde{N}_{\max} \ge M^{\frac2n}$.
Therefore it suffices to show that 
\begin{align*}
\sum_{ L_1\in 2^{\Z} : L_1>M^{-1+\delta_1 }} \sum_{ \widetilde{\mathbf{N}}\in (2^{\N\cup\{0\}})^3\, : \, N_{1},N_{2},N_{3}\le M^{\frac2n}  } \left| \int_{t_1}^{t_2}e^{-i \mathcal B_\theta(s,\xi)}\mathcal N_{\mathbf{\Theta}; L_1, \widetilde{\mathbf{N}}}(s,\xi) ds\right|\les\ve_{1}^{3}M^{-\de}\bra{\xi}^{-k}.
\end{align*}

We further localize with respect to the $\sigma$ variable:
\begin{align*}
\mathcal N_{\mathbf{\Theta}; L_1, \widetilde{\mathbf{N}}}(s,\xi)&= \sum_{L_2\in 2^{\Z}}\mathcal N_{\mathbf{\Theta}; \widetilde{\mathbf{N}}\mathbf{L}}(s,\xi),  
\end{align*} 
where \begin{align*} 
\mathcal N_{\mathbf{\Theta}; \widetilde{\mathbf{N}}\mathbf{L}}(s,\xi) & := \iint_{\R^{2}\times\R^2}e^{isq_{\mathbf{\Theta}}\freq}\chi_{L_1}(\eta) \chi_{L_2}(\sigma) |\eta|^{-1}\Pi_{\theta}(\xi)\wh{\Psi_{\theo;N_1}}(s,\xi+\eta)\\
	&\hspace{4cm}\times\bra{\wh{\Psi_{\thet;N_2}}(s,\xi+\eta+\sigma),\wh{\Psi_{\theth;N_3}}(s,\xi+\sigma)}d\eta d\sigma,
\end{align*}
where $\mathbf{L}=(L_1,L_2)\in (2^{\Z})^2$ and $\widetilde{\mathbf{N}}\mathbf{L}:=(N_1,N_2,N_3,L_1,L_2)$. 
By the frequency relation, the restriction  $\widetilde{N}_{\max} \le M^{\frac2n}$ also implies $|\xi|, L_2\les M^{\frac2n}$. 
Now, let $|\xi|\sim N_0$ for some $N_0\in 2^{\N\cup\{0\}}$. 
Collecting all frequency restrictions on $\widetilde{\mathbf{N}}=(N_1,N_2,N_3)$ and $\mathbf{L}=(L_1,L_2)$, we define the constraint set $ \boldsymbol{\Omega} (M,N_0)$ as follows:
\begin{align*}
\boldsymbol{\Omega}(M,N_0):=\big\{  &(N_1,N_2,N_3)\in (2^{\N\cup\{0\}})^3,(L_1,L_2)\in (2^{\Z})^2 : N_{1},N_{2},N_{3}\le M^{\frac{2}{n}},\\  &\qquad N_0 \les \max(N_1,N_2,N_3), \;
    L\le L_1 \le \max(N_2,N_3), \;  L_2\le \max(N_1,N_2)\big\} .
\end{align*}
Thus, our goal is to show that for $|\xi|\sim N_0\in 2^{\N\cup\{0\}}$,
\begin{align}
\sum_{ \widetilde{\mathbf{N}}\mathbf{L} \in  \boldsymbol{\Omega} (M,N_0)} \left| \int_{t_1}^{t_2} e^{-i \mathcal B_\theta(s,\xi)}\mathcal N_{\mathbf{\Theta}; \widetilde{\mathbf{N}}\mathbf{L}}(s,\xi) ds\right|\les\ve_{1}^{3}M^{-\de}N_0^{-k}, \;\; t_1,t_2 \in\left[M/4, 2M \right]. \label{eq:goal2-high}
\end{align}

For the proof of \eqref{eq:goal2-high}, we begin with the following simple observation. Considering the singularity together with dyadically localized cut-off functions  in $\mathcal N_{\mathbf{\Theta}; \widetilde{\mathbf{N}}\mathbf{L}}$, one has 
\begin{align*}
\left| \iint_{\R^{2}\times \R^2} \chi_{L_1}(\eta)\chi_{L_2}(\sigma)  |\eta|^{-1}    d\eta d\sigma\right| \sim L_1L_2^{\,2}.
\end{align*}
Consequently, the desired bound \eqref{eq:goal2-high} is immediate for the sum over those indices satisfying $L_1L_2^{\,2} \le M^{-1-2\de}$. Indeed, by H\"older inequality together with the  a priori assumption \eqref{assumption-apriori}, we estimate
\begin{align*}
	&\sum_{\substack{ \widetilde{\mathbf{N}}\mathbf{L} \in \boldsymbol{\Omega}  (M,N_0)\\ L_1L_2^{\,2} \le M^{-(1+2\delta)}}}
	|\mathcal N_{\mathbf{\Theta}; \widetilde{\mathbf{N}}\mathbf{L}}(s,\xi)| \\ 
    &\qquad \les\sum_{\substack{ \widetilde{\mathbf{N}}\mathbf{L} \in \boldsymbol{\Omega}  (M,N_0)\\ L_1L_2^{\,2} \le M^{-(1+2\delta)}}}L_1^{-1}\|\chi_{L_1}\|_{L^1(\R^2)}\|\chi_{L_2}\|_{L^1(\R^2)}\prod_{j=1}^{3}N_j^{-k}\normo{\bra{\xi}^{k}\wh{\Psi_{\thej;N_j}}(s,\xi)}_{L_{\xi}^{\infty}(\R^2)}\\
	&\qquad\les\ve_1^3\sum_{\substack{ \widetilde{\mathbf{N}}\mathbf{L} \in \boldsymbol{\Omega}  (M,N_0)\\ L_1L_2^{\,2} \le M^{-(1+2\delta)}}}L_1L_2^{\,2}\prod_{j=1}^{3}N_j^{-k}\\
	&\qquad \les\ve_{1}^{3}M^{-(1+\de)}N_0^{\,-k}.
\end{align*}
This establishes the contribution from the range $L_1L_2^{\,2} \le M^{-1-2\de}$. 
We now turn to the complementary case $L_1L_2^{\,2}\ge M^{-(1+2\delta)}$. Equivalently, we are left with a summation over the following constraint set 
\begin{align*}
\boldsymbol\Gamma(M,N_0) := \boldsymbol\Omega (M,N_0) \cap\left\{ L_1L_2^{\,2}\ge M^{-(1+2\de)}\right\}.
\end{align*}
For the remainder of this section, we are devoted to showing for $|\xi|\sim N_0$,
\begin{align}
\sum_{ \widetilde{\mathbf{N}}\mathbf{L} \in  \boldsymbol\Gamma(M,N_0)} \left| \int_{t_1}^{t_2} e^{-i \mathcal B_\theta(s,\xi)} \mathcal N_{\mathbf{\Theta}; \widetilde{\mathbf{N}}\mathbf{L}}(s,\xi) ds\right|\les\ve_{1}^{3}M^{-\de}N_0^{-k}, \;\; t_1,t_2 \in\left[M/4, 2M \right]. \label{eq:goal2-high2}
\end{align}
To prove \eqref{eq:goal2-high2}, we divide the analysis into three cases according to the sign configuration of
$(\theta_1,\theta_2,\theta_3)$.
Each case relies on a different resonance structure.
We emphasize that the sign $\theta$ plays no role in the arguments below.
\begin{align*} \left\{ \begin{aligned}
	\textbf{Case 1: } & \theta_1=\theta_2, \\ 
	\textbf{Case 2: } & \theta_1\neq\theta_2 \text{ and } \theta_2\neq\theta_3,\\
    \textbf{Case 3: } & \theta_1\neq\theta_2 \text{ and } \theta_2=\theta_3.  
    \end{aligned}\right.
\end{align*}

\subsubsection{Estimates for Case 1.} One representative configuration $(\theta,\theta,\theta,\theta)$ was previously analyzed in a joint work by the present authors and a third collaborator \cite{kly2023}, and the same method applies to the remaining configurations. For completeness and the reader's convenience, we briefly recall the main argument.

Let
$\mathbf{\Xi}_1=(\theta,\theta_1,\theta_1,\theta_3)$
be a fixed sign configuration belonging to Case~1. Then, as a key observation, we note from \eqref{ineq:Phase lower bound} that the following space non-resonance estimate holds
\begin{align}
\begin{aligned}\label{lower bound of nabla q}
\left|\nabla_{\eta}q_{\mathbf{\Xi}_1}(\xi,\eta,\sigma)\right|  
    &=\left|\frac{\xi+\eta+\sigma}{\bra{\xi+\eta+\sigma}}-\frac{\xi+\eta}{\bra{\xi+\eta}}\right| \\ 
    &\gtrsim \frac{|\sigma|}{\max(\langle \xi+\eta+\sigma\rangle, \langle\xi+\eta\rangle) \min(\langle \xi+\eta+\sigma\rangle, \langle\xi+\eta\rangle)^2}.
\end{aligned}\end{align}
With this lower bound at hand, we perform an integration by parts in $\eta$ twice, with a slight abuse of notation, which leads to 
\begin{align*}  
|\mathcal{N}_{\mathbf{\Xi}_1;\widetilde{\mathbf{N}}\mathbf{L}}(s,\xi)|  
&\le |\mathcal{N}_{\mathbf{\Xi}_1;\widetilde{\mathbf{N}}\mathbf{L}}^1(s,\xi)|+|\mathcal{N}_{\mathbf{\Xi}_1;\widetilde{\mathbf{N}}\mathbf{L}}^2(s,\xi)|+|\mathcal{N}_{\mathbf{\Xi}_1;\widetilde{\mathbf{N}}\mathbf{L}}^3(s,\xi)|,\\
\mathcal{N}_{\mathbf{\Xi}_1;\widetilde{\mathbf{N}}\mathbf{L}}^1(s,\xi)  &:=\frac{1}{s^{2}}\int\!\!\!\!\int_{\R^2\times\R^2} e^{isq_{\mathbf{\Xi}_1}(\xi,\eta,\sigma)} \frac{\nabla_{\eta}q_{\mathbf{\Xi}_1}}{|\nabla_{\eta}q_{\mathbf{\Xi}_1}|^{2}}\mathcal{Q}_{\mathbf{\Xi}_1;\mathbf{L}}(\xi,\eta,\sigma)
 \\ &\qquad \times \nabla_{\eta}^{2}\Big( \wh{\Psi_{\theo;N_1}}(s,\xi+\eta)\bra{\wh{\Psi_{\theo;N_2}}(s,\xi+\eta+\sigma),\wh{\Psi_{\theth;N_3}}(s,\xi+\sigma)} \Big)d\eta d\sigma, \\
 \mathcal{N}_{\mathbf{\Xi}_1;\widetilde{\mathbf{N}}\mathbf{L}}^2(s,\xi) & :=\frac{1}{s^{2}}\int\!\!\!\!\int_{\R^2\times\R^2} e^{isq_{\mathbf{\Xi}_1}(\xi,\eta,\sigma)} \nabla_{\eta} \Big( \frac{\nabla_{\eta}q_{\mathbf{\Xi}_1}}{|\nabla_{\eta}q_{\mathbf{\Xi}_1}|^{2}} \mathcal{Q}_{\mathbf{\Xi}_1;\mathbf{L}} \Big) (\xi,\eta,\sigma)
 \\ &\qquad \times \nabla_{\eta}\Big( \wh{\Psi_{\theo;N_1}}(s,\xi+\eta)\bra{\wh{\Psi_{\theo;N_2}}(s,\xi+\eta+\sigma),\wh{\Psi_{\theth;N_3}}(s,\xi+\sigma)} \Big)d\eta d\sigma, \\
    \mathcal{N}_{\mathbf{\Xi}_1;\widetilde{\mathbf{N}}\mathbf{L}}^3(s,\xi) &:=\frac{1}{s^{2}}\int\!\!\!\!\int_{\R^2\times\R^2} e^{isq_{\mathbf{\Xi}_1}\freq}\nabla_{\eta}\Big(\frac{\nabla_{\eta}q_{\mathbf{\Xi}_1}}{|\nabla_{\eta}q_{\mathbf{\Xi}_1}|^{2}}\nabla_{\eta} \mathcal{Q}_{\mathbf{\Xi}_1;\mathbf{L}}\Big)\freq
 \\ &\qquad \times  \wh{\Psi_{\theo;N_1}}(s,\xi+\eta)\bra{\wh{\Psi_{\theo;N_2}}(s,\xi+\eta+\sigma),\wh{\Psi_{\theth;N_3}}(s,\xi+\sigma)} d\eta d\sigma, \\
\end{align*}
    where 
\begin{align*}
\mathcal{Q}_{\mathbf{\Xi}_1;\mathbf{L}}(\xi,\eta,\sigma) & :=\frac{\nabla_{\eta}q_{\mathbf{\Xi}_1}}{|\nabla_{\eta}q_{\mathbf{\Xi}_1}|^{2}}(\xi,\eta,\sigma)|\eta|^{-1}\Pi_{\theta}(\xi)\Pi_{\theta_1}(\xi+\eta)\chi_{L_1}(\eta)\chi_{L_2}(\sigma).
\end{align*}
Thanks to the lower bound \eqref{lower bound of nabla q}, we deduce that the symbols obey the following bounds:
\begin{align}\begin{aligned}\label{eq:k_1 bound}
&\sup_{|\xi|\sim N_0}\left\| \frac{\nabla_{\eta}q_{\mathbf{\Xi}_1}}{|\nabla_{\eta}q_{\mathbf{\Xi}_1}|^{2}}\mathcal{Q}_{\mathbf{\Xi}_1;\mathbf{L}}\rho_{\widetilde{\mathbf{N}}}(\xi,\eta,\sigma)\right\|_{\textup{CM}_{\eta,\sigma}[(\R^2)^2]} \\ 
&\;:=\sup_{|\xi|\sim N_0}\left\Vert \iint_{(\R^2)^2}\frac{\nabla_{\eta}q_{\mathbf{\Xi}_1}}{|\nabla_{\eta}q_{\mathbf{\Xi}_1}|^{2}}\mathcal{Q}_{\mathbf{\Xi}_1;\mathbf{L}}\rho_{\widetilde{\mathbf{N}}}(\xi,\eta,\sigma) e^{iy\cdot\eta}e^{iz\cdot\sigma} d\eta d\sigma\right\Vert _{L_{y,z}^{1}((\R^2)^2)} \\ 
&\;\;\lesssim L_1^{-1}L_2^{-2}\max(N_1,N_2)^2 \min(N_1,N_2)^{12},
\end{aligned}\end{align}
where 
\begin{align*}
\rho_{\widetilde{\mathbf{N}}}(\xi,\eta,\sigma)=
    \rho_{N_1}(\xi+\eta)\rho_{N_2}(\xi+\eta+\sigma)\rho_{N_3}(\xi+\sigma),
\end{align*}
and similarly,
\begin{align}\label{eq:k_2 bound}
\sup_{|\xi|\sim N_0}\left\| \nabla_{\eta} \Big( \frac{\nabla_{\eta}q_{\mathbf{\Xi}_1}}{|\nabla_{\eta}q_{\mathbf{\Xi}_1}|^{2}} \mathcal{Q}_{\mathbf{\Xi}_1;\mathbf{L}} \Big) \rho_{\widetilde{\mathbf{N}}}(\xi,\eta,\sigma)\right\|_{\textup{CM}_{\eta,\sigma}[(\R^2)^2]} 
\lesssim L_1^{-2}L_2^{-2}\max(N_1,N_2)^2\min(N_1,N_2)^{12}.
\end{align}
In addition, the following pointwise bound holds
\begin{align}\label{eq:pw-bound-k3}
\sup_{|\xi|\sim N_0}\left\| \nabla_{\eta}\Big(\frac{\nabla_{\eta}q_{\mathbf{\Xi}_1}}{|\nabla_{\eta}q_{\mathbf{\Xi}_1}|^{2}}\nabla_{\eta} \mathcal{Q}_{\mathbf{\Xi}_1;\mathbf{L}}\Big)\rho_{\widetilde{\mathbf{N}}}(\xi,\eta,\sigma)\right\|_{L_{\eta,\sigma}^\infty((\R^2)^2)} \lesssim L_1^{-3}L_2^{-2}\max(N_1,N_2)^2\min(N_1,N_2)^{4}.
\end{align}
By applying the operator inequality \eqref{eq:coif-1-2} together with the multiplier bound \eqref{eq:k_1 bound}, we first estimate 
\begin{align*}
 &\sum_{ \widetilde{\mathbf{N}}\mathbf{L} \in  \boldsymbol\Gamma(M,N_0)}|\mathcal{N}_{\mathbf{\Xi}_1;\widetilde{\mathbf{N}}\mathbf{L}}^1(s,\xi)|\\
 & \les M^{-2}\sum_{ \widetilde{\mathbf{N}}\mathbf{L} \in  \boldsymbol\Gamma(M,N_0)}L_1^{-1}L_2^{-2}\max(N_1,N_2)^2 \min(N_1,N_2)^{12} \\ 
 &\quad\times \normo{\psi_{\theta_3;N_3}}_{L^{\infty}(\R^2)} \Big( 
\normo{x^{2}\Psi_{\theta_1;N_1}(s)}_{L_x^{2}(\R^2)}\normo{\Psi_{\theta_1;N_{2}}(s)}_{L^{2}(\R^2)}
\\ &\qquad\qquad\qquad  + \normo{x\Psi_{\theta_1;N_1}(s)}_{L_x^{2}(\R^2)}\normo{x\Psi_{\theta_1;N_2}(s)}_{L_x^{2}(\R^2)} +\normo{\Psi_{\theta_1;N_1}}_{L^{2}(\R^2)}\normo{x^{2}\Psi_{\theta_1;N_2}}_{L_x^{2}(\R^2)} \Big) \\
 & \les \ve_1^3 M^{-2+2\delta+2\delta_0}\sum_{ \substack{N_{1},N_{2},N_{3}\le M^{\frac{2}{n}} \\   L\le L_1 \le \max(N_2,N_3)}}\max(N_1,N_2)^2 \min(N_1,N_2)^{12}  N_3^{-k}\\
 &\les  \ve_1^3 M^{-2+3\delta+2\delta_0} N_0^{14}  \les\ve_{1}^{3}M^{-(1+\de)}N_0^{-k},
\end{align*}
where we used the constraint
$L_1L_2^{\,2}\ge M^{-(1+2\delta)}$
from the definition of $\boldsymbol\Gamma(M,N_0)$. In the last step, and throughout the rest of this subsection, we repeatedly use the hierarchy
\[
0<\delta\ll\delta_0\ll\delta_1\ll1,
\]
the frequency relation
\[
N_0\lesssim \max(N_1,N_2,N_3)\le M^{2/n},
\]
and the fact that $n$ is sufficiently large. These allow us to absorb all harmless dyadic summation losses (including logarithmic losses) and simultaneously recover the required decay factor $N_0^{-k}$.
Here we also used \eqref{frequency localized with x weight} and the following second-order weighted frequency-localized estimate.
Suppose that $\psi$ satisfies the a priori assumption
\eqref{assumption-apriori}
for some
$\varepsilon_1>0$.
Then, for every dyadic number
\(N\in2^{\mathbb N\cup\{0\}}\)
and
\(\theta\in\{+,-\}\),
\begin{align}
\label{frequency localized with x^2 weight}
\|x^2\Psi_{\theta;N}(s)\|_{L_x^2(\R^2)}
\lesssim
\langle s\rangle^{2\delta_0}\varepsilon_1.
\end{align}
Indeed,
\begin{align*}
\|x^2\Psi_{\theta;N}(s)\|_{L_x^2(\R^2)}
&\approx
\|\nabla^2\widehat{\Psi_{\theta;N}}(s)\|_{L^2(\R^2)} \\ 
&\lesssim 
\|(\nabla^2\rho_N)\widehat{\Psi_\theta}(s)\|_{L^2(\R^2)}
+\|(\nabla\rho_N)\nabla\widehat{\Psi_\theta}(s)\|_{L^2(\R^2)}
+
\|P_N(x^2\Psi_\theta)(s)\|_{L_x^2(\R^2)}
\lesssim
\langle s\rangle^{2\delta_0}\varepsilon_1.
\end{align*}
For the second term, the same argument, combined with the multiplier bound
\eqref{eq:k_2 bound}, yields
\begin{align*}
 &\sum_{ \widetilde{\mathbf{N}}\mathbf{L} \in  \boldsymbol\Gamma(M,N_0)}|\mathcal{N}_{\mathbf{\Xi}_1;\widetilde{\mathbf{N}}\mathbf{L}}^{\,2}(s,\xi)|\\
 & \les M^{-2}\sum_{ \widetilde{\mathbf{N}}\mathbf{L} \in  \boldsymbol\Gamma(M,N_0)}L_1^{-2}L_2^{-2}\max(N_1,N_2)^2\min(N_1,N_2)^{12} \normo{\psi_{\theta_3;N_3}(s)}_{L^{\infty}(\R^2)}\\ 
 &\qquad\times \left( \normo{\Psi_{\theta_1;N_{1}}(s)}_{L^{2}(\R^2)} \normo{x\Psi_{\theta_1;N_2}(s)}_{L_x^{2}(\R^2)} + \normo{x\Psi_{\theta_1;N_1}(s)}_{L_x^{2}(\R^2)}\normo{\Psi_{\theta_1;N_{2}}(s)}_{L^{2}(\R^2)} \right)\\
 & \les \ve_1^3 M^{-2+2\delta+2\delta_0}\sum_{ \substack{ N_{1},N_{2},N_{3}\le M^{\frac{2}{n}} \\   L\le L_1 \le \max(N_2,N_3)}}L_1^{-1}\max(N_1,N_2)^2 \min(N_1,N_2)^{12} N_3^{-k}\big( N_1^{-n}+N_2^{-n}\big ) \\
  & \les \ve_1^3 M^{-1-\delta_1+2\delta+2\delta_0}\sum_{  N_{1},N_{2},N_{3}\le M^{\frac{2}{n}} }\max(N_1,N_2)^2 \min(N_1,N_2)^{12}  N_3^{-k}\big( N_1^{-n}+N_2^{-n}\big )\\
 & \les \ve_1^3 M^{-1-\delta_1+2\delta+2\delta_0} N_0^{14}  \les\ve_{1}^{3}M^{-(1+\de)}N_0^{-k}.
 \end{align*}
Finally, arguing exactly as above but using the pointwise bound
\eqref{eq:pw-bound-k3} instead, we obtain
\begin{align*}
 &\sum_{ \widetilde{\mathbf{N}}\mathbf{L} \in  \boldsymbol\Gamma(M,N_0)}|\mathcal{N}_{\mathbf{\Xi}_1;\widetilde{\mathbf{N}}\mathbf{L}}^3(s,\xi)|\\
    & \les M^{-2}\sum_{ \widetilde{\mathbf{N}}\mathbf{L} \in  \boldsymbol\Gamma(M,N_0)}L_1^{-3}L_2^{-2}\max(N_1,N_2)^2\min(N_1,N_2)^{4}\|\chi_{L_1}\|_{L^{1}(\R^2)}\|\chi_{L_2}\|_{L^{1}(\R^2)}\prod_{j=1}^3   \normo{\wh{\Psi_{\theta_j;N_{j}}}(s)}_{L^{\infty}(\R^2)}\\ 
    & \les\ve_{1}^{3}M^{-2}\sum_{\widetilde{\mathbf{N}}\mathbf{L}\in\boldsymbol{\Gamma}(M,N_0)}L_1^{-1}\max(N_1,N_2)^2\min(N_1,N_2)^{4}
    \prod_{j=1}^3N_j^{-k} \\
    & \les\ve_{1}^{3}M^{-1-\delta_1}\sum_{ \substack{ N_{1},N_{2},N_{3}\le M^{\frac{2}{n}} \\ L_2\lesssim \max(N_1,N_2)} }\max(N_1,N_2)^2\min(N_1,N_2)^{4}
    \prod_{j=1}^3N_j^{-k}\\
    & \les\ve_{1}^{3}M^{-(1+\de)}N_0^{-k}.
   \end{align*}

\subsubsection{Estimates for Case 2.}
Let $\mathbf{\Xi}_2=(\theta,\theta_1,-\theta_1,\theta_1)$ be fixed. In this sign configuration, a key observation is that
\begin{align*}
	q_{\mathbf{\Xi}_2}(\xi,\eta,\sigma) = \theta \bra{\xi} - \theo \big( \bra{\xi+\eta} + \bra{\xi+\eta+\sigma} + \bra{\xi+\sigma} \big),
\end{align*}
which does not exhibit the time resonance. Indeed, as already observed in \eqref{time resonance}, we have $|q_{\mathbf{\Xi}_2}(\xi,\eta,\sigma) | \gtrsim N_{\max}^{\,-1}$. Therefore, to apply the normal form approach, we consider the following time integral
\begin{align*}
\int_{t_1}^{t_2} e^{-i\mathcal B_{\theta}(s,\xi)} \mathcal N_{\mathbf{\Xi}_2;\widetilde{\mathbf{N}}\mathbf{L}}(s,\xi) ds.
\end{align*}
Integrating by parts in the time variable, we obtain 
\begin{align*}
\int_{t_1}^{t_2} e^{-i\mathcal B_{\theta}(s,\xi)} \mathcal N_{\mathbf{\Xi}_2;\widetilde{\mathbf{N}}\mathbf{L}}(s,\xi)  ds= e^{-i\mathcal B_{\theta}(s,\xi)}\wt{\mathcal N}_{\mathbf{\Xi}_2;\widetilde{\mathbf{N}}\mathbf{L}}(s,\xi)\Big|_{s=t_1}^{s=t_2} - \sum_{j=1}^2\int_{t_1}^{t_2}
e^{-i\mathcal B_{\theta}(s,\xi)}
\mathcal N_{\mathbf{\Xi}_2;\widetilde{\mathbf{N}}\mathbf{L}}^{j}(s,\xi)
ds,
\end{align*}
where 
\begin{align*}
\wt{\mathcal N}_{\mathbf{\Xi}_2;\widetilde{\mathbf{N}}\mathbf{L}}(s,\xi)=&-i\iint_{\R^{2}\times\R^2}  \frac{ e^{isq_{\mathbf{\Xi}_2}(\xi,\eta,\sigma)} }{|\eta| q_{\mathbf{\Xi}_2}(\xi,\eta,\sigma)} \Pi_{\theta}(\xi)\Pi_{\theta_1}(\xi+\eta)\chi_{L_1}(\eta)\chi_{L_2}(\sigma)   \\
&\quad\times \wh{\Psi_{\theo;N_1}}(s,\xi+\eta)\bra{ \wh{\Psi_{-\theo;N_2}}(s,\xi+\eta+\sigma), \wh{\Psi_{\theo;N_3}}(s,\xi+\sigma)} d\eta d\sigma,\\
\mathcal N_{\mathbf{\Xi}_2;\widetilde{\mathbf{N}}\mathbf{L}}^{1}(s,\xi)=& -\p_s \mathcal B_\theta (s,\xi)    \iint_{\R^{2}\times\R^2}  \frac{e^{isq_{\mathbf{\Xi}_2}(\xi,\eta,\sigma)} }{|\eta| q_{\mathbf{\Xi}_2}(\xi,\eta,\sigma)} \Pi_{\theta}(\xi)\Pi_{\theta_1}(\xi+\eta)\chi_{L_1}(\eta)\chi_{L_2}(\sigma) \\
&\quad\times \wh{\Psi_{\theo;N_1}}(s,\xi+\eta)\bra{ \wh{\Psi_{-\theo;N_2}}(s,\xi+\eta+\sigma), \wh{\Psi_{\theo;N_3}}(s,\xi+\sigma)} d\eta d\sigma, \\
	\mathcal N_{\mathbf{\Xi}_2;\widetilde{\mathbf{N}}\mathbf{L}}^{2}(s,\xi) = & -i \iint_{\R^{2}\times\R^2}   \frac{ e^{isq_{\mathbf{\Xi}_2}(\xi,\eta,\sigma)}}{|\eta|q_{\mathbf{\Xi}_2}(\xi,\eta,\sigma)} \Pi_{\theta}(\xi)\Pi_{\theta_1}(\xi+\eta)\chi_{L_1}(\eta)\chi_{L_2}(\sigma) \\
	&\quad\times \p_s \left[ \wh{\Psi_{\theo;N_1}}(s,\xi+\eta)  \bra{ \wh{\Psi_{-\theo;N_2}}(s,\xi+\eta+\sigma), \wh{\Psi_{\theo;N_3}}(s,\xi+\sigma)} \right] d\eta d\sigma .
\end{align*}

One verifies that 
\begin{align}\begin{aligned}\label{k2 multiplier}
&\sup_{|\xi|\sim N_0}\left\| 
\frac{\rho_{\widetilde{\mathbf{N}}}(\xi,\eta,\sigma)\chi_{L_1}(\eta)\chi_{L_2}(\sigma) }{|\eta| q_{\mathbf{\Xi}_2}(\xi,\eta,\sigma)} \Pi_{\theta}(\xi)\Pi_{\theta_1}(\xi+\eta) 
\right\|_{\textup{CM}_{\eta,\sigma}[(\R^2)^2]} \\
&:= \sup_{|\xi|\sim N_0}\normo{\iint_{\R^{2}\times \R^2} e^{i y \cdot \eta} e^{i z \cdot \sigma}\frac{\rho_{\widetilde{\mathbf{N}}}(\xi,\eta,\sigma)\chi_{L_1}(\eta)\chi_{L_2}(\sigma) }{|\eta| q_{\mathbf{\Xi}_2}(\xi,\eta,\sigma)} \Pi_{\theta}(\xi)\Pi_{\theta_1}(\xi+\eta)  d\eta d \sigma}_{L_{y,z}^1(\R^2\times\R^2)} \\ 
&\lesssim L_1^{-1} N_{\max}.
\end{aligned}\end{align}
Then, applying the operator inequality \eqref{eq:coif-1-2} with this bound and subsequently invoking \eqref{Frequency localized inequalities:Inhomogeneous}, we obtain for $s\in[M/4,2M]$,
\begin{align*}
&\sum_{ \widetilde{\mathbf{N}}\mathbf{L} \in  \boldsymbol\Gamma(M,N_0)}|\wt{\mathcal N}_{\mathbf{\Xi}_2;\widetilde{\mathbf{N}}\mathbf{L}}(s,\xi)|\\
 &\les \sum_{ \widetilde{\mathbf{N}}\mathbf{L} \in  \boldsymbol\Gamma(M,N_0)} L_1^{-1} N_{\max}  \| \psi_{\theo;N_1}(s)\|_{L^2(\R^2)}\| \psi_{-\theo;N_2}(s)\|_{L^2(\R^2)} \|\psi_{\theo;N_3}(s)\|_{L^\infty(\R^2)}\\
&\les \ve_1^3 M^{-\delta_1 +2\delta_0 } \sum_{ \substack{ N_{1},N_{2},N_{3}\le M^{\frac{2}{n}} \\ L_2\lesssim \max(N_1,N_2)} }  N_{\max} N_1^{-n}N_2^{-n}N_3^{-k}\les \ve_1^3 M^{-\de}N_0^{-k}.
\end{align*}

Next, we note that the phase modification satisfies the decay bound 
\begin{align}\label{eq:decay-correction}
\sup_{\xi\in\R^2}|\p_t \mathcal B_\theta(t,\xi)| \les \bra{t}^{-1}\ve_1^2.
\end{align}
This follows from the explicit formula 
\begin{align*}
\p_t \mathcal B_\theta(t,\xi) = \frac{\lam}{(2\pi)^{2}} \bra{t}^{-1}\rho(t^{-\frac{2}{n}}\xi)\sum_{\theta' \in \{+,-\}}\int_{\mathbb{R}^{2}}\left|\theta\frac{\xi}{\langle\xi\rangle}-\theta' \frac{\sigma}{\langle\sigma\rangle}\right|^{-1}\abs{\wh{\psi_{\theta'}}(t,\sigma)}^{2}d\sigma,
 \end{align*}
together with the lower bound in \eqref{ineq:Phase null structure}. Indeed, we estimate
\begin{align*}
\sup_{\xi\in\R^2}|\p_t \mathcal B_\theta(t,\xi)| 
&\les  \bra{t}^{-1}\sum_{\theta' \in \{+,-\}}\sup_{\xi\in\R^2}\int_{\R^2}  \left|\theta\frac{\xi}{\langle\xi\rangle}-\theta' \frac{\sigma}{\langle\sigma\rangle}\right|^{-1}\abs{\wh{\psi_{\theta'}}(t,\sigma)}^{2} d\sigma  \\ 
&\les \langle t\rangle^{-1}\sum_{\theta' \in \{+,-\}}\big\|\langle \xi\rangle^{k}\widehat{\psi_{\theta'}}(t,\xi)\big\|_{L_\xi^\infty(\R^2)}^2
\les \bra{t}^{-1}\ve_1^2.
\end{align*}
Similarly to the estimates for $\wt{\mathcal N}_{\mathbf{\Xi}_2;\widetilde{\mathbf{N}}\mathbf{L}}$, we apply \eqref{eq:coif-1-2}, making use of the multiplier bound \eqref{k2 multiplier} together with the decay estimate for the phase correction \eqref{eq:decay-correction}, to obtain 
\begin{align*}
&\sum_{ \widetilde{\mathbf{N}}\mathbf{L} \in  \boldsymbol\Gamma(M,N_0)}\left| \mathcal N_{\mathbf{\Xi}_2;\widetilde{\mathbf{N}}\mathbf{L}}^1(s,\xi)\right|  \\ 
&\les |\p_s \mathcal B_\theta (s,\xi)| \sum_{ \widetilde{\mathbf{N}}\mathbf{L} \in  \boldsymbol\Gamma(M,N_0)}L_1^{-1}  N_{\max}   \| \psi_{\theo;N_1}(s)\|_{L^2(\R^2)}\| \psi_{-\theo;N_2}(s)\|_{L^2(\R^2)} \|\psi_{\theo;N_3}(s)\|_{L^\infty(\R^2)} \\ 
&\les \ve_1^5 M^{-1-\delta_1+2\delta_0} \sum_{ \substack{ N_{1},N_{2},N_{3}\le M^{\frac{2}{n}} \\ L_2\lesssim \max(N_1,N_2)} }   N_{\max} N_1^{-n}N_2^{-n}N_3^{-k} \les \ve_1^3 M^{-1-\de}N_0^{-k}.
\end{align*}

Lastly, recalling the frequency-localized estimate for the time derivative of the profile \eqref{eq:esti-timederi-f}, an argument entirely analogous to the above, together with the multiplier bound \eqref{k2 multiplier}, yields
\begin{align*}
&\sum_{ \widetilde{\mathbf{N}}\mathbf{L} \in  \boldsymbol\Gamma(M,N_0)}\left| \mathcal N_{\mathbf{\Xi}_2;\widetilde{\mathbf{N}}\mathbf{L}}^2 (s,\xi)\right| \\
&\les  \sum_{ \widetilde{\mathbf{N}}\mathbf{L} \in  \boldsymbol\Gamma(M,N_0)}L_1^{-1}N_{\max}
\Big( \| \partial_s \Psi_{\theo;N_1}(s)\|_{L^2(\R^2)}\|\psi_{-\theo;N_2}(s)\|_{L^2(\R^2)} \|\psi_{\theo;N_3}(s)\|_{L^\infty(\R^2)}  \\ 
&\qquad\qquad\qquad\qquad\qquad + \|  \psi_{\theo;N_1}(s)\|_{L^2(\R^2)}\| \partial_s \Psi_{-\theo;N_2}(s)\|_{L^2(\R^2)} \|\psi_{\theo;N_3}(s)\|_{L^\infty(\R^2)} \\
&\qquad\qquad\qquad\qquad\qquad + \| \psi_{\theo;N_1}(s)\|_{L^2(\R^2)}\| \psi_{-\theo;N_2}(s)\|_{L^\infty(\R^2)} \|\partial_s \Psi_{\theo;N_3}(s)\|_{L^2(\R^2)} \Big) \\ 
&\les \ve_1^3 M^{-1-\delta_1+2\delta_0}\sum_{ \substack{ N_{1},N_{2},N_{3}\le M^{\frac{2}{n}} \\ L_2\lesssim \max(N_1,N_2)} } N_{\max} N_1^{-k}N_2^{-k}N_3^{-k}
\les \ve_1^3 M^{-1-\de}N_0^{-k}.
\end{align*}

\subsubsection{Estimates for Case 3.}
Let $\mathbf{\Xi}_3=(\theta,\theta_1,-\theta_1,-\theta_1)$ be fixed. Similar to Case~1, we employ space resonance analysis; however, in this case the non-resonant structure arises in the $\sigma$ variable instead of $\eta$. The key identity corresponding to \eqref{lower bound of nabla q} is
\begin{align*}
		|\nabla_{\sigma}q_{\mathbf{\Xi}_3}(\xi,\eta,\sigma)| &=
\left|\frac{\xi+\eta+\sigma}{\langle \xi+\eta+\sigma\rangle} - \frac{\xi+\sigma}{\langle \xi+\sigma\rangle}  \right|   \\ 
&\gtrsim \frac{|\eta|}{\max(\langle \xi+\eta+\sigma\rangle, \langle\xi+\sigma\rangle) \min(\langle \xi+\eta+\sigma\rangle, \langle\xi+\sigma\rangle)^2}.
\end{align*}
Then, by using 
\begin{align*}
	e^{isq_{\mathbf{\Xi}_3}} =  -i \frac1s \frac{\nabla_\sigma q_{\mathbf{\Xi}_3}\cdot \nabla_\sigma e^{isq_{\mathbf{\Xi}_3}}}{|\nabla_\sigma q_{\mathbf{\Xi}_3}|^2},
\end{align*}
we perform integration by parts in $\sigma$ twice 
\begin{align*}  
|\mathcal{N}_{\mathbf{\Xi}_3;\widetilde{\mathbf{N}}\mathbf{L} }(s,\xi)|  &\le |\mathcal{N}_{\mathbf{\Xi}_3;\widetilde{\mathbf{N}}\mathbf{L} }^1(s,\xi)|+|\mathcal{N}_{\mathbf{\Xi}_3;\widetilde{\mathbf{N}}\mathbf{L} }^{\,2}(s,\xi)|+|\mathcal{N}_{\mathbf{\Xi}_3;\widetilde{\mathbf{N}}\mathbf{L} }^{\,3}(s,\xi)|,\\
  \mathcal{N}_{\mathbf{\Xi}_3;\widetilde{\mathbf{N}}\mathbf{L} }^1(s,\xi)  &=\frac{1}{s^{2}}\int\!\!\!\!\int_{\R^2\times\R^2} e^{isq_{\mathbf{\Xi}_3}(\xi,\eta,\sigma)} \frac{\nabla_{\sigma}q_{\mathbf{\Xi}_3}}{|\nabla_{\sigma}q_{\mathbf{\Xi}_3}|^{2}}\widetilde{\mathcal{Q}}_{\mathbf{\Xi}_3;\mathbf{L}}(\xi,\eta,\sigma)
 \\ &\qquad \times \nabla_{\sigma}^{2}\Big( \wh{\Psi_{\theo;N_1}}(s,\xi+\eta)\bra{\wh{\Psi_{-\theo;N_2}}(s,\xi+\eta+\sigma),\wh{\Psi_{-\theo;N_3}}(s,\xi+\sigma)} \Big)d\eta d\sigma, \\
     \mathcal{N}_{\mathbf{\Xi}_3;\widetilde{\mathbf{N}}\mathbf{L}}^{\,2}(s,\xi) & =\frac{1}{s^{2}}\int\!\!\!\!\int_{\R^2\times\R^2} e^{isq_{\mathbf{\Xi}_3}(\xi,\eta,\sigma)} \nabla_{\sigma} \Big( \frac{\nabla_{\sigma}q_{\mathbf{\Xi}_3}}{|\nabla_{\sigma}q_{\mathbf{\Xi}_3}|^{2}}\widetilde{\mathcal{Q}}_{\mathbf{\Xi}_{3};\mathbf{L}} \Big) (\xi,\eta,\sigma)
 \\ &\qquad \times \nabla_{\sigma}\Big( \wh{\Psi_{\theo;N_1}}(s,\xi+\eta)\bra{\wh{\Psi_{-\theo;N_2}}(s,\xi+\eta+\sigma),\wh{\Psi_{-\theo;N_3}}(s,\xi+\sigma)} \Big)d\eta d\sigma, \\
    \mathcal{N}_{\mathbf{\Xi}_3;\widetilde{\mathbf{N}}\mathbf{L}}^{\,3}(s,\xi) &=\frac{1}{s^{2}}\int\!\!\!\!\int_{\R^2\times\R^2} e^{isq_{\mathbf{\Xi}_3}\freq}\nabla_{\sigma}\Big(\frac{\nabla_{\sigma}q_{\mathbf{\Xi}_3}}{|\nabla_{\sigma}q_{\mathbf{\Xi}_3}|^{2}}\nabla_{\sigma} \widetilde{\mathcal{Q}}_{\mathbf{\Xi}_3;\mathbf{L}}\Big)
 \\ &\qquad \times  \wh{\Psi_{\theo;N_1}}(s,\xi+\eta)\bra{\wh{\Psi_{-\theo;N_2}}(s,\xi+\eta+\sigma),\wh{\Psi_{-\theo;N_3}}(s,\xi+\sigma)} d\eta d\sigma, \\
\end{align*}
    where 
\begin{align*}
\widetilde{\mathcal{Q}}_{\mathbf{\Xi}_3;\mathbf{L}}(\xi,\eta,\sigma) & :=\frac{\nabla_{\sigma}q_{\mathbf{\Xi}_3}}{|\nabla_{\sigma}q_{\mathbf{\Xi}_3}|^{2}}(\xi,\eta,\sigma)|\eta|^{-1}\Pi_{\theta}(\xi)\Pi_{\theta_1}(\xi+\eta)\chi_{L_1}(\eta)\chi_{L_2}(\sigma).
\end{align*}
One can verify that the symbols satisfy the following bounds:
\begin{align*}
\sup_{|\xi|\sim N_0}\left\| \frac{\nabla_{\sigma}q_{\mathbf{\Xi}_3}}{|\nabla_{\sigma}q_{\mathbf{\Xi}_3}|^{2}}\widetilde{\mathcal{Q}}_{\mathbf{\Xi}_3;\mathbf{L}}\rho_{\widetilde{\mathbf{N}}}(\xi,\eta,\sigma)\right\|_{L_{\eta,\sigma}^\infty((\R^2)^2)} &\lesssim L_1^{-3}\max(N_2,N_3)^2\min(N_2,N_3)^{4}, \\ 
\sup_{|\xi|\sim N_0}\left\| \nabla_{\sigma} \Big( \frac{\nabla_{\sigma}q_{\mathbf{\Xi}_3}}{|\nabla_{\sigma}q_{\mathbf{\Xi}_3}|^{2}} \widetilde{\mathcal{Q}}_{\mathbf{\Xi}_3;\mathbf{L}} \Big) \rho_{\widetilde{\mathbf{N}}}(\xi,\eta,\sigma) \right\|_{L_{\eta,\sigma}^\infty((\R^2)^2)} &\lesssim L_1^{-3}L_2^{-1}\max(N_2,N_3)^2\min(N_2,N_3)^{4}, \\ 
\sup_{|\xi|\sim N_0}\left\| \nabla_{\sigma}\Big(\frac{\nabla_{\sigma}q_{\mathbf{\Xi}_3}}{|\nabla_{\sigma}q_{\mathbf{\Xi}_3}|^{2}}\nabla_{\sigma} \widetilde{\mathcal{Q}}_{\mathbf{\Xi}_3;\mathbf{L}}\Big)\rho_{\widetilde{\mathbf{N}}}(\xi,\eta,\sigma) \right\|_{L_{\eta,\sigma}^\infty((\R^2)^2)} &\lesssim L_1^{-3}L_2^{-2}\max(N_2,N_3)^2\min(N_2,N_3)^{4}.
\end{align*}
Using the first multiplier bound together with the frequency-localized estimates
\eqref{Frequency localized inequalities:Inhomogeneous},
\eqref{frequency localized with x weight},
\eqref{frequency localized with x^2 weight},
and the relation
\(L\sim M^{-1+\delta_1}\)
from \eqref{eq:lzero}, we estimate
\begin{align*}
 &\sum_{ \widetilde{\mathbf{N}}\mathbf{L} \in  \boldsymbol\Gamma(M,N_0)}|\mathcal{N}_{\mathbf{\Xi}_3;\widetilde{\mathbf{N}}\mathbf{L}}^1(s,\xi)|\\
& \les M^{-2}\sum_{ \widetilde{\mathbf{N}}\mathbf{L} \in  \boldsymbol\Gamma(M,N_0)}L_1^{-3}\max(N_2,N_3)^2\min(N_2,N_3)^{4}\|\chi_{L_1}\|_{L^{1}(\R^2)}\normo{\widehat{\Psi_{\theta_1;N_{1}}}(s)}_{L^{\infty}(\R^2)} \\ 
&\quad\times\Big( 
\normo{x^{2}\Psi_{-\theta_1;N_2}(s)}_{L_x^{2}(\R^2)}\normo{\Psi_{-\theta_1;N_{3}}(s)}_{L^{2}(\R^2)} + \normo{x\Psi_{-\theta_1;N_2}(s)}_{L_x^{2}(\R^2)}\normo{x\Psi_{-\theta_1;N_3}(s)}_{L_x^{2}(\R^2)} \\
&\hspace{7cm}+ \normo{\Psi_{-\theta_1;N_2}(s)}_{L^{2}(\R^2)}\normo{x^{2}\Psi_{-\theta_1;N_3}(s)}_{L_x^{2}(\R^2)} \Big)\\
& \les\ve_{1}^{3}M^{-2+2\delta_0}\sum_{ \substack{ N_{1},N_{2},N_{3}\le M^{\frac{2}{n}} \\   L\le L_1 \le \max(N_2,N_3)}}L_1^{-1}\max(N_2,N_3)^2\min(N_2,N_3)^{4}N_{1}^{-k} \\
& \les\ve_{1}^{3}M^{-1-\delta_1+2\delta_0}\sum_{  N_{1},N_{2},N_{3}\le M^{\frac{2}{n}} }\max(N_2,N_3)^2\min(N_2,N_3)^{4}N_{1}^{-k}\\
& \les\ve_{1}^{3}M^{-1-\de}N_0^{-k}.
\end{align*}
Proceeding similarly with the second multiplier bound, and recalling that the summation is restricted by $L_1L_2^{\,2}\ge M^{-1-2\delta}$, we estimate
\begin{align*}
 & \sum_{ \widetilde{\mathbf{N}}\mathbf{L} \in  \boldsymbol\Gamma(M,N_0)}|\mathcal{N}_{\mathbf{\Xi}_3;\widetilde{\mathbf{N}}\mathbf{L}}^2(s,\xi)|\\
& \les M^{-2}\sum_{ \widetilde{\mathbf{N}}\mathbf{L} \in  \boldsymbol\Gamma(M,N_0)}L_1^{-3}L_2^{-1}\max(N_2,N_3)^2\min(N_2,N_3)^{4}\|\chi_{L_1}\|_{L^{1}(\R^2)}\normo{\widehat{\Psi_{\theta_1;N_{1}}}(s)}_{L^{\infty}(\R^2)} \\ 
&\qquad\times\left( 
\normo{x\Psi_{-\theta_1;N_2}(s)}_{L_x^{2}(\R^2)}\normo{\Psi_{-\theta_1;N_{3}}(s)}_{L^{2}(\R^2)} + \normo{\Psi_{-\theta_1;N_2}(s)}_{L^{2}(\R^2)}\normo{x\Psi_{-\theta_1;N_3}(s)}_{L_x^{2}(\R^2)} \right)\\
& \les\ve_{1}^{3}M^{-\frac32+\delta}\sum_{ \substack{ N_{1},N_{2},N_{3}\le M^{\frac{2}{n}} \\   L\le L_1 \le \max(N_2,N_3)}}L_1^{-\frac12}\max(N_2,N_3)^2\min(N_2,N_3)^{4}N_{1}^{-k} \\
& \les\ve_{1}^{3}M^{-1+\delta-\frac12\delta_1}\sum_{  N_{1},N_{2},N_{3}\le M^{\frac{2}{n}}}\max(N_2,N_3)^2\min(N_2,N_3)^{4}N_{1}^{-k}\\
& \les\ve_{1}^{3}M^{-1-\de}N_0^{-k}.
\end{align*}
Finally, applying the third multiplier bound in the same manner yields
\begin{align*}
 &  \sum_{ \widetilde{\mathbf{N}}\mathbf{L} \in  \boldsymbol\Gamma(M,N_0)}|\mathcal{N}_{\mathbf{\Xi}_3;\widetilde{\mathbf{N}}\mathbf{L}}^3(s,\xi)|\\
    & \les M^{-2}\sum_{ \widetilde{\mathbf{N}}\mathbf{L} \in  \boldsymbol\Gamma(M,N_0)}L_1^{-3}L_2^{-2}\max(N_2,N_3)^2\min(N_2,N_3)^{4}\|\chi_{L_1}\|_{L^{1}(\R^2)}\|\chi_{L_2}\|_{L^{1}(\R^2)} \\ 
&\qquad\qquad\times
    \normo{\wh{\Psi_{\theta_1;N_{1}}}(s)}_{L^{\infty}(\R^2)}\normo{\wh{\Psi_{-\theta_1;N_{2}}}(s)}_{L^{\infty}(\R^2)}\normo{\wh{\Psi_{-\theta_1;N_{3}}}(s)}_{L^{\infty}(\R^2)}\\
    & \les\ve_{1}^{3}M^{-2}\sum_{ \substack{ N_{1},N_{2},N_{3}\le M^{\frac{2}{n}} \\   L\le L_1 \le \max(N_2,N_3)}}L_1^{-1}\max(N_2,N_3)^2\min(N_2,N_3)^{4}N_{1}^{-k}N_{2}^{-k}N_{3}^{-k}\\
    & \les\ve_{1}^{3}M^{-1-\delta_1}\sum_{  N_{1},N_{2},N_{3}\le M^{\frac{2}{n}} }\max(N_2,N_3)^2\min(N_2,N_3)^{4}N_{1}^{-k}N_{2}^{-k}N_{3}^{-k}\\
    & \les\ve_{1}^{3}M^{-(1+\de)}N_0^{-k},
   \end{align*}
which completes the analysis of Case~3.

\bibliographystyle{plain}
\bibliography{ReferencesLY}

\medskip

\end{document}